\documentclass[12pt]{amsart}
\usepackage{amsmath}
\usepackage{amsthm}
\usepackage{amscd}
\usepackage{amsfonts}
\usepackage{amssymb}
\usepackage{latexsym}
\usepackage{vmargin} 
\usepackage[latin1]{inputenc}
\usepackage[T1]{fontenc}
\newtheorem{thm}{Theorem}[section]
\newtheorem{lem}[thm]{Lemma}

\newtheorem{cor}[thm]{Corollary}
\newtheorem{prop}[thm]{Proposition}

\theoremstyle{definition}
\newtheorem{example}[thm]{Example}
\newtheorem{defn}[thm]{Definition}
\newtheorem{remark}[thm]{Remark}

\newcommand{\mycx}{N}

\newcommand{\mbw}{\mathbf{w}}

\newcommand{\mbx}{\mathbf{x}}
\newcommand{\mby}{\mathbf{y}}
\newcommand{\mbu}{\mathbf{u}}
\newcommand{\mbv}{\mathbf{v}}
\newcommand{\mbr}{\mathbf{r}}
\newcommand{\mbt}{\mathbf{t}}
\newcommand{\mba}{\mathbf{a}}
\newcommand{\mbb}{\mathbf{b}}
\newcommand{\mbc}{\mathbf{c}}
\newcommand{\mbk}{\mathbf{k}}

\newcommand{\kr}{{K}}

\newcommand{\Na}{\mathcal{N}}

\newcommand{\one}{{\mathbf{1}}}
\newcommand{\link}{{\mathrm{lk}}}
\newcommand{\cost}{{\mathrm{cost}}}

\newcommand{\mm}{\mathfrak{m}}

\newcommand{\Z}{{\mathbb{Z}}}

\newcommand{\N}{{\mathbb{N}}}

\newcommand{\op}{{\operatorname{op}}}

\newcommand{\Hom}{\operatorname{Hom}}
\newcommand{\Ext}{\operatorname{Ext}}
\newcommand{\HOM}{\underline{\operatorname{Hom}}}
\newcommand{\EXT}{\underline{\operatorname{Ext}}}

\newcommand{\colim}[1]{\mathop{\underset{#1}
            {{\text{\rm colim}}}}}

\newcommand{\mto}[1]{\overset{#1}{\longrightarrow}}
\newcommand{\te}{\otimes}

\begin{document}

\title [The Auslander-Reiten translate on monomial quotient rings]
{The Auslander-Reiten translate on monomial rings}
\author { Morten Brun \and Gunnar Fl{\o}ystad}
\address{ Matematisk Institutt\\
          Johs. Brunsgt. 12 \\
          5008 Bergen \\
          Norway}   
        
\email{ Morten.Brun@mi.uib.no \and gunnar@mi.uib.no}

\begin{abstract}
For $\mbt$ in $\N^n$, E.Miller has defined a category
of
positively $\mbt$-determined modules over the polynomial ring $S$ in $n$ variables. 
We consider
the Auslander-Reiten translate, $\Na_\mbt$, on the (derived) category
of such modules. A monomial ideal $I$ is $\mbt$-determined if 
each generator $\mbx^\mba$ has $\mba \leq \mbt$. We compute the multigraded
cohomology and betti spaces of $\Na_\mbt^k(S/I)$ for every iterate
$k$ and also the $S$-module structure of these cohomology modules.
This comprehensively generalizes results of Hochster and
Gr\"abe on local cohomology of Stanley-Reisner rings.
\end{abstract}

\pagestyle{myheadings}

\maketitle
\vskip -1cm
\tableofcontents

\section{Introduction}

Let $S = K[x_1, \ldots, x_n]$ be the polynomial ring over a field $K$. 
On the category
of graded projective $S$-modules there is the standard duality
$D = \Hom_S(-,S)$. In the case of $\N^n$-graded $S$-modules, on the 
subcategory of square free modules there is also a more combinatorially defined
Alexander duality introduced by E.Miller, \cite{Miller2000}, 
and T.R\"omer, \cite{Romer2001}.
More generally Miller for any $\mbt$ in $\N^n$ defines a subcategory
of $\N^n$-graded modules, the category of  positively 
$\mbt$-determined modules.
They have the property that they are determined (in a specified way) by their 
$S$-module structure for degrees in the interval $[\mathbf 0, \mbt]$. On this
category Miller also defines an Alexander duality $A_\mbt$. Considering 
standard duality in this setting it will be convenient to consider a 
twisted version. Replace $S$ with an injective resolution $I$, and
consider the functor $D_\mbt = \Hom_S(-, I(-\mbt))$, which becomes a duality 
(up to quasi-isomorphism) on the category of chain complexes of 
 positively $\mbt$-determined $S$-modules.

\medskip
    The main object of this paper is to study the composition $A_{\mbt} \circ
D_{\mbt}$ of the two functors and its iterates. This composition appears in
various guises in the literature. Firstly, if $M$ is any  positively $\mbt$-determined 
module,
the cohomology of $A_{\mbt} \circ D_{\mbt}(M)$ is simply the local cohomology 
modules of $M$, slightly rearranged. Secondly, the category of  
positively $\mbt$-determined
modules is equivalent to the category of modules over the incidence algebra
of the poset $[\mathbf 0, \mbt]$, which is an artin algebra. With this
identification $A_{\mbt} \circ D_{\mbt}$ is the Auslander-Reiten translate
on the derived category of modules over this incidence algebra, 
see \cite[p.37]{Ha}. 
In this derived setting $A_{\mbt} \circ D_{\mbt}$ may also be identified as 
the Nakayama
functor, \cite[p.126]{ARS}, and this is the terminology we
shall use for this composition. Thirdly, in the case when $\mbt$ is
equal to $\one = (1,1, \ldots, 1)$, the case of square free modules, the 
second author, in \cite {Floystad2004}, 
considers $A_{\mbt} \circ D_{\mbt} (K[\Delta])$
where $K[\Delta]$ is the Stanley-Reisner ring of a simplicial complex $\Delta$,
and studies the cohomology groups of this complex, termed the {\it enriched
cohomology modules} of the simplicial complex. Fourthly, also in the square
free case, K.Yanagawa, \cite{Yanagawa2004}, shows that the third iteration
$(A_\mbt \circ D_\mbt)^3$ is isomorphic to the translation functor $T^{-2n}$
on the derived category of square free modules.

\medskip

In this paper we study the iterates of the Nakayama functor $\Na_\mbt = A_\mbt
\circ D_\mbt$ applied to a monomial quotient ring $S/I$ where $I$ is a 
positively $\mbt$-determined ideal, meaning that every minimal generator $x^\mba$ of 
$I$ has $\mba \leq \mbt$. We compute the multigraded pieces of the cohomology
groups of every iterate $\Na^k_{\mbt} (S/I)$, Theorem \ref{thecalc1}. They turn
out to be given by the reduced cohomology groups of various simplicial complexes
derived from the monomial ideal $I$. In the square free case $\mbt = \one$ and
$k = 1$ this specializes to Hochster's classical computation, 
\cite{Hochster1977}, of the local 
cohomology modules of Stanley-Reisner rings, and in the case $k=1$ and 
general $\mbt$ we recover the more recent computations of Takayama, 
\cite{Takayama2004}, of 
the local cohomology modules of a monomial quotient ring.

%It is a standard fact that a bounded complex of finitely generated graded 
%$S$-modules is quasi-isomorphic to a minimal complex of free finitely generated
%$\N^n$-graded $S$-modules. Hence 
We consider the multigraded Betti spaces
of the complex $\Na^k_{\mbt}(S/I)$, and compute all these spaces for every $k$, 
Theorem \ref{bettithm}. Again they are given by the reduced cohomology 
of various simplicial complexes derived from the monomial ideal $I$. In the
case $k=0$ we recover the well-known computations of the multigraded Betti 
spaces of $S/I$, see \cite[Cor.5.12]{MiSt}. 
A striking feature is how similar in form the 
statements of Theorems \ref{thecalc1} and \ref{bettithm} are, which compute
respectively the multigraded cohomology and Betti spaces.

Not only are we able to compute the multigraded parts of the cohomology modules
of $\Na^k_{\mbt}(S/I)$ for every $k$. We also compute their $S$-module structure,
Theorem \ref{nakcohS}, 
extending results of Gr\"abe, \cite{Graebe1984}. To give their $S$-module 
structure
it is sufficient to show how multiplication with $x_i$ acts on the various
multigraded parts of the cohomology modules, and we show how these actions 
correspond
to natural maps on reduced cohomology groups of simplicial complexes.

Actually all these computations are done in a more general setting. 
Since the interval $[\mathbf 0, \mbt]$ is a product $\Pi_{i=1}^n [0, t_i]$,
it turns out that for every $\mbk$ in $\N^n$ we may define the "multi-iterated"
Nakayama functor $\Na^\mbk_\mbt$. All our computations above are done, 
with virtually
no extra effort, for the complexes $\Na^\mbk_\mbt(S/I)$ instead of for the 
simple iterates $\Na^k_{\mbt}(S/I)$.
Letting $\mbt + {\mathbf 2}$ be $(t_1 +2, t_2 + 2, \ldots, t_n+2)$
we also show that the multi-iteration $\Na^{\mbt + {\mathbf 2}}_\mbt$ is 
quasi-isomorphic
to the translation functor $T^{-2n}$ on complexes of  positively $\mbt$-determined
$S$-modules, thereby generalizing Yanagawa's result \cite{Yanagawa2004}.

\medskip 
The organization of the paper is as follows. In Section 2 we recall 
the notion of Miller of  positively $\mbt$-determined modules. We define the 
Nakayama functor $\Na_\mbt$ on chain complexes of  positively $\mbt$-determined modules, 
not as $A_{\mbt} \circ D_{\mbt}$ as explained in 
the introduction, but rather
in a more direct way adapting slightly a defining property of local cohomology. 
More generally
we define the multi-iterated Nakayama functor $\Na_\mbt^\mbk$ for any 
$\mbk$ in $\N^n$. In Section 3 we show that the Nakayama functor $\Na_\mbt$
is the composition $A_\mbt \circ D_\mbt$ of Alexander duality and (twisted)
standard duality. We also show how the category of  positively $\mbt$-determined 
modules is equivalent to the category of modules over the incidence algebra
of the poset $[\mathbf 0, \mbt]$, and that the composition $A_\mbt \circ D_\mbt$
corresponds to the Auslander-Reiten translate on the derived category 
of this module category. Section 4 gives our main results. This is the 
computation of the multigraded cohomology spaces and Betti spaces of 
$\Na^{\mbk}_{\mbt}(S/I)$ for a  positively $\mbt$-determined ideal, and also 
the $S$-module structure of its cohomology groups.

In Section 5 we give sufficient conditions for vanishing of various cohomology
groups of $\Na_{\mbt}^{\mbk}(S/I)$. In particular we consider in detail the case 
when $S$ is a polynomial ring in two variables, and investigate when 
$\Na_{\mbt}^{\mbk}(S/I)$ has only one non-vanishing cohomology module.
In Section 6 we consider the Nakayama functor applied to quotient
rings of $S$ by ideals generated by variable powers.
The essential thing is doing the case of one variable $n=1$, and here
we develop lemmata which will be pivotal in the final proofs. 
Section 7 considers posets. Two different
constructions one may derive from a poset $P$ are, algebraically, 
the $KP$-modules over the poset, and, topologically, the order complexes of 
various subposets of $P$. 
We give some results relating various Ext-groups of $KP$-modules, 
to the reduced cohomology groups of various associated
order complexes. 

A monomial ideal $I$ will be positively $\mbt$-determined for many $\mbt$.  
In Section 8 we show how
the complexes $\Na_\mbt^\mbk(S/I)$ for various $\mbt$ (and $\mbk$ depending on
$\mbt$) are equivalent. Therefore various conditions like vanishing of 
cohomology
modules and linearity conditions turn out to be independent of $\mbt$ and thus
intrinsic to the monomial ideal.
Finally Section 9 contains the proofs of the main theorems given in Section 4. 

{\it Acknowledgements.} We thank {\O}.Solberg for pointing out to us 
that the functors considered in this paper correspond to the Nakayama
functor in the theory of Artin algebras, and that this functor in the 
setting of derived categories corresponds to the Auslander-Reiten translate.
 
 Also we thank the referee for thorough reading of the manuscript, pointing 
out mistakes, and for suggestions leading to a clearer presentation.

\section{The Nakayama Functor and local cohomology}
\label{sec:nakfuny}

In this section we introduce the Nakayama functor and relate it
to local cohomology. 
For a given $\mbt$ in $\N^n$, Miller \cite{Miller2000} has introduced
the category of positively $\mbt$-determined $S$-modules, a subcategory
of $\N^n$-graded modules. First we recall this module category. The 
Nakayama functor is then a functor applied to chain complexes of positively
$\mbt$-determined modules.

%This is a functor which is applied to complexes over a module category.
%And first we introduce this module category. It is a subcategory of the 
%$\N^n$-graded modules over the polynomial ring in $n$ variables, introduced
%by E.Miller \cite{Miller2000}: For a given $\mbt$ in $\N^n$, it is the category
%of  $\mbt$-determined modules.

\subsection{Positively $\mbt$-determined modules}
\label{sec:posdet}
In this paper $K$ denotes a fixed arbitrary field and $S$ denotes the
polynomial algebra $S = K[x_1,\dots,x_n]$ for $n \ge 1$. If $n=1$ we
write $S = K[x]$. Given $\mba = (a_1,\dots,a_n)$ in the monoid $\N^n$, 
we shall use the notation
\begin{displaymath}
  x^{\mba} = x_1^{a_1}\cdots x_n^{a_n} \qquad \text{and} \qquad \mm^\mba
  = (x_i^{a_i} \, \colon a_i \ge 1).
\end{displaymath}

Let $\mbt \in \N^n$. We shall begin this section by recalling Miller's
concept of positively $\mbt$-determined ideals in $S$ and how this
concept extends to $\N^n$-graded $S$-modules. We use the partial order
on $\N^n$ where $\mba \leq \mbb$ if each coordinate $a_i \leq b_i$. Similarly
we have a partial order on $\Z^n$.

First of all a monomial ideal $I$ in $S$ is positively $\mbt$-determined
%(as $\Z^m$-graded $S$-module) 
if 
every minimal generator $x^{\mba}$ of $I$
satisfies $\mba \le \mbt$. 
In particular, with the notation $\one = (1,\dots,1)$ the ideal
$I$ is $\one$-determined if and only if it is 
monomial and square free. 
Miller has extended this concept to $\N^n$-graded $S$-modules.
%Let us agree that a $\Z^n$-graded $S$-module
%$M$ is $\N^n$-graded if $M_{\mba} = 0$ unless $\mba \in \N^n$, and
Let $\varepsilon_i \in \N^n$ be the $i$'th unit vector with $1$ in
position $i$ and zero 
in other positions. An
$\N^n$-graded module $M$ is {\em  positively $\mbt$-determined} if the
multiplication  
\[ M_{\mba} \xrightarrow {\cdot x_i} M_{{\mba} + \varepsilon_i} \]
is an isomorphism whenever $a_i \ge t_i$.
%(Miller uses the adjective ``positively'' to distinguish this case from
%certain related cases. We shall only be concerned with this one case,
%so we will drop the adjective.)
Note that
  a monomial ideal $I$ in $S$ is positively $\mbt$-determined if and only if $I$
  considered as a
  $\Z^n$-graded $S$-module is  positively $\mbt$-determined.
This again holds if and only if $S/I$ is positively $\mbt$-determined.

Let us agree that an $\N^n$-graded module is also $\Z^n$-graded by letting
it be zero in all degrees in $\Z^n \backslash \N^n$. Conversely a $\Z^n$-graded
module is called $\N^n$-graded if it is zero in all degrees in $\Z^n 
\backslash \N^n$.

Given an order preserving map $q \colon \Z^n \to \Z^n$
and a $\Z^n$-graded $S$-module $M$. We then get a $\Z^n$-graded
$S$-module $q^*M$ with $(q^*M)_\mba = M_{q(\mba)}$, and with 
multiplication $x^{\mbb} \colon (q^* M)_{\mba} \to (q^*M)_{\mba+ \mbb}$ defined to be
multiplication with $x^{q(\mba+\mbb)-q(\mba)}$ from $M_{q(\mba)}$ to $M_{q(\mba+\mbb)}$.
\begin{example}
  Given a
  $\Z^n$-graded $S$-module $M$ and $\mba \in \N^n$ there is a
  $\Z^n$-graded $S$-module 
  $M(\mba)$ with
  $M(\mba)_\mbr = M_{\mba+\mbr}$. In the above terminology $M(\mba) =
  \tau_\mba^*M$ for 
  the order preserving translation map $\tau_\mba \colon \Z^n_* \to
  \Z^n_*$ with $\tau_\mba(\mbr) = 
  \mba+\mbr$.
\end{example}
A variation of this, actually this is the only situation we shall be concerned with,
is when we have an order preserving map $q \colon \N^n \to \Z^n$. Then 
$q^*M$ is $\N^n$ graded (and hence by our conventions $\Z^n$-graded by considering it
to be zero in all degrees in $\Z^n \backslash \N^n$).

To a $\Z^n$-graded $S$-module and a multidegree $\mbt$ in $\N^n$ we can associate a 
 positively $\mbt$-determined module as follows.

\begin{defn}
Let $p_{\mbt} \colon \N^n \to \Z^n$ be given as $p_{\mbt}(\mba) = (p_{t_1}(a_1),\dots,p_{t_n}(a_n))$
where the coordinates are defined by
\begin{displaymath}
  p_{t_i}(a_i) =
  \begin{cases}
    a_i & \text{if $0 \le a_i \le t_i$}, \\
    t_i & \text{if $a_i \ge t_i$}.
  \end{cases}
\end{displaymath}
The {\it $\mbt$-truncation} of $M$ is then $p_{\mbt}^*M$. 
\end{defn}

Note that the $\mbt$-truncation of $M$ only depends on the module
structure of $M$ for degrees in the interval $[{\mathbf 0}, \mbt]$.

\begin{example}
  Let $n = 1$. Given $r \ge 1$ and $t \ge 0$ we have:
  \begin{displaymath}
    p_t^*(S/x^rS) =
    \begin{cases}
%      S/x^rS & \text{if t= -1} \\
      S & \text{if $t \le r-1$} \\
      S/x^rS & \text{if $r \le t$}.
    \end{cases}
  \end{displaymath}
  On the other hand
  \begin{displaymath}
%    \quad \text{and} \quad
    p_t^*(x^rS) =
    \begin{cases}
%      x^rS & \text{if t= -1} \\
      0 & \text{if $t \le r-1$} \\
      x^rS & \text{if $r \le t$}.
    \end{cases}
  \end{displaymath}
\end{example}
 
The above example show that homologically the functor $M \mapsto
p_{\mbt}^*M$ is somewhat ill-behaved because it takes some modules of rank
zero to modules of rank one and some modules of rank one to 
zero. However in this 
paper we take advantage of this, much like we usually take
advantage of non-exactness to obtain derived functors.

\medskip 

Note that if $f \colon M \to N$ is a homomorphism of $\Z^n$-graded $S$-modules,
and $q$ is an order preserving map from $\N^n$ or $\Z^n$, to $\Z^n$,
then the homomorphisms $f_{q(\mba)} \colon M_{q(\mba)}
\to N_{q(\mba)}$ induce a homomorphism $q^*f \colon
q^*M \to q^*N$ of $\Z^n$-graded $S$-modules. Thus $M
\mapsto q^*M$ is an endofunctor on the category of $\Z^n$-graded
$S$-modules. This endofunctor is exact in the following sense.
\begin{lem}
\label{truncateexact}
  Let $L \xrightarrow g M \xrightarrow f N$ be an exact sequence of $\Z^n$-graded
  $S$-modules and $q$ an order preserving map from $\N^n$ or $\Z^n$, to $\Z^n$.
  The induced sequence
  $q^*L \xrightarrow {q^*g}  q^*M \xrightarrow
  {q^*f}  q^*N$ is exact. 
\end{lem}

%\begin{example}
%  When we defined the $\mbt$-truncation we considered the
%  order-preserving homomorphism $p_\mbt 
%  \colon \N^n \to \Z^n$. If we extend this to an order-preserving
%  homomorphism $p_{\mbt} \colon \Z^n_* \to \Z^n_*$ by letting
%  $p_{\mbt}(\mba) = -\infty$ for every $\mba$ in $\Z^n_* \setminus
%  \N^n$, then the $\mbt$-truncation of
%  $p_{\mbt}^*M$ is an example of the above more general construction.
%Note that $M$ is 
% positively $\mbt$-determined if and only if 
%the homomorphism $p_{\mbt}^*M \to M$ taking $m \in (p_{\mbt}^*M)_\mba =
%M_{p_{\mbt}(\mba)}$ to $x^{\mba - p_{\mbt}(\mba)} m \in M_{\mba}$ is an
%isomorhism. 
%\end{example}

For $\mbt \in \N^n$ every  positively $\mbt$-determined
$S$-module has a filtration of modules
which are isomorphic to
 positively $\mbt$-determined $S$-modules of the following form.
\begin{defn}
\label{intervelmodule}
  Given $\mbt$, $\mba$ and $\mbb$ in $\N^n$ with $\mbb \le \mbt$  and
$\mba \leq \mbb$. The  
  positively $\mbt$-determined {\em interval} $S$-module
  $K_{\mbt}\{\mba,\mbb\}$ is defined as:
  $$K_{\mbt}\{\mba,\mbb\}
  = p_{\mbt}^*((S/\mm^{\mbb-\mba+\one})
%(x_1^{b_1},\dots,x_n^{b_n})
(-\mba))$$
\end{defn}
Note that for all $\mbr$ in the interval $[0,\mbt]$ 
the graded piece $K_{\mbt}\{\mba,\mbb\}_\mbr$ 
is $K$ if $\mbr \in [\mba,\mbb]$ and is zero 
if $\mbr \notin [\mba,\mbb]$. 
It will also be convenient to have the convention that 
$K_{\mbt}\{\mba, \mbb\}$ is zero if $\mba$ is not less or equal to $\mbb$.

%Note also that $K_\mbt \{\mba, \mbb\}$ is 
%nonzero only if $\mba \leq \mbb$. That we allow $\mba$ in the range
%up to $\mbb + \one$ is a matter of technical convenience.
%\begin{remark}
%  If $\mba$ and $\mbb$ are in $[0,\mbt]$ and $M$ is a 
%  positively $\mbt$-determined $S$-module with $M_{\mbr}= 0$ unless $\mbr \le \mbb$
%  then there is
%  an isomorphism
 % \begin{displaymath}
%    \Hom_K(K,M_{\mba}) \xrightarrow \cong \HOM_S(K_{\mbt}\{\mba,\mbb\},M)
%  \end{displaymath}
%  taking $g \colon K \to M_{\mba}$ to the homomorphism 
%  $
%  G
%  \colon K_{\mbt}\{\mba,\mbb\} \to M
%  $ with 
%$G_{\mbr}$
% $$
%   f_\mbr \colon K =
%   K_{\mbt}\{\mba,\mbb\}_{\mbr} \to M_{\mbr}
% $$ 
%  given by the formula
%  $G_{\mbr}(\lambda) = x^{\mbr -\mba}g(\lambda)$ for all $\mbr$ in
%  $[\mba,\mbb]$. 
% Moreover, if $g$ is an isomorphism, then $G$ is an 
%   isomorphism if and only if $M$ is an indecomposable 
%   positively $\mbt$-determined $S$-module with
%   $M_{\mbr} = K$ 
%   if $\mbr \in [\mba,\mbb]$ and with $M_{\mbr} = 0$
%   if $\mbr$ is in $[0,\mbt]$ but not in $[\mba,\mbb]$,
%\end{remark}
\begin{example}
\label{canex}
  Let us consider the case $n=1$.
% We begin in 
% this section with the case
% $n=1$ and $S = K[x]$. 
Given $t \in \N$ and $0 \le a \le b \le t$ we have
 % In this situation we have the following description
% of the  $t$-determined $S$-module $K_{t}\{r,s\}$
% introduced in Definition \ref{intervelmodule}.
\begin{displaymath}
  K_{t}\{a,b\} = 
%p_t^*((S/x^{b-a+1}S)(-a)) = 
  \begin{cases}
    (S/x^{b-a+1}S)(-a) & \text{if $b < t$} \\
    S(-a) & \text{if $b=t$}.
  \end{cases}
\end{displaymath}
Note that if $a \le a' \le b$, there is an
inclusion $i \colon K_{t}\{a',b\} \to K_{t}\{a,b\}$, and that if $a
\le b \le b'$ there is a surjection $p \colon K_{t}\{a,b'\} \to
K_{t}\{a,b\}$. We shall call $i$ the canonical inclusion and $p$ the
canonical projection.
\end{example}
Taking tensor products over $K$ we have the isomorphism
\begin{displaymath}
  K_{\mbt}\{\mba,\mbb\} \cong \bigotimes_{j=1}^n
  K_{t_j}\{a_j,b_j\}    
\end{displaymath}
of modules over $S = \bigotimes_{j=1}^n K[x_j]$.

\subsection{Local cohomology}
\label{sec:loccohny}

We let $\EXT$ denote the $\Z^n$-graded version of the $\Ext$-functor.
%With the notation $\one = (1,\dots,1)$
The local cohomology modules of a $\Z^n$-graded $S$-module $M$ are the
$\Z^n$-graded $S$-modules
\begin{equation*}
H^i_\mm(M) = \colim{k} \EXT^i_S(S/(\mm^{k}), M), \label{LokLigLoKo}
\end{equation*}
where $\mm = (x_1,\dots,x_n)$ is the unique monomial maximal ideal in $S$.
Instead of the modules $\mm^k$ any sequence of modules containing these as
a cofinal sub-sequence
may be used. In particular 
this local cohomology module is the colimit of 
%the $\EXT$-groups
$
%\begin{equation*}
\EXT^i_S(S/\mm^{\mbt}
%(x_1^{t_1 +1}, \ldots, x_n^{t_n + 1})
, M) \label{LokLigExtl}
%\end{equation*}
$
as $\mbt$ tends to infinity. 
Given $t \in \N$
we let
$P_{t\varepsilon_j}$ 
denote the projective resolution of 
$S/x_j^{t}S$
given by the inclusion  $x_j^tS \to S$.
% $$x_j^{t}S
% \xrightarrow {\subseteq} S$$
%Here we use the convention that
%$M(\mba)$ is the $\Z^n$-graded $S$-module with $M(\mba)_{\mbb} = M_{\mba+\mbb}$.  
We will work with chain complexes
throughout, that is, differentials decrease the homological degree. If
$C_*$ is a chain complex we shall write $H^iC$ for $H_{-i}C$.
The $\Ext$-group $\EXT^i_S(S/\mm^{\mbt}, M)$
is the
$i$'th cohomology module of the chain complex
$
%\begin{displaymath}
  \HOM_S(P_{\mbt},M)
%\end{displaymath}
$
where $P_{\mbt}$, given by the tensor product $P_{\mbt} =
\bigotimes_{j=1}^n P_{t_j \varepsilon_j}$ over $S$, is the 
Koszul complex for $S/\mm^{\mbt}$ and $\HOM_S$ denotes the
$\Z^n$-graded version of $\Hom_S$.
%(x_1^{t_1 +1}, \ldots, x_n^{t_n + 1})$.
\begin{prop}
\label{extstab}
  Let $\mbt \in \N^n$ and let 
  $M$ be a  positively $\mbt$-determined $S$-module. If $\mbt'
  \ge \mbt$ then the homomorphism  
  $$p_\mbt^* \EXT^i_S(S/\mm^{\mbt+\one}, M(-\one)) \to p_\mbt^* \EXT^i_S(S/\mm^{\mbt'+\one},
  M(-\one))$$
  is an isomorphism for all $i$.
\end{prop}
\begin{proof}
  The homomorphism $P_{\mbt'+\one} \to P_{\mbt+\one}$ induces an isomorphism
$$
  p_{\mbt}^*\HOM_S(P_{\mbt+\one},M(-\one)) \to
  p_{\mbt}^*\HOM_S(P_{\mbt'+\one},M(-\one)).
$$
In order to see this it suffices to consider the case $n=1$
 and to show that the
homomorphism
$$
  M(t) \cong \HOM_S(x^{t+1}S,M(-1)) \to
  \HOM_S(x^{t'+1}S,M(-1)) \cong M(t')
$$
induces an isomorphism $p_{t}^* M(t) \to p_{t}^* M(t')$. This is the
case because since $M$ 
is $t$-determined the multiplication by $x^{t'-t}$ induces an isomorphism
\begin{displaymath}
  M(t)_a = M_{t+a} \xrightarrow {x^{t'-t}} M_{t'+a} = M(t')_a
\end{displaymath}
for all $a \ge 0$. 
\end{proof}
The above proposition shows that if $M$ is
positively $\mbt$-determined then the $\mbt$-truncation of $H^i_\mm(M(-\one))$ is
isomorphic to the $\mbt$-truncation of $\EXT^i_S(S/\mm^{\mbt+\one},
M(-\one)).$ The following proposition shows that in this case the
$\Z^n$-graded $S$-module $H^i_\mm(M(-\one))$ is completely
determined by its $\mbt$-truncation.
\begin{prop}
\label{extlok}
  Let $\mbt \in \N^n$ and let 
  $M$ be a  positively $\mbt$-determined $S$-module.   
\begin{enumerate}
\item If $a_j \geq t_j+1$ for some $j$ then $H^i_{\mm}(M(-\one))_{\mba}$ is zero.
\item If $a_j \leq -1$ for some $j$ then the multiplication map
\[ H^i_{\mm}(M(-\one))_{\mba} \xrightarrow {x_j} H^i_{\mm}(M(-\one))_{\mba + \varepsilon_j} \]
is an isomorphism.  
\end{enumerate}
\end{prop}
\begin{proof}
  Since the statements only concern the $j$'th graded direction we can
  use the adjunction between $\otimes$ and $\HOM$ to reduce to
  the case $n=1$. For (1) we note that if $M$ is $t$-determined,
  then the homomorphism
  \begin{displaymath}
    M_{a-1} \cong \HOM_S(S,M(-1))_a \to \HOM_S(x^rS,M(-1))_a \cong M_{a-1+r},
  \end{displaymath}
  induced by multiplication by $x^r$, is an isomorphism whenever $a \ge
  t+1$. Thus the group $\Ext^i_S(S/\mm^r,M(-1))_a$ is zero for all $r$
  and $i$. This implies
  that the degree $a$ part of local cohomology of $M$ is zero.

  For (2) we consider the resolution $P_r = (x^rS \to S)$ of
  $S/x^{r}S$.
  We shall show that for $a \le -1$ the multiplication 
  \begin{displaymath}
    x \colon \HOM_S(P_r,M(-1))_{a} \to \HOM_S(P_r,M(-1))_{a+1}
  \end{displaymath}
  is an isomorphism for all sufficiently large $r$. 
  Since $M_b = 0$ for $b<0$ the zero cochains of both $\HOM_S(P_r,M(-1))_{a}$ and
  $\HOM_S(P_r,M(-1))_{a+1}$ are zero.
  Thus we only need to
  consider one cochains, that is, the multiplication
  \begin{displaymath}
    x_j \colon M_{r+a-1} \to M_{r+a}.
  \end{displaymath}
  Since $M$ is  $t$-determined this is an isomorphism
  provided that $r+a \ge t+1$. 
\end{proof}

\begin{remark} Using the terminology of \cite{Miller2000}, Section 2, this may be phrased
as saying that the local cohomology of $M(-\one)$ is negatively 
$\mbt$-determined. 
This may be demonstrated by providing $\mbt + \one$-determined 
injective resolutions of $M(- \one)$.
%which 
%Another way to demonstrate this fact is to use that
%any minimal $\Z^n$-graded injective
%resolution of $M(-\one)$ is $\mbt + \one$-determined in the
%sense of \cite{Miller2000}, Section 2. That this is so follows by appropriate use of 
%exact functors from loc.cit. on this resolution, which fix $M(-\one)$.
We thank E. Miller for pointing this out. 
%This argument shows more generally that the local cohomology of $M(-\one)$
%with arbitrary monomial support is also $\mbt$-determined.
\end{remark}

\subsection{The Nakayama Functor}
\label{sec:nakcomes}

We now shift attention from the local cohomology of a
$\Z^n$-graded $S$-module $M$ to the groups
$\EXT^i_S(S/\mm^{\mbt+\one},M(-\one))$ and the chain complexes that
compute them.
We construct an endofunctor $M \mapsto \Na^{\one}_{\mbt}(M)$, called the Nakayama functor,
on the category of chain complexes of  positively $\mbt$-determined $S$-modules.
The Nakayama functor is designed so that
  for every $\mbt \ge 0$ there is an isomorphism
  \begin{displaymath}
    H^i \Na^{\one}_{\mbt}(M) \cong p_{\mbt}^*\EXT^i_S(S/\mm^{\mbt+\one},M(-\one))
  \end{displaymath}
  of $\Z^n$-graded $S$-modules.
In view of Propositions \ref{extstab} and \ref{extlok} 
the local cohomology of a 
positively $\mbt$-determined $S$-module $M$ is completely determined by the
cohomology of $\Na^{\one}_{\mbt}(M)$, and vice versa. In fact
$H^i \Na^{\one}_{\mbt}(M) $ will be $p_{\mbt}^*$ of the local cohomology
module $H^i_{\mm} M(-\one)$.

Recall that for $t \ge 0$ we denote the resolution $x_j^{t}S \to S$ of
$S/x_j^{t}S$ by
$P_{t\varepsilon_j}$. 
Given $\mbk \in \{0,1\}^n$ we let
$P_{\mbt}^{\mbk}$ denote the tensor product $P_{\mbt}^{\mbk} =
\bigotimes_{j=1}^n P_{t_j \varepsilon_j}^{\otimes k_j}$ with the
convention that undecorated tensor products are over $S$ and that the
zeroth tensor power of an $S$-module is $S$ 
itself. 

\begin{defn}
\label{DefineNak}
  Given $\mbt \in \N^n$ and $\mbk \in \{0,1\}^n$ the $\mbk$-iterate of
  the $\mbt$-Nakayama functor is the endofunctor on the 
  category of chain complexes of  positively $\mbt$-determined $S$-modules given by 
the formula
  \begin{displaymath}
    \Na^{\mbk}_{\mbt}(C) = p_{\mbt}
    ^*\HOM_S(P^{\mbk}_{\mbt+\one},C(-\mbk)).
  \end{displaymath}
\end{defn}

%   Whenever convenient we shall consider $\Na^{\mbk}_{\mbt}$ as an
%   endofunctor on the category of  positively $\mbt$-determined
%   $S$-modules. 
% \begin{remark}
%   Considered as an endofunctor on the category of 
%   positively $\mbt$-determined $S$-modules the $0$-iterate of the $\mbt$-Nakayama
%   functor is isomorphic to the identity functor. However, since the
%   $\mbt$-Nakayama functor takes values in the category of 
%   positively $\mbt$-determined $S$-modules $\Na^{0}_{\mbt}(M)$ is only isomorphic
%   to $M$ if $M$ is  positively $\mbt$-determined.  
%   The above definition has the defect that if $\mbk = 0$ then 
% \end{remark}
% Let $\frac \mbt {|{\mbk}|} \in \Z^n$ be defined by the formula
% \begin{displaymath}
%   \left(\frac \mbt {|{\mbk}|}\right)_j =
%   \begin{cases}
%     t_j & \text {if $k_j \ne 0$} \\
%     \infty & \text{if $k_j = 0$.} 
%   \end{cases}
% \end{displaymath}

%With the notation
%$\mbn = \mbn({\varepsilon_j},{t \varepsilon_j}) = (t+1)\varepsilon_j - \one$
%Given a $\Z^n$-graded $S$-module $M$ 
% \begin{defn}
% \label{definenakextra}
%   Given $t \in \N^n$ and $\mbk \in \{0,1\}^n$ the $(\mbt,\mbk)$-th
%   Nakayama functor is the endofunctor on the 
%   category of chain complexes of $\Z^n$-graded $S$-modules given by the formula
%   \begin{displaymath}
%     \Na^{\mbk}_{\mbt}(C) = p_{\frac \mbt {|\mbk|}}
%     ^*\HOM_S(P^{\mbk}_{\mbt+\one},C(-\mbk)).
%   \end{displaymath}
%   Whenever convenient we shall consider $\Na^{\mbk}_{\mbt}$ as an
%   endofunctor on the category of  $\frac \mbt {|\mbk|}$-determined
%   $S$-modules. 
% \end{defn}

We shall shortly generalize this so that we for any $\mbk \in \N^n$ define a 
functor $C \mapsto \Na^{\mbk}_{\mbt}(C)$. Our goal is to compute the cohomology of 
$\Na^{\mbk}_{\mbt}(S/I)$ for a positively $\mbt$-determined monomial ideal $I$. These 
computations will eventually reduce to the case when $n=1$ and the module is an 
interval module, as in the following example.

\begin{remark} As said above, when $\mbk = \one$, the cohomology of 
$\Na^{\one}_{\mbt}(M)$ for a module $M$
is essentially the local cohomology of $M(- \one)$. Theorem 6.2 of 
\cite{Miller2000} or rather its proof computes the local cohomology module
in essentially the same way as in the above definition.
\end{remark}

\begin{example}
\label{firstnakone}
 Let $n=1$ and $0 \le a \le
  b \le t$. There are isomorphisms
  \begin{displaymath}
    \HOM_S(x^{t+1}S,K_t\{a, b \}(-1)) \cong K_t\{a, b \}(t) 
%\cong
%    \begin{cases}
%      K_t\{a-t,b-t\} & b < t \\
%      S(t-a) & b = t
%    \end{cases}
  \end{displaymath}
  and
  \begin{displaymath}
    \HOM_S(S,K_t\{a, b \}(-1)) \cong K_t\{a, b \}(-1) 
%    \cong
%    \begin{cases}
%      K_t \{a+1,b+1\} & b < t \\
%      S(-1-a) & b=t.
%    \end{cases}
  \end{displaymath}
Taking the $\mbt$-truncation, it follows that if $b<t$ then $\Na^1_t(K_t\{a, b \})$ 
is isomorphic to the chain complex
  \begin{displaymath}
    K_t\{a+1,b+1\} \to 0,
  \end{displaymath}
  and if $b=t$ then $\Na^1_t(K_t\{a, b \})$ is isomorphic to the chain complex
  \begin{displaymath}
    K_t\{a+1,t\} \to K_t\{0,t\}.
  \end{displaymath}
  On cohomology we obtain
  \begin{displaymath}
    H^* \Na^1_t K_t\{a,b\} \cong
    \begin{cases}
      K_t\{a+1,b+1\} e^0 & \text{if $b <t$} \\
      K_t\{0,a\} e^{1} & \text{if $b=t$}.
    \end{cases}
  \end{displaymath}
    Here $e^0$ is a generator in cohomological degree zero and
    $e^{1}$ is a generator in cohomological degree one.
\end{example}

Note
that if $\mbk,\mbk' \in \{0,1\}^n$ are such that $\mbk+\mbk' \in
\{0,1\}^n$, then the adjunction between $\HOM$ and the tensor product
over $S$ induces an isomorphism 
\begin{displaymath}
  \Na_{\mbt}^{\mbk'}(\Na^{\mbk}_{\mbt}(C)) \cong \Na^{\mbk' + \mbk}_{\mbt}(C).
\end{displaymath}
We now define $\Na^{\mbk}_{\mbt}$ in general. 

\begin{defn} \label{DefineNak2}
Given multidegrees $\mbk = (k_1,\dots,k_n)$ and 
$\mbt = (t_1,\dots,t_n)$ in $\N^n$, 
we define 
the $\mbk$-iteration of the  $\mbt$-Nakayama functor on $C$, 
denoted $\Na^{\mbk}_{\mbt}(C)$, to be the chain complex 
\begin{displaymath}
  \Na^{\mbk}_{\mbt}(C) = (\Na^{\varepsilon_1}_{t_1
    \varepsilon_1})^{\circ k_1} \circ \dots \circ (\Na^{
  \varepsilon_n}_{t_n \varepsilon_n})^{\circ k_n} (C).
\end{displaymath}
This gives an endofunctor $\Na^{\mbk}_{\mbt}$ on the
category of chain complexes of  positively $\mbt$-determined
$S$-modules.
\end{defn}

\begin{remark}
In this paper we work in the category of chain complexes, and not
in the derived category. The reason is that we want to work with
explicit morphisms between chain complexes.
The only place we refer to the derived category is in Section 3
where we relate the Nakayama functor to the Auslander-Reiten translate.
\end{remark}

\begin{remark} \label{RemDefineNak2}
Two facts we will use on some occasions are the following. 
Since $P^{\mbk}_{\mbt+\one}$
in Definition \ref{DefineNak} is a bounded 
complex of finitely generated  projectives,
the Nakayama functor commutes with colimits.
Also, since the functor $\HOM_S(P^{\mbk}_{\mbt+\one},-)$ is a right adjoint
it commutes with limits, and so also the Nakayama functor.
\end{remark}

One might expect that the effect of the truncation $p_{\mbt}$ is
minor. However this is not the case.
Given a chain complex $C$ we let $TC$ denote the homologically shifted
chain complex
with $(TC)_i = C_{i-1}$ and with $d_{TC} = -d_C$. 
In Section \ref{sec:intmod} we show that the truncation $p_{\mbt}$ gives the 
Nakayama functor $\Na^{\mbk}_{\mbt}$ the
following remarkable property, generalizing a result of Yanagawa 
\cite{Yanagawa2004}
in the case $\mbt = \one$.  
\begin{prop}
\label{trans}
  Let $\mbt \in \N^n$ and let  $C$ be a
  chain complex of  positively $\mbt$-determined $S$-modules. For every 
  $j$ the the evaluation 
  $\Na^{(t_j +2)\varepsilon_j}_{\mbt} C$ of the Nakayama functor
  at $C$ is naturally quasi-isomorphic to $T^{-2}C$.
  In particular, for every $\mbk$ in $\N^n$ the Nakayama functor
  $\Na^{\mbk}_{\mbt}$ induces a self-equivalence on the derived category of
  the category of chain complexes of  positively $\mbt$-determined $S$-modules.
\end{prop}

% The following result shows that Lemma \ref{nakstab} implies
% that if $M$ is 
% $(\mbt -\one)$-determined then the
% cohomology of $\Na^\one_{\mbt}(M)$ 
% determines the local cohomology of $M$. 
% \begin{cor}
% \label{comparelocnak}
% Let $M$ be a  $(\mbt-\one)$-determined module and $\mba$ a multidegree in $\Z^n$.
% Then
% \begin{enumerate}
% \item There is an isomorphism of $S$-modules of the form
%   $H^i_\mm(M)(-\one) \cong H^i(\Na^{\one}_{\mbt}(M))$.
% \item If $\mbt - \one \ge \mba \ge 0$ then there is an isomorphism
%   $H^i_\mm(M)_{\mba-\one} \cong H^i(\Na^{\one}_{\mbt-\one}(M))_{\mba}$.
% \item If some $a_j \geq t_j$ then $H^i_{\mm}(M)_{\mba-\one}$ is zero.
% \item If some $a_j \leq -1$ the multiplication map
% \[ H^i_{\mm}(M)_{\mba-\one} \xrightarrow {x_j} H^i_{\mm}(M)_{\mba -\one
%   + \varepsilon_j} \]
% is an isomorphism.  
% \end{enumerate}
% \end{cor}
% \begin{proof}
%   Only part (4) is not a direct consequence of Lemma
%   \ref{nakstab}. 
% \end{proof}

% With this convention we have 
% \begin{displaymath}
%   H^i \Na^{\one}_{\mbt} (M)_{\mba} \cong \EXT^i_S(S/(x_1^{t_1+1},\dots,x_n^{t_n+1})S,M)_{\mba-\one}
% \end{displaymath}
% for $0 \le a \le \mbt$.

\section{Duality Functors}
\label{sec:dufu}

In this section we describe the Nakayama functor as the
composition of two duality functors. 
The Nakayama functor in the theory of Artin algebras is defined
on the category of chain
complexes of left modules
over an Artin algebra.
It is the
composition of two duality functors and it represents the
Auslander--Reiten translate on the bounded derived category of chain complexes
of fintiely generated positively $\mbt$-determined modules. 
Let us denote the incidence algebra of the partially ordered set $P$
consisting of the multidegrees $\mba \in \Z^n$ satisfying $0 \le \mba
\le \mbt$ by $\Lambda_\mbt$.
We shall show that under an
equivalence of
categories between the category of 
 $\mbt$-deter\-mined $S$-modules and the category of 
left $\Lambda_\mbt$-modules
the functor $\Na^{\one}_{\mbt}$ corresponds to the
Nakayama functor from the theory of Artin algebras.

This section is not a prerequisite for the rest of the paper. 
The reader who wishes may safely go directly to Section \ref{sec:mainresults}
where the main results are presented.

\noindent{\it Note.} Our definition of the Nakayama functor, Definition 
\ref{DefineNak2}, works well wether $M$ is finitely generated or not.
However since we here compare it to the composition of two duality functor,
{\it we assume in this section that all modules have graded parts which are
finite-dimensional vector spaces}, in particular all positively $\mbt$-determined
modules are finitely generated.

% \subsection{Digression}
% \label{sec:ovdu}
% In this subsection we 
% Let us open this section by 
% a brief overview on how the Nakayama functor constructed
% in the previous section is related to local duality and to the
% Nakayama functor as it 
% occurs in the theory of Artin algebras. 

% We will use
% the concept of derived categories freely here. In the next subsection we
% shall return to the more concrete situation where we study explicit
% chain complexes. 

\subsection{Local duality}
\label{sec:locdua}

There are two highly different duals of a $\Z^n$-graded $S$-module
$M$. One is the
$\Z^n$-graded Matlis dual $A(M) = \HOM_S(M,E)$, where $E$ is the
$\Z^n$-graded injective envelope $E = K[x_1^{-1},\dots,x_n^{-1}]$ of $S/\mm$. The other is the
chain complex $D(M) = \HOM_S(M,Q)$ where $Q$ is an
injective resolution of
$T^nS(-\one)$.
Here
$T^n$ shifts homological degrees up by $n$. 
% in the derived category of the
% category of chain complexes of $\Z^n$-graded $S$-modules. 
% Here $\one =
% (1,\dots,1)$. 
The complex $Q$ is a dualizing complex for finitely generated $\Z^n$-graded $S$-modules
in the sense that the natural 
homomorphism $M \to (D \circ D)(M)$ is a quasi-isomorphism for every
finitely generated $\Z^n$-graded $S$-module $M$.
Grothendieck's local duality theorem  provides an
% Moreover 
% The complex $Q$ is a twisted dualizing complex for cohomology with support in
% $\mm$ in the sense that 
% there is an 
isomorphism of $\Z^n$-graded $S$-modules between
% The local duality theorem implies 
% that the 
%an isomorphism between the 
%cohomology 
%$S$-module 
 $H^{i}(A \circ D(M))$ 
%is isomorphic to 
and the local cohomology $H^i_{\mm}(M)$ \cite[Example 3.6.10 and Theorem 3.6.19]{BH}.
Below we study 
twisted versions $M \mapsto p_{\mbt}^* A(M(\mbt))$ and $M
\mapsto D(M(\mbt-\one))$ of the duality
functors $A$ and $D$.

\subsection{Alexander duality}
\label{sec:alexdu}

In \cite{Miller2000} Miller introduced the Alexander duality functor on the category
of  positively $\mbt$-determined $S$-modules. The value on a 
positively $\mbt$-determined $S$-module $M$ is given by the formula
\begin{displaymath}
  A_{\mbt}(M) = p_{\mbt}^* A(M(\mbt)) = p_{\mbt}^* \HOM_S(M(\mbt),E).
\end{displaymath}
The functor $A_{\mbt}$ generalizes Alexander duality for
simplicial complexes. 
%Miller uses it to genearlize
%a result of
%Eagon and Reiner \cite{ER} saying that if a monomial ideal $I$ in $S$ is squarefree, that is,
%$\one$-determined, then $I$ 
%has a linear free 
%resolution if and only if $A_{\one}(S/I)$ is Cohen--Macaulay.

In order to make the definition of $A_{\mbt}(M)$ more explicit we note that 
if $\mba \in [0,\mbt]$, then there is an
isomorphism
\begin{displaymath}
  A_{\mbt}(M)_{\mba} \cong \HOM_K(M({\mbt}),K)_{\mba} = 
\Hom_K(M_{\mbt - \mba},K).
\end{displaymath}
If $\mba +
\varepsilon_i \le \mbt$, then the multiplication map  
$$x_i \colon
A_{\mbt}(M)_{\mba} \to A_{\mbt}(M)_{\mba+\varepsilon_i}
$$
is dual to the multiplication
\begin{displaymath}
  x_i \colon M_{\mbt-\mba-\varepsilon_i} \to M_{\mbt-\mba}.
\end{displaymath}
% \begin{displaymath}
%   A_{\mbt}(M) = p_{\mbt}^* \HOM_K(M(\mbt),K)
%   \cong p_{\mbt}^*  \HOM_S(M(\mbt),K[x_1^{-1},\dots,x_n^{-1}]).
% \end{displaymath}
% Explicitely, 
%We leave it as an exercise for the reader to check that 
If $M$ is
 positively $\mbt$-determined, then $A_{\mbt} \circ A_{\mbt}(M)$ is
isomorphic to $M$.
% Thus the Nakayama functor  
% In this section we modify the dualities $D$ and $A$ so that we obtain
% dualities $D_{\mbt}$ and $A_{\mbt}$ on the category of  positively $\mbt$-determined
% $S$-modules with $\Na^{\one}_{\mbt}$ quasi-isomorphic to $A_{\mbt} \circ D_{\mbt}$.
% The truncation functor
% $M \mapsto p_{\mbt}^*M$ associates to $M$ the  positively $\mbt$-determined
% $S$-module determined by the information the $S$-module sturcture
% gives on the $K$-vector spaces $M_{\mba}$ for $0 \le \mba \le \mbt$.
% The $\Z^n$-graded $S$-module $M$ is positively $\mbt$-determined if it is
% isomorphic to $p_{\mbt}^*M$.
% The duality functor $D_{\mbt}$ and Miller's Alexander duality functor
% $A_{\mbt}$ are easy to explain. Explicitely the Alexander dual $A_{\mbt}(M)$ of a
%  positively $\mbt$-determined $S$-module $M$ is the
% twisted and truncated version $A_{\mbt}(M) = p_{\mbt}^*(A(M)(-\mbt))$ of Matlis duality.

\begin{example}
  Let $I$ be a monomial ideal
  in $S$ and suppose that the minimal generators in degrees less than
  or equal to $\mbt$ are given by the monomials
  $x^{\mba^1},\dots,x^{\mba^r}$.
  Following Miller and   
  Sturmfels \cite{MiSt} we let $I^{[\mbt]}$ denote the Alexander dual positively $\mbt$-determined
  ideal 
  \begin{displaymath}
    I^{[\mbt]} = \mm^{\mbb^1} \cap \dots \cap \mm^{\mbb^r},
  \end{displaymath}
  where 
  \begin{displaymath}
    b^i_j =
    \begin{cases}
      t_j + 1 - a^i_j & \text{if $a^i_j \ge 1$} \\
      0 & \text{if $a^i_j = 0$}.
    \end{cases}
  \end{displaymath}
%  It is a straight forward exercise to check that
%   Now note that for $0 \le \mba \le \mbt$ we have that $I_{\mba} = K$
%   if and only if $\mba^i \le \mba$ for some $i$. 
%   It follows that $A_{\mbt}(I)_a = \Hom_K(I_{\mbt-\mba},K)$ is equal
%   to $K$ if and only if $\mba^i \le \mbt - \mba $ for some $i$.
%   On the other hand
%   $I^{[\mbt]}_a = 0$ if and only if there is an $i$ such that $b^i_j >
%   0$ implies $a_j < b^i_j$. Using the definition of $b^i_j$ we see
%   that $I^{[\mbt]}_\mba = 0$ if
%   and only if $\mba \le \mbt - \mba^i$ for some $i$. Thus
%   $I^{[\mbt]}_a = 0$ if and only if
%   $A_{\mbt}(I)_a = K$. With this information at hand it is easy to see
  The $\Z^n$-graded $S$-modules $A_{\mbt}(I)$ and  $S/I^{[\mbt]}$ are isomorphic.
%  of positively $\mbt$-determined $S$-modules.
\end{example}
\begin{example}
  Let $n=1$ and let $M = K_t\{a,b\}$ for $0 \le a \le b \le t$. There
  is an isomorphism $A_t(M) \cong K_t\{t-b,t-a\}$. In particular 
  \begin{displaymath}
    A_t(K_t\{a,t\}) = A_t(S(-a)) =
    \begin{cases}
      S & \text{if $a=0$} \\
      S/x^{t-a+1}S & \text{if $a>0$.} 
    \end{cases}
  \end{displaymath}
\end{example}

\subsection{A twisted duality}
\label{sec:matdua}

We shall now study the twisted duality
\begin{displaymath}
  M \mapsto D(M(\mbt-\one)) = \HOM_S(M,Q(\one -\mbt)).
\end{displaymath}
It has the property that $p_{\mbt}^*\circ A \circ D(M(-\one))$ is equal
to $A_{\mbt}\circ D(M({\mbt- \one})))$. 
Since we are interested in the Nakayama functor
$\Na_{\mbt}^{\one}(M)$ on a  positively $\mbt$-determined $S$-module
$M$, it will be convenient to work with a version $M \mapsto
D_{\mbt}(M)$ of this twisted duality functor 
which restricts to an endofunctor on the category of 
positively $\mbt$-determined
$S$-modules.
%with the property that 
%$\Na_{\mbt}^{\one}(M)$ is 
%isomorphic to $A_{\mbt} \circ D_{\mbt}(M)$.

For $\mbt \in \N^n$ recall the comples $P_\mbt = P_\mbt^\one$ of Subsection
\ref{sec:nakcomes}.  We let $I_{\mbt}$ denote the injective resolution
\begin{displaymath}
  I_{\mbt} =
  \HOM_S(\HOM_S(P_{\mbt},S),K[x_1^{-1},\dots,x_n^{-1}]) \cong
  \HOM_K(\HOM_S(P_{\mbt},S),K)
\end{displaymath}
of $T^nS/\mm^{\mbt}(-\one)$. We define $D_{\mbt}$ to be
the endofunctor on the category of chain complexes of $\Z^n$-graded
$S$-modules given by the formula
\begin{displaymath}
  D_{\mbt}(C) = p_\mbt^* \HOM_S(C,I_{\mbt+\one}(\one-\mbt)).
\end{displaymath}
We shall show in Proposition  \ref{DualProEkvdu} 
that this is essentially the same as the functor $D(M(\mbt- \one))$.

\begin{example}
\label{Dexmaplen=1}
  Let $n=1$. Then
  $I_{t+1}$ is isomorphic to the chain complex $K[x^{-1}](-t-1) \to
  K[x^{-1}]$ with $K[x^{-1}](-t-1)$ in degree one. Let $0 \leq a \leq b \leq t$.
  There
  are isomorphisms
  \begin{displaymath}
    \HOM_S(K_t\{a, b\},K[x^{-1}](-t-1))(1-t) \cong \HOM_K(K_t\{a, b\},K)(-2t)
  \end{displaymath}
  and
  \begin{displaymath}
    \HOM_S(K_t\{a, b\},K[x^{-1}])(1-t) \cong \HOM_K(K_t\{a, b\},K)(1-t).
  \end{displaymath}
  If $b < t$ then $D_t(K_t\{a, b\})$ is isomorphic to the chain complex
  \begin{displaymath}
    0 \to K_t\{t-b-1,t-a-1\} 
  \end{displaymath}
  and if $b=t$, it is isomorphic to the chain complex
  \begin{displaymath}
    K_t\{0,t\} \to K_t\{0,t-a-1\}.
  \end{displaymath}
  On homology we obtain an isomorphism
  \begin{displaymath}
    H_* D_t(K_t\{a,b\}) \cong
    \begin{cases}
      K_t\{t-b-1,t-a-1\} e_{0} & \text{if $b<t$} \\
      K_t\{t-a,t\} e_1 &\text{if $b=t$}.
    \end{cases}
  \end{displaymath}
    Here $e_0$ is a generator in homological degree zero and
    $e_{1}$ is a generator in homological degree one.
\end{example}

\begin{lem}
\label{convertnaka}
  Let $M$ be a finitely generated $\Z^n$-graded $S$-module. For every
  $\mbr \in \N^n$ there is a natural isomorphism
  \begin{eqnarray*}
    \HOM_S(P_{\mbr},M) &\cong&
    \HOM_K(\HOM_S(M,I_{\mbr}),K) 
% \\
%     &=&
%     \HOM_K(\HOM_S(M,I^{\varepsilon_j}_{t_j}(-t_j\varepsilon_j)),K)(-t_j\varepsilon_j). 
  \end{eqnarray*}
\end{lem}
\begin{proof}
  First note that $M \otimes_S
  \HOM_S(P_{\mbr},S)$ is naturally isomorphic to $\HOM_S(P_{\mbr},M)$. 
%This isomorphisms gives the result
  Together 
with the adjunction isomorphism
  \begin{displaymath}
    \HOM_S(M,\HOM_K(\HOM_S(P_{\mbr},S),K)) \cong
    \HOM_K(M \otimes_S \HOM_S(P_{\mbr},S),K) 
  \end{displaymath}
  this gives an isomorphism $\HOM_S(M,I_{\mbr}) \cong
  \HOM_K(\HOM_S(P_{\mbr},M),K)$.
  The result now follows from the isomorphism $N \cong
  \HOM_K(\HOM_K(N,K),K)$ with $N = \HOM_S(P_{\mbr},M)$.

\end{proof}

The following is the main observation we make in this subsection.
\begin{cor} \label{AusCorAD}
  For every $\mbt \in \N^n$ there is an isomorphism
  \begin{displaymath}
    \Na^{\one}_{\mbt} \cong A_{\mbt} \circ D_{\mbt}
  \end{displaymath}
  of endofunctors on the category of  chain complexes of 
finitely generated positively $\mbt$-determined $S$-modules.
\end{cor}
\begin{proof}
  Applying the
  truncation $p_{\mbt}$ 
  and Lemma \ref{convertnaka} 
%with $\mbr = \mbt + \one$ 
  we obtain
  isomorphisms
  \begin{eqnarray*}
    p_{\mbt}^* \HOM_S(P_{\mbt + \one},C(-\one)) &\cong& 
    A_{\mbt}(\HOM_S(C,I_{\mbt + \one}(\one-\mbt))) \\
&\cong& A_{\mbt}(p_{\mbt}^*\HOM_S(C,I_{\mbt + \one}(\one-\mbt))) =  
A_{\mbt} \circ D_{\mbt}(C).
  \end{eqnarray*}

\end{proof}

 The following shows that our duality $D_{\mbt}$
is a version for positively $\mbt$-determined modules of the standard 
duality $D$. Also since $A_{\mbt}$ is exact,
the above corollary together with the following proposition implies
the local duality theorem for 
$\Z^n$-graded $S$-modules.
\begin{prop} \label{DualProEkvdu}
  Let $M$ be a  positively $\mbt$-determined $S$-module. There is a
  natural quasi-isomorphism between the chain
  complexes $D_{\mbt}(M)$ and $D(M(\mbt-\one))$. In particular there is a
  natural isomorphism
  \begin{displaymath}
    H_i D_{\mbt}(M) \cong \EXT^{n-i}_S(M,S(-\mbt)).
  \end{displaymath}
\end{prop}
\begin{proof}
  It suffices to show that if $P$ is a projective 
  positively $\mbt$-determined 
  $S$-module, then there is a quasi-isomorphism $T^n\HOM_S(P,S(-\mbt))
  \xrightarrow \simeq D_{\mbt}(P)$ and that this quasi-isomorphism is natural
  in $P$. However since $I_{\mbt+\one}$ is a resolution of
  $T^nS/\mm^{\mbt+\one}(-\one)$ 
  and $P$ is projective there is a quasi-isomorphism
  $$T^n p_{\mbt^*} [\HOM_S(P,S/\mm^{\mbt+\one}(-\one))(\one-\mbt)] \to
  D_{\mbt}(P).$$ 
  On the other hand every projective
  positively $\mbt$-determined $S$-module is a direct sum of $S$-modules of the
  form $S(-\mba)$ for $0 \le \mba \le \mbt$ and we have isomorphisms
  \begin{eqnarray*}
    p_{\mbt}^* [\HOM_S(S(-\mba),S/\mm^{\mbt + \one}S(-\one))(\one-\mbt)] &\cong&
    p_{\mbt}^* [(S/\mm^{\mbt + \one}S)(\mba-\mbt)]
    \\ & \cong&
    S(\mba -\mbt) \\
    &\cong& \HOM_S(S(-\mba),S(-\mbt)).
  \end{eqnarray*}

\end{proof}
%\begin{remark}
%  Musta\c t\v a has shown that if $B$ is a monomial ideal in $S$ with the
%  property that $B_{\mba} = 0$ for $\mba \le \mbt$, then the
%  homomorphism $\EXT_S^i(S/B,S(-\mbt))_{\mba} \to H^i_{B}(S(-\mbt))_{\mba}$ is an
%  isomorphism whenever $\mba \ge 0$ \cite[Theorem 1.1]{Mustata2000}.
%\end{remark}

\subsection{Squarefree modules}
\label{sec:sqfm}

In the case $\mbt = \one$  positively $\mbt$-determined
$S$-modules are termed squarefree $S$-modules by Yanagawa \cite{Yanagawa2000}. 
Alexander duality for square free modules, defined by T.R\"omer 
\cite{Romer2001} and
E.Miller \cite{Miller2000}, has been well studied. For instance the 
Alexander dual of a Stanley-Reisner ring $K[\Delta]$ is the 
Stanley-Reisner ideal $I_{\Delta^*}$ of the Alexander dual simplicial complex
$\Delta^*$.
In \cite{Floystad2004}, the cohomology modules of $\Na_{\one}^{\one}(S/I)$ were termed the enriched
cohomology modules of $\Delta$ by the second author, 
because the $S$-module rank of $H^i\Na_{\one}^{\one}(S/I)$ equals the
$K$-vector space dimension of the reduced cohomology $\widetilde
H^i(\Delta)$. 
 Yanagawa, \cite{Yanagawa2004},
observed that in the derived category of chain complexes of
positively $\one$-determined $S$-modules, which he calls the category
of squarefree
$S$-modules, the three times iterated Nakayama functor $\Na_{\one}^{\one} \circ\Na_{\one}^{\one}
\circ\Na_{\one}^{\one} (M)$ is isomorphic to a homological shift 
of $M$ \cite{Yanagawa2004}. This is an instance of Proposition \ref{trans}.

\subsection{Duality over incidence algebras}
\label{sec:posdu}

In order to give our next description of the Nakayama functor we note
that up to canonical isomorphism a  positively $\mbt$-determined
$S$-module $M$ is determined by the $K$-vector spaces $M_{\mba}$ for $0 \le
\mba \le \mbt$ and the multiplication homomorphisms $x_j \colon
M_{\mba-\varepsilon_j} \to M_{\mba}$ for $\varepsilon_j \le \mba \le \mbt$. This is
exactly the structure encoded by a (left) module over the incidence
algebra $\Lambda_{\mbt}$ of the poset $[0,\mbt] = \{\mba
\colon 0 \le \mba \le \mbt \}$.

Let us recall that to every finite
poset $P$ there is an associated incidence algebra, denoted
$I(P)$. It is the $\kr$-algebra which as a $\kr$-vector space has the
pairs $(q,p)$ (where $p \leq q$) as basis.
The product is given on basis
elements by $(r,q) \cdot (q,p) = (r,p)$, and $(r,q^\prime) \cdot (q,p)$
is zero when $q \not = q^\prime$. The identity element in the incidence
algebra is $\sum_{p \in P} (p,p)$.  A (left) module $M$ over the incidence
algebra may then be written as $M = \sum_{p \in P} (p,p) \cdot M$ or,
letting $M_p = (p,p) \cdot M$, we get $M = \oplus_{p \in P} M_p$.
%So $M$ is graded by elements of $P$.
The category of (left) modules over $I(P)$ is isomorphic to the category
of functors from $P$, considered as a category with a morphism $p \to p'$ if and
only if $p \le p'$, to the category of $K$-vector spaces. Such
functors are often called $KP$-modules. 
%\begin{remark}
%The incidence algebra of a partially ordered set is an instance of a
%more general construction where we associate an algebra $K[\C]$ to every
%finite category $\C$. 
%The morphisms of $\C$ form a $\kr$-basis for $K[\C]$ and
%the product is given on basis
%elements by $f \cdot g = f\circ g$, if $f$ and $g$ are composable
%morphisms in $\C$ and
%$f \cdot g = 0$ otherwise. 
%For example if $\C$ is the category with one object and with morphism
%set given by a group $G$, then $\kr[G]$ is the group algebra. On the
%other hand, considering $P$ as a category as we did above, the
%incidence algebra $I(P)$ is equal to the algebra $K[P]$.
% , where
% $P^{\op}$ is the opposite category of $P$.
% Note that if $\C^{\op}$ denotes the opposite category
% of $\C$, then $K[\C^{\op}]$ is the opposite algebra of $K[\C]$, and
% thus $I(P)$ is the opposite of the algebra $K[P]$.
%\end{remark}

Given two (left) $I(P)$-modules $M$ and $N$ we denote
by $\Hom_{I(P)}(M,N)$ the $K$-vector space of $I(P)$-linear homomorphisms
from $M$ to $N$. Since $I(P)$ is an
$I(P)$-bimodule, $\Hom_{I(P)}(M,I(P))$ becomes a right $I(P)$-module
which we may consider as a left $I(P^\op)$-module.
Explicitly
\begin{displaymath}
  \Hom_{I(P)}(M,I(P))_x = \Hom_{I(P)}(M,KP(x,-)) 
\end{displaymath}
where the latter is the $K$-vector space of $I(P)$-linear homomorphisms from 
$M$ to the
$I(P)$-module $KP(x,-)$ with $KP(x,-)_y = KP(x,y) = K$ if $x \le y$ and with 
$KP(x,-)_y = 0$ otherwise.

For $P = [0,\mbt]$ we shall denote the incidence algebra $I(P)$ by
$\Lambda_{\mbt}$.
There is an order preserving bijection
$\tau 
\colon [0,\mbt] \to [0,\mbt]^{\op}$ with $\tau(\mba) = \mbt-\mba$
and an induced homomorphism $\tau \colon \Lambda_{\mbt} \to \Lambda_{\mbt}^{\op}$. Given a
projective $\Lambda_{\mbt}$-module $M$ we shall be interested in the $\Lambda_{\mbt}$-module
$\tau^*\Hom_{\Lambda_{\mbt}}(M,\Lambda_{\mbt})$. 
\begin{prop}
\label{posetD}
  Let $\mbt \in \N^n$. If $M$ is a projective 
   positively $\mbt$-determined $S$-module, then considering $M$ and
  $D_{\mbt}(M)$ as $\Lambda_{\mbt}$-modules there is a quasi-isomorphism
  $D_{\mbt}(M) \to T^n\tau^*\Hom_{\Lambda_{\mbt}}(M,\Lambda_{\mbt})$. Moreover this
  quasi-isomorphism is natural in $M$.
\end{prop}
\begin{proof}
  Since every indecomposable projective  positively $\mbt$-determined
  $S$-module is isomorphic to 
  $x^{\mba}S$ for some $0 \le \mba \le \mbt$ it suffices to consider the
  case $n=1$. If $M$ is the  positively $\mbt$-determined $S$-module
  $M = x^aS = S(-a)$ corresponding to the $\Lambda_{\mbt}$-module
  $KP(a,-)$, then $\Hom_{\Lambda_{\mbt}}(M,\Lambda_{\mbt}) 
  \cong KP(-,a)$, and thus $\tau^*\Hom_{\Lambda_{\mbt}}(M,\Lambda_{\mbt})$ corresponds to the
  $S$-module $S(a-t)$. On the other hand, as we saw in Example
  \ref{Dexmaplen=1}, $D^1_t(M)$ is isomorphic to
  the chain complex
  \begin{displaymath}
    K_t\{0,t\} \to K_t\{0,t-a-1\},
  \end{displaymath}
  whose homology is $S(a-t)$ concentrated in homological degree one.
\end{proof}
\begin{lem}
\label{posetA}
  Let $\mbt \in \N^n$. If $M$ is a  
  positively $\mbt$-determined $S$-module, then considering $M$ and
  $A_{\mbt}(M)$ as $\Lambda_{\mbt}$-modules there is an isomorphism of the form
  $\tau^*\Hom_{K}(M,K) \cong A_{\mbt}(M)$. Moreover this
  isomorphism is natural in $M$.  
\end{lem}
\begin{proof}
  Since $\Hom_K(M,K)_{\mba} = \Hom_K(M_{\mba},K)$ we have 
  \begin{displaymath}
   (\tau^*\Hom_K(M,K))_{\mba} = \Hom_K(M_{\mbt-\mba},K) = A_{\mbt}(M)_{\mba} 
  \end{displaymath}
  We leave it to the reader to check that this is an isomorphism of $\Lambda_{\mbt}$-modules.
\end{proof}
Combining the above two results we obtain the following description of
the Nakayama functor.
\begin{prop}
\label{convertnak}
  If $F \to M$ is a projective resolution of a  positively $\mbt$-determined $S$-module
  $M$, then considering $F$, $\Na^{\one}_{\mbt}(F)$ and $\Na^{\one}_{\mbt}(M)$ 
  as $\Lambda_{\mbt}$-modules there are quasi-isomorphisms of the form
  \begin{displaymath}
    \tau^* \Hom_K(\tau^*\Hom_{P}(F,\Lambda_{\mbt}),K) \xrightarrow \simeq
    T^n\Na^{\one}_{\mbt}(F) \xrightarrow \simeq  T^n\Na^{\one}_{\mbt}(M).
  \end{displaymath}
\end{prop}
\begin{proof}
  The first quasi-isomorphism follows directly from Corollary \ref{AusCorAD}, Proposition
  \ref{posetD}, and Lemma 
  \ref{posetA}. The second quasi-isomorphism is a consequence of
  Lemma \ref{truncateexact} on exactness of the functors $M \mapsto
  p_{t_j}^*M$ and 
  the fact that the chain complexes $P_{\mbt+\one}$ are
  projective over $S$.
\end{proof}
\subsection{The Auslander-Reiten translate}
\label{sec:artinalgs}

Given a self-injective Artin algebra $\Lambda$ over $K$,
the endofunctor
\begin{displaymath}
  M \mapsto \Hom_K(\Hom_\Lambda(M,\Lambda),K) 
\end{displaymath}
of the category of left $\Lambda$ modules, is termed the Nakayama functor by
Auslander, Reiten and Smal{\o} \cite[p. 126]{ARS}, and they denote it by 
$\Na$. More generally, if $\Lambda$ is not self-injective we choose
an injective bi-module resolution $\Lambda \to J$ of $\Lambda$ and term the
endofunctor
\begin{displaymath}
  M \mapsto \Hom_K(\Hom_\Lambda(M,J),K) 
\end{displaymath}
the Nakayama functor.

In the situation $\Lambda = \Lambda_{\mbt}$ we have
an isomorphism
\begin{displaymath}
  \Hom_K(\Hom_{\Lambda_{\mbt}}(M,J),K) \cong
  \tau^* \Hom_K(\tau^*\Hom_{\Lambda_{\mbt}}(M,J),K).
\end{displaymath}
Thus Proposition \ref{convertnak} implies that the Nakayama functor in the
context of  positively $\mbt$-determined $S$-modules corresponds to the
Nakayama functor for the incidence algebra $\Lambda_{\mbt}$.
On the bounded 
derived categoriy of $\Lambda_{\mbt}$-modules the Nakayama functor 
represents the Auslander--Reiten translate \cite[p.37]{Ha}, since
$\Lambda_{\mbt}$ has finite global dimension.

% Our last description of the Nakayama functor is not significantly
% different from the previous one. Every finite poset $P$ has an associated
% incidence $K$-algebra $\Lambda_{\mbt}$, see e.g. \cite{ARS}, and the category
% of
% modules over $\Lambda = \Lambda_{\mbt}$ is isomorphic to the category of
% $KP$-modules. (See e.g \cite[Remark 6.5]{BBR}.) 
% Under this isomorphisms of categories the endofunctor $M \mapsto
% \tau^*\Hom_{KP}(M,KP)$ corresponds to the endofunctor
% $M
% \mapsto \tau^* M^* = \tau^* \Hom_{\Lambda}(M,\Lambda)$ on the
% category of $\Lambda$-modules, where
% $\varphi \colon \Lambda \to \Lambda^{\op}$ is the isomorphism
% induced by the order-preserving bijection $\varphi \colon [0,l] \to
% [0,l]^{\op}$.
% Moreover the endofunctor
% $M \mapsto
% \varphi^*\Hom_{K}(M,K)$ on the category of $KP$-modules corresponds to
% the endofunctor 
% $M
% \mapsto \varphi^* M^* = \varphi^* \Hom_{K}(M,K)$ on the
% category of $\Lambda$-modules.
% This tells us that the translated Nakayama functor $T^n\Na^{\one}_{\mbt}$
% corresponds to the composition 
% \begin{displaymath}
%   M \mapsto \Hom_K(\Hom_\Lambda(M,\Lambda),K) = \varphi^*
%   \Hom_K(\varphi^* \Hom_\Lambda(M,\Lambda),K) 
% \end{displaymath}

\section{Nakayama cohomology of monomial ideals}
\label{sec:mainresults}
In this section we present the main results of this paper. Firstly, we
calculate the multigraded cohomology and Betti spaces of
iterations of the Nakayama functor applied to monomial
quotient rings. The results of these calculations are expressed in terms of
cohomology of certain simplicial complexes. Secondly, we describe the
$S$-module structure of the cohomology modules of the Nakayama functor
applied to monomial quotient rings. In order to specify the simplicial
complexes occurring in our calculations we need some notation. 

\subsection{Simplicial complexes from monomial ideals}
\label{sec:cxfromid}

Let $I$ be a positively $\mbt$-determined monomial ideal in $S$, and let 
$\mbb \le \mbt$ and 
$\mba \le \mbb + \one$ in $\N^n$. The simplicial complex
$\Delta_{\mba}^{\mbb}(S/I;\mbt)$ on the vertex set $\{1,\dots,n\}$
is defined to consist of the subsets $F$ of the
set $\{1,\dots,n\}$ with the property that the degree 
$$\mbx = \sum_{i \notin F} a_i
\varepsilon_i + \sum _{i \in F} (b_{i} +1) \varepsilon_i$$
satisfies
both $\mbx \le \mbt$ and $(S/I)_{\mbx} = K$. 

\begin{example}
\label{linkex}
  Let $I$ be a square free monomial ideal in $S$ and let $\mba$ be the
  indicator vector of a subset $A$ of $\{1,\dots,n\}$. The simplicial
  complex $\Delta = \Delta_{0}^{0}(S/I;\one)$ is the simplicial
  complex with Stanley-Reisner ring equal to $S/I$. The simplicial
  complex $\Delta^{\mba}_{\mba}(S/I;\one)$ is the 
  link $\link_{\Delta}(A)$ of $A$ in $\Delta$, that is,
  the simplicial complex consisting 
  of the faces $F$ in $\Delta$ satisfying that $F \cap A =
  \emptyset$ and $F \cup A \in \Delta$. Finally $\Delta_0^{\one -
    \mba}(S/I;\one)$ is the restriction $\Delta|_A$ consisting of the faces $F$
  in $\Delta$ contained in $A$.
\end{example}
If $\mbt$ and the positively $\mbt$-determined ideal $I$ are given explicitly  we shall 
also denote the
simplicial complex $\Delta_{\mba}^{\mbb}(S/I;\mbt)$ simply as
$\Delta_{\mba}^{\mbb}$. The following lemma records a combinatorial
relation between the simplicial complexes $\Delta_{\mba}^{\mbb}$ which
is crucial for our calculations.
\begin{lem}
\label{pbsituation}
Let $I$ be a positively $\mbt$-determined monomial ideal in $S$, and 
$\mbb \leq \mbt$ and $\mba \le \mbb + \one$  in $\N^n$.
Given $\alpha, \beta \in \N$ 
  with $a_j + \alpha \le b_j + 1$. Then
  \begin{displaymath}
    \Delta^{\mbb}_{\mba} =
    \Delta^{\mbb+\beta\varepsilon_j}_{\mba} \cup
    \Delta_{\mba+\alpha
      \varepsilon_j}^{\mbb}
  \end{displaymath}
  and
  \begin{displaymath}
    \Delta_{\mba+\alpha \varepsilon_j}^{\mbb+\beta\varepsilon_j} =
    \Delta_{\mba}^{\mbb+\beta\varepsilon_j} \cap
    \Delta_{\mba+\alpha
      \varepsilon_j}^{\mbb}. 
  \end{displaymath}  
\end{lem}
\begin{proof}
  This is seen by splitting into the cases whether a subset $F$ of $\{1,2, \ldots, n \}$
contains $j$ or not.
\end{proof}
Given a simplicial complex $\Delta$ we
denote by $\widetilde H^*(\Delta)$ its reduced cohomology with coefficients in $K$.
The above lemma implies that there is a Mayer--Vietoris exact sequence
of reduced cohomology groups. 
\begin{cor}
\label{mayervirtoriscor}
  Let $I$ be a positively $\mbt$-determined monomial ideal in $S$, and let 
$\mbb \leq \mbt$ and $\mba \le \mbb+\one$ 
in $\N^n$.
  Given $\alpha, \beta \in \N$ 
  with $a_j + \alpha \le b_j+1$. Then there is a Mayer--Vietoris exact sequence of the form
  \begin{eqnarray*}
   \dots  
\to  
   \widetilde H^i 
   (\Delta^{\mbb}_{\mba}) \to 
   \widetilde
    H^i(\Delta_{\mba}^{\mbb+\beta\varepsilon_j}) \oplus
    \widetilde H^i(\Delta_{\mba+\alpha
      \varepsilon_j}^{\mbb}) \to
 \widetilde H^i (\Delta_{\mba+\alpha
      \varepsilon_j}^{\mbb+\beta\varepsilon_j})
   \xrightarrow {\delta^i} 
   \widetilde H^{i+1} (\Delta
   ^{\mbb}_{\mba}) \to 
\dots 
  \end{eqnarray*}
\end{cor}

\subsection{Nakayama cohomology}
\label{sec:nakcho}

We are now going to compute the cohomology of 
$\Na^{\mbk}_{\mbt}(S/I)$ for every  positively $\mbt$-determined
ideal $I$ in  
$S$. 
In Example \ref{firstnakone} we have treated the case $n=1$. 
Note from this that for every 
interval $[a,b]$ contained in $[0,t]$ there are easily computable
$\gamma$, $x$ and $y$ such that $\Na^k_t(K_t\{a, b\})$ is quasi-isomorphic to
$T^{-\gamma}K_t\{x,y\}$. For $n > 1$ we have the isomorphism
  \begin{displaymath}
    \Na^{\mbk}_{\mbt} (K_{\mbt}\{ \mba,\mbb\}) \cong
    \bigotimes_{j=1}^n \Na^{k_j}_{t_j}(K_{t_j}\{a_j,b_j\})
  \end{displaymath}
where the outer tensor product is over $K$. 
Therefore we can easily compute $\gamma$, $\mbx$ and $\mby$ such that
$\Na_{\mbt}^{\mbk} (K_{\mbt}\{\mba,\mbb\})$ is quasi-isomorphic to
$T^{-\gamma}(K_{\mbt}\{\mbx, \mby\})$.

\begin{thm}
\label{thecalc1}
\label{thecalc3}
  Let $\mbt,\mbk \in \N^n$ and let $\mbr \in [0,\mbt]$. 
  Moreover let $\mbu$, $\mbv$ in $\N^n$ and
  $\gamma \in \Z$ be such that the only nonzero cohomology group
  \begin{displaymath}
    H^{\gamma + n} \Na_{\mbt}^{\mbk+ \one} K_{\mbt}\{\mbt-\mbr, \mbt\} \cong 
    K_{\mbt}\{\mbt -\mbv, \mbt -
  \mbu\}.
  \end{displaymath}
There is an isomorphism
$$
    H^i \Na^{\mbk}_{\mbt} (S/I)_{\mbr} \cong \widetilde
    H^{i-\gamma-1}(\Delta^{\mbv}_{\mbu}(S/I;\mbt))
$$
for every  positively $\mbt$-determined ideal $I$ in  $S$.
\end{thm}

\begin{remark} \label{RemMainPro1}
Note that $\Na_{\mbt}^{\mbk+ \one} K_{\mbt}\{\mbt-\mbr, \mbt\}$
is quasi-isomorphic to $T^{-n}\Na_{\mbt}^{\mbk} K_{\mbt}\{\mathbf 0, \mbt-\mbr \}$
and to determine $\mbu, \mbv$ and $\gamma$ we shall use the latter. We have
stated the theorem as above to emphasize the analogy with Theorem \ref{bettithm}
below.
\end{remark}

\medskip
Let $I$ be a positively $\mbt$-determined ideal in $S$. 
The following theorem
describes the Betti-spaces $B^p (\Na^{\mbk}_{\mbt} (S/I))$, which may be 
computed as $H^p(F\otimes_S \Na^{\mbk}_{\mbt} (S/I))$ for $F$ a projective $\N^n$-graded
resolution of $K$ over $S$.

% The discussion above Theorem \ref{thecalc1} also shows that we easily can
% easily find $\gamma$, $\mba$ and $\mbb$ such that
% $\Na_{\mbt}^{\mbk+\one} K_{\mbt}\{\mbt-\mbr, \mbt-\mbr\}$ is quasi-isomorphic to
% $T^{-\gamma}K_{\mbt}\{\mbt -\mba, \mbt -
%   \mbb\}$.

\begin{thm}
\label{bettithm}
  Let $\mbt,\mbk \in \N^n$ and let $\mbr \in [0,\mbt]$. 
  Moreover let $\mba$, $\mbb$ in $\N^n$ and
  $\gamma \in \Z$ be such that the only nonzero cohomology group
\begin{displaymath}
  H^{\gamma +n} \Na^{\mbk+\one}_{\mbt}
  \kr_{\mbt}\{\mbt-\mbr,\mbt-\mbr\} \cong
  \kr_{\mbt} \{\mbt-\mbb, \mbt - \mba\}.
\end{displaymath}
There is an isomorphism
\[ B^i (\Na^{\mbk}_{\mbt} (S/I))_{\mbr} \cong \widetilde
H^{i-\gamma-1}(\Delta_{\mba}^{\mbb}(S/I;\mbt)) \]
  for every  positively $\mbt$-determined ideal
  $I$ in
  $S$.
\end{thm}

It is easy to 
give an explicit description of $\mbu$, $\mbv$ and
$\gamma$ in Theorem \ref{thecalc1}. In order to do this we first note that
in view of 
Proposition \ref{trans} it suffices to consider the case where $\mbk
\le \mbt + \one$.
%$k_i \le
%t_i+1$ for all $i$. 
\begin{defn}
  \label{gammauv}
Given $r,k,t \in \N$ with $k \le t+1$ we define
$\gamma^k_t(r)$, $u^k_t(r)$ and $v^k_t(r)$ by the formulas
\begin{displaymath}
  \gamma^k_t(r) =
  \begin{cases}
    0 & \text{if $k \le r$} \\
    1 & \text{if $r + 1 \le k$},
  \end{cases} 
\end{displaymath}
\begin{displaymath}
  u^k_t(r) =
  \begin{cases}
    r - k  & \text{if $k \le r$} \\
    t - k+1 & \text{if $r + 1 \le k$}
  \end{cases}
\end{displaymath}
and
\begin{displaymath}
  v^k_t(r) =
  \begin{cases}
    t - k  & \text{if $k \le r$} \\
    t-k+r + 1 & \text{if $r + 1 \le k$}.
  \end{cases}
\end{displaymath}
\end{defn}
 
Let $\gamma_\mbt^\mbk(\mbr) = \sum_{j=1}^n
\gamma_{t_j}^{k_j}(r_j)$, let $\mbu_\mbt^\mbk(\mbr) 
= (u_{t_1}^{k_1}(r_1),\dots,u_{t_n}^{k_n}(r_n))$ ,
and let $\mbv_\mbt^\mbk(\mbr) =
(v_{t_1}^{k_1}(r_1),\dots,v_{t_n}^{k_n}(r_n))$. In Section \ref{sec:genint} we
prove the following.
\begin{prop}
\label{nakcoh}
The elements $\mbu = \mbu_\mbt^\mbk(\mbr)$ and $\mbv =
\mbv_\mbt^\mbk(\mbr)$ in $\N^n$ and the
integer $\gamma = \gamma_\mbt^\mbk(\mbr)$ satisfy the first equation in 
Theorem \ref{thecalc1}.   
\end{prop}
% \begin{thm}
%   
%   Let $\mbt,\mbk \in \N^n$ with $k_i \le t_i +1$ for all $i$. Let $I$
%   be a  $((\mbt+\one) p(\mbk)-\one)$-determined ideal in 
%   $S$ and
%   let $\mbr \in  
%   [0,\mbt]$.
%   With the notation $\mbu = \mbu_\mbt^\mbk(\mbr)$, $\mbv =
%   \mbv_\mbt^\mbk(\mbr) $ and $\gamma = \gamma_\mbt^\mbk(\mbr)$ there is an
%   isomorphism of the form 
%   \begin{displaymath}
%     H^i \Na^{\mbk}_{\mbt} (S/I)_{\mbr} \cong \widetilde
%     H^{i-\gamma-1}(\Delta^{\mbv}_{\mbu}(I)). 
%   \end{displaymath}
% \end{thm}

In the particular case $\mbk = \one$ we recover Takayama's calculation of local
cohomology \cite{Takayama2004}.
\begin{cor}[Takayama]
  Let $\mbt \in \N^n$ and $\mbr \in [0,\mbt]$.
  With the notation $\mbu = \mbu_\mbt^\one(\mbr)$, $\mbv =
  \mbv_\mbt^\one(\mbr)$ and $\gamma = \gamma_\mbt^\one(\mbr)$ there are
  isomorphisms of the form 
  \begin{displaymath}
    H^i_\mm(S/I)_{\mbr-\one} \cong 
    H^i \Na^{\one}_{\mbt} (S/I)_{\mbr} \cong \widetilde
    H^{i-\gamma-1}(\Delta^{\mbv}_{\mbu}(S/I;\mbt)). 
  \end{displaymath}
for every positively $\mbt$-determined ideal $I$ in $S$.
\end{cor}
\begin{proof}
  The first isomorphism is a direct consequence of Proposition
  \ref{extstab}. The second isomorphism is just a special case 
  of Proposition \ref{nakcoh} and Theorem \ref{thecalc1}.
\end{proof}
The above corollary on the other hand is an extension of Hochster's
calculation  \cite{Hochster1977} of local cohomology of square-free
monomial ideals, or in 
other words, of $\one$-determined ideals in $S$.

We now proceed to describe the $S$-module structure of the Nakayama
cohomology groups $H^i \Na^{\mbk}_{\mbt} (S/I)$ of a
positively $\mbt$-determined ideal $I$ in $S$. Since this $S$-module is  positively $\mbt$-determined it
suffices to describe the multiplication maps 
$$x_j \colon H^i \Na^{\mbk}_{\mbt}
(S/I)_{\mbr - \varepsilon_j} \to H^i \Na^{\mbk}_{\mbt}
(S/I)_{\mbr}$$
whenever $\varepsilon_j \le \mbr \le \mbt$.

\begin{thm}
\label{nakcohS}
  Let $I$ be a  positively $\mbt$-determined ideal in $S$ and let $\mbr
  \in \N^n$ with $\varepsilon_j \le \mbr \le \mbt$. Using
  $\Delta^{\mbb}_{\mba}$ as short hand notation for
  $\Delta^{\mbb}_{\mba}(S/I;\mbt)$ the 
  multiplication map
  $$x_j \colon H^i \Na^{\mbk}_{\mbt} (S/I)_{\mbr - \varepsilon_j} \to H^i \Na^{\mbk}_{\mbt}
  (S/I)_{\mbr}$$ 
  can be described as follows: 
  \begin{enumerate}
  \item If $k_j \ne r_j$, then
    $\Delta^{\mbv_\mbt^\mbk(\mbr)}_{\mbu_\mbt^\mbk(\mbr)}$
    is a subcomplex of 
    $\Delta^{\mbv_\mbt^\mbk(\mbr-\varepsilon_j)}_{\mbu_\mbt^\mbk(\mbr-\varepsilon_j)}$ and 
    multiplication by $x_j$ corresponds to
    the map of cohomology groups induced by this inclusion.
  \item If $k_j = r_j$, then multiplication by $x_j$ corresponds to
    $(-1)^{\sum_{i=1}^{j-1}\gamma_i(r_i)}$ times the
    Mayer-Vietoris connecting homomorphism $\delta^{i-\gamma-2}$ of Corollary
    \ref{mayervirtoriscor} with $\mba =
    \mbu_\mbt^\mbk(\mbr)$, $\mbb = \mbv_\mbt^\mbk(\mbr)$, 
    $\gamma = \gamma_\mbt^\mbk(\mbr)$, $\alpha =
    t_j - r_j +1$ and $\beta = r_j$.
  \end{enumerate}
\end{thm}

The case $\mbk=\mbt=\one$ of the above theorem is due to Gr\"abe
\cite{Graebe1984} as we shall now explain. In this
case $I$ is $\one$-determined, and thus $S/I$ is the Stanley-Reisner ring
of the simplicial complex 
$\Delta = \Delta(S/I)$. 
Note that in this case $\varepsilon_j \le \mbr \le \mbt = \one$  implies $r_j
= 1 = k_j$, 
and thus the $S$-module structure of $H^i \Na^\one_\one (S/I)$
% \cong p_{\one}^*(H^i_\mm(S/I)(-\one))$ 
is given by Mayer-Vietoris boundary maps. Further $\mbv_\one^\one(\mbr) = 
\mbu_\one^\one(\mbr)  = 
\one - \mbr$ (check!) for every $\mbr \in
[0,\one]$. If we let  $\overline R$ denote the subset $\overline R = \{j
\colon r_j = 0\}$  of $\{1,\dots,n\}$ then we saw in Example
\ref{linkex} that 
$\Delta^{\mbv_\one^\one(\mbr)}_{\mbu_\one^\one(\mbr)}(S/I;\one) $,
which is $\Delta^{\one-\mbr}_{\one
  -\mbr}(S/I;\one)$, 
is the link 
$\link_{\Delta}(\overline R)$ of $\overline R$ in $\Delta$.
Gr\"abe considers the
contrastar $\cost_{\Delta} (\overline R)$ of $\overline R$ in $\Delta$ given as
\begin{displaymath}
  \cost_{\Delta}(\overline R) := \{ G \in \Delta \colon \overline R \not \subseteq G \},
\end{displaymath}
and he notes that there is an isomorphism
\begin{eqnarray*}
  T^{|\overline R|}\widetilde C_*(\link_{\Delta} (\overline R)) &\cong& \widetilde
  C_*(\Delta)/\widetilde C_*(\cost_{\Delta}(\overline R)), \\
  G & \mapsto&
  (-1)^{a(\overline R,G)} (\overline R \cup G) + \widetilde
  C_*(\cost_{\Delta}(\overline R)),
\end{eqnarray*}
where $a(\overline R,G)$ is the cardinality of  
$\{(f,g) \in \overline R \times G \colon f < g \}$. 
%Let us return to the assumption that $\varepsilon_j \le \mbr$.
%  and let
% $\overline S = \overline R \cup \{j\}$.
On the
level of cohomology groups we obtain a commutative diagram of the form
\begin{displaymath}
  \begin{CD}
    \widetilde H^{i-|\overline R|-1}(\link_{\Delta}(\overline R \cup \{j\}))
    @>{\cong}>> \widetilde H^i(\Delta,\cost_{\Delta}(\overline R \cup \{j\})) \\
     @V{(-1)^{a(\overline R,\{j\})} \delta}VV @VVV \\
    \widetilde H^{i-|\overline R|}(\link_{\Delta}(\overline R))
    @>{\cong}>> \widetilde H^i(\Delta,\cost_{\Delta}(\overline R)).
  \end{CD}
\end{displaymath}
Here $\delta$ is the Mayer-Vietoris connecting homomorphism of
Corollary \ref{mayervirtoriscor} 
\[\widetilde H^{i-|\overline R|-1}(\Delta^{\one - \mbr +\varepsilon_j}_{\one - \mbr +\varepsilon_j}
(S/I;\one)) \rightarrow 
\widetilde H^{i-|\overline R|}(\Delta^{\one -\mbr}_{\one - \mbr}(S/I;\one) ) \]
(in the situation $\mba  = \mbb = \one-\mbr$ and 
$\alpha = \beta = 1$).
The right hand vertical
homomorphism is induced from the inclusion $\cost_{\Delta}(\overline R)
\subseteq 
\cost_{\Delta}(\overline R \cup \{j\})$.
Since $H^i \Na_\one^\one (S/I)$ is isomorphic to $p_\one^* (H^i_\mm(S/I)(-\one))$
we recover Gr\"abe's result \cite{Graebe1984} in the
following formulation.
\begin{thm}[Gr\"abe, 1984]
  Let $I \subseteq S$ be a square-free monomial ideal and let $\Delta =
  \Delta(S/I)$. For every $\mbr \in
  \{0,1\}^n$ there is, with the notation $\overline R = \{j \colon r_j =
  0\}$, an isomorphism
  \begin{displaymath}
    H^i_\mm(S/I)_{\mbr-\one} \cong
    \widetilde H^{i-1}(\Delta,\cost_{\Delta}(\overline R)).
  \end{displaymath}
  If $\varepsilon_j \le \mbr \le \one$, then under the above
  isomorphism the multiplication
  \begin{displaymath}
    x_j \colon H^i_\mm(S/I)_{\mbr-\varepsilon_j-\one} \to H^i_\mm(S/I)_{\mbr-\one}
  \end{displaymath}
  corresponds to the homomorphism of cohomology groups induced by the
  inclusion $\cost_{\Delta} (\overline R) \subseteq \cost_{\Delta}
  (\overline R \cup \{j\})$.
\end{thm}

\section{Vanishing of Nakayama cohomology}
\label{sec:van}

In this section we give conditions for the vanishing or non-vanishing
of the cohomology groups of the Nakayama functor $\Na^{\mbk}_{\mbt}(S/I)$.
Towards the end we discuss the simple case when the
polynomial ring $S$ is in two variables, and give the conditions for 
when the Nakayama functor has exactly one non-vanishing cohomology group.
The ideal $I$ is assumed to be positively $\mbt$-determined.

\subsection{General results}

\begin{lem}
\label{boundsonnacoh}
  The Nakayama cohomology
groups $H^i \Na^{\mbk}_{\mbt} (S/I)$ vanish for $i \ge 2n$ and for $i<0$.
\end{lem}

\begin{proof}
Theorem \ref{thecalc1} gives the isomorphism
$$
    H^i \Na^{\mbk}_{\mbt} (S/I)_{\mbr} \cong \widetilde
    H^{i-\gamma-1}(\Delta^{\mbv}_{\mbu}(S/I;\mbt)).
$$
where $\gamma = \gamma_{\mbt}^{\mbk}(\mbr)$, $\mbu =
\mbu_{\mbt}^{\mbk}(\mbr)$  and $\mbv = \mbv_{\mbt}^{\mbk}(\mbr)$ are
specified in Proposition \ref{nakcoh}. In particular note that  $0 \leq \gamma \leq n$.
Since  
$\Delta^{\mbv}_{\mbu}$ is a simplicial complex on $n$ vertices, 
the reduced homology group $\widetilde
H^j(\Delta^{\mbv}_{\mbu})$ can be nonvanishing only in the range $-1 \leq j \leq n-2$.
Hence the right side above can only be nonvanishing in the range
$-1 \leq i-\gamma - 1 \leq n-2$. 
But since $0 \leq \gamma \leq n$ this gives $0 \leq i \leq 2n-1$.
%Hence if $i-\gamma - 1 \ge n-1$, that is, $i \ge \gamma + n$ implies 
%the vanishing of the left side in the equation above.
%Proposition \ref{nakcoh}  in particular states that $0 \le
%\gamma \le n$, and thus $H^i \Na^{\mbk}_{\mbt}
%(S/I) = 0$ for $i \ge 2n$ and for $i < 0$.
\end{proof}

\begin{defn}
\label{emptybox}
Let $I$ be a positively $\mbt$-determined ideal and $\mby \in [{\bf 0}, \mbt]$.
\begin{itemize}
\item $S/I$ has a {\it peak} at $\mby$ if $(S/I)_{\mby} \neq 0$ and
$(S/I)_{\mbx} = 0$ for all $ \mby < \mbx \leq \mbt$. 
\item $S/I$ has an 
{\it indent} at $\mby$ if $(S/I)_{\mby} = 0$ and
$(S/I)_{\mbx} \neq 0$ for all ${\bf 0} \leq \mbx < \mby$.
\item A peak $\mby$ is said to be of relative dimension $m$ below $\mbx$ if 
$\mby \leq \mbx$
and the number of non-zero coordinates of $\mbx - \mby$ is $m$.
An indent $\mby$ is said to be of relative dimension $m$ above $\mbx$
if $\mby \geq \mbx$ and the number of nonzero coordinates of $\mby - \mbx$ 
is $m$.
\end{itemize}
\end{defn}

The following gives sufficient criteria for the nonvanishing of
cohomology groups. They are however by no means necessary. The vanishing of 
cohomology groups is of course more interesting and the following lemma is therefore
a step towards Proposition \ref{sharpzero}.

\begin{lem}
\label{nonvanish}
Let $I$ be a positively $\mbt$-determined ideal in $S$ and $\one \leq \mbk \leq \mbt + \one$.

a. If $S/I$ has a peak of relative dimension $m$ below $\mbt + \one - \mbk$,
then $H^{n-m}\Na_{\mbt}^{\mbk}(S/I)$ is non-zero.

b. If $S/I$ has an indent of relative dimension $m$ above $\mbt + \one - \mbk$,
then $H^{n-1+m}\Na_{\mbt}^{\mbk}(S/I)$ is non-zero.
\end{lem}
\begin{proof}
a. Let $\mbu$ be the peak. Define $\mbv \in [0,\mbt]$ by letting 
\begin{displaymath}
  v_j =
  \begin{cases}
    t_j - k_j & \text{if $t_j - k_j + 1 > u_j$} \\
    t_j & \text{if $t_j - k_j + 1 = u_j$} .
  \end{cases}
\end{displaymath}
Moreover define $\mbr \in [0,\mbt]$ by letting
\begin{displaymath}
  r_j =
  \begin{cases}
    u_j + k_j & \text{if $t_j - k_j + 1 > u_j$} \\
    k_j -1 & \text{if $t_j - k_j + 1 = u_j$}.
  \end{cases}
\end{displaymath}
Then
$\gamma_{\mbt}^{\mbk}(\mbr) = n-m$, $\mbu =
\mbu_{\mbt}^{\mbk}(\mbr)$  and $\mbv = \mbv_{\mbt}^{\mbk}(\mbr)$. Thus
Theorem \ref{thecalc1} and Proposition \ref{nakcoh} imply that
\begin{displaymath}
  H^{n-m} \Na_{\mbt}^{\mbk}(S/I)_\mbr \cong \widetilde
    H^{-1}(\Delta_{\mbu}^{\mbv}(S/I)).
\end{displaymath}
Since $\mbu$ is a peak of relative dimension $m$ below $\mbt +\one- \mbk$,
the simplicial complex $\Delta_{\mbu}^{\mbv}(S/I)$ contains only the empty
subset of $\{1,\dots,n\}$, and therefore the right side above is isomorphic to $K$,
and so $H^{n-m} \Na_{\mbt}^{\mbk}(S/I)$ is nonvanishing.

\medskip
b. Let $\mby$ be the indent and set $\mbv = \mby - \one$. 
Define $\mbu \in [0,\mbt]$ by letting
\begin{displaymath}
  u_j =
  \begin{cases}
    0 & \text{if $v_j = t_j - k_j$} \\
    t_j - k_j + 1 & \text{if $v_j > t_j - k_j$}. 
  \end{cases}
\end{displaymath}
Moreover define $\mbr \in [0,\mbt]$ by letting
\begin{displaymath}
  r_j =
  \begin{cases}
    k_j & \text{if $v_j = t_j - k_j$} \\
    v_j - t_j + k_j - 1 & \text{if $v_j > t_j - k_j$}. 
  \end{cases}
\end{displaymath}
Then
$\gamma_{\mbt}^{\mbk}(\mbr) = m$, $\mbu =
\mbu_{\mbt}^{\mbk}(\mbr)$  and $\mbv = \mbv_{\mbt}^{\mbk}(\mbr)$. Thus
Theorem \ref{thecalc1} and Proposition \ref{nakcoh} imply that
\begin{displaymath}
  H^{n-1+m} \Na_{\mbt}^{\mbk}(S/I)_\mbr \cong \widetilde
    H^{n-2}(\Delta_{\mbu}^{\mbv}(S/I)).
\end{displaymath}
Since $\mbv$ is an indent  of relative dimension $m$ above  
$\mbt + \one - \mbk$, the
simplicial complex $\Delta_{\mbu}^{\mbv}(S/I)$ is the
simplicial $(n-2)$-sphere consisting of all proper subsets of $\{1,\dots,n\}$.
Then the right side above is isomorphic to $K$ and so $H^{n-1+m} \Na_{\mbt}^{\mbk}(S/I)$
is nonvanishing.
\end{proof}

\begin{cor}
\label{loknonvanish} 
   If $I$ is a positively $\mbt$-determined ideal in $S$ and $S/I$ has a peak of 
relative dimension $m$ below $\mbt$, then $H^{n-m}_{\mm}(S/I) \ne 0$. 
\end{cor}

%\begin{proof}
%  I want to refer to a (as yet non-existing) proposition in Section
%  2.3. 
%\end{proof}

In the extremal cases $m = n$
we can characterize exactly when the Nakayama cohomology vanishes.
\begin{prop}
\label{sharpzero}
Let $I$ be a positively $\mbt$-determined ideal and $\one \leq \mbk \leq \mbt + \one$.

a.  The zero'th Nakayama cohomology group $H^0 \Na^{\mbk}_{\mbt} (S/I)$ 
vanishes if and only if no peak of $S/I$ is contained in 
$[{\bf 0}, \mbt - \mbk]$.

b.  The $(2n-1)$-th Nakayama cohomology group 
$H^{2n-1} \Na^{\mbk}_{\mbt} (S/I)$ vanishes 
if and only if no indent of $S/I$ is greater or equal to 
$\mbt + 2 \cdot \one  - \mbk$, save an exception when $n=1$ and
$I = 0$.
\end{prop}

%\begin{remark} \label{VanRemInt} Condition a. is equivalent to the set of multidegrees in 
%$[{\bf 0}, \mbt + \one - \mbk]$ where $S/I$ is zero, being
%exactly an interval $[{\mathbf a}, \mbt + \one - \mbk]$ where 
%$\mba \leq \mbt - \mbk$, or it being contained in an interval 
%$[{\mathbf a}, \mbt + \one - \mbk]$ where some $a_i = t_i + 1 - k_i$.

%Condition b. is equivalent to the set of multidegrees greater or 
%equal to $\mbt + \one - \mbk$ where $S/I$ is nonzero, being
%exactly an interval $[\mbt + \one - \mbk, {\mathbf b}]$ where
%${\mathbf b} \geq \mbt + 2\cdot \one  - \mbk$, or it being contained
%in an interval $[\mbt + \one - \mbk, {\mathbf b}]$ where some
%$b_i = t_i + 1 - k_i$.
%\end{remark}

\begin{proof}
a. Suppose the zero'th Nakayama cohomology group does not vanish
in degree $\mbr$.
Since Theorem \ref{thecalc1} gives
$$
    H^0 \Na^{\mbk}_{\mbt} (S/I)_{\mbr} \cong \widetilde
    H^{-\gamma-1}(\Delta^{\mbv}_{\mbu}(S/I)),
$$
we must have $\gamma = 0$ and the simplicial complex on the right side must be 
$\{\emptyset\}$.
Moreover Proposition \ref{nakcoh}
implies that $\mbk \le \mbr$, that $\mbu = \mbr - \mbk$ and
that $\mbv = \mbt - \mbk$. Now the fact that $\Delta^{\mbv}_{\mbu}(S/I)$
is $\{\emptyset\}$ implies that $S/I$ has a peak in the interval $[\mbu, \mbv]$ 
which is of relative dimension $n$ below $\mbt + \one- \mbk$.

Conversely, if $S/I$ has a peak of relative dimension $n$
below $\mbt + \one - \mbk$, then Proposition \ref{nonvanish}a. implies 
the non-vanishing of the zeroth Nakayama cohomology group.

\medskip
b. Suppose the $2n-1$'th Nakayama cohomology group is non-vanishing in
degree $\mbr$.
Since Theorem \ref{thecalc1} gives
$$
    H^{2n-1} \Na^{\mbk}_{\mbt} (S/I)_{\mbr} \cong \widetilde
    H^{2n-1-\gamma-1}(\Delta^{\mbv}_{\mbu}(S/I)),
$$
and since $\Delta^{\mbv}_{\mbu}(S/I)$ is a simplicial complex on the vertex
set $\{1,\dots, n\}$ we have $2n-1-\gamma-1 \le n-2$, that is, $n \le
\gamma$. By Proposition \ref{nakcoh}, $\gamma$ can be at most $n$ and so $\gamma = n$. 
By the same Proposition \ref{nakcoh} we get 
$\mbu = \mbt - \mbk + \one$ and
that $\mbv = \mbt - \mbk + \mbr + \one$. 
Therefore $\widetilde H^{n-2}(\Delta^{\mbv}_{\mbu}(S/I))$ is nonvanishing, and consequently
$\Delta^{\mbv}_{\mbu}(S/I)$ is the simplicial sphere consisting of all
proper subsets of $\{1,\dots,n\}$. 
If $n > 1$ we conclude that $\mbv + \one \leq \mbt$. If $n=1$ we can also conclude
that $v + 1 \leq t$ provided $r \leq k-2$. Assume then that $\mbv + \one \leq \mbt$.
Consider the elements $\mbx$ in $[\mbu, \mbv + \one]$ such that 
$(S/I)_\mbx$ is nonzero. Note that i) $(S/I)_{\mbv + \one}$ is zero, and 
ii) if $x_i = u_i$ for some $i$, then $(S/I)_\mbx$ is nonzero. Therefore a 
minimal
element in  $[\mbu, \mbv + \one]$ such that $(S/I)_\mbx$ is zero must exist and
be in the interval $[\mbu + \one, \mbv + \one]$. As such it is an indent of 
relative dimension $n$ above $\mbu$. 

The only case left is when $n = 1$ and we must have $r \geq k-1$. 
Then we see that  $(S/I)_t$ must be nonzero, and so $I = 0$. 
In this case we see that the above cohomology group is
always nonzero in degree $r = k-1$. 

Conversely, if $S/I$ has indent, 
Proposition \ref{nonvanish}b. implies the non-vanishing of the $2n-1$'th 
Nakayama cohomology group.
\end{proof}

\subsection{The case of two variables}
We now assume that $S = K[x, y]$, the polynomial ring in two variables. In this case
there can only be four possible nonvanishing cohomology modules of 
$\Na^{\mbk}_{\mbt}(S/I)$ and these
are in degrees $0,1,2,$ and $3$. We assume in this subsection that 
$\one \le \mbk \le \mbt + \one$.

\begin{lem} Let $I$ be a positively $\mbt$-determined ideal.

a. The zero'th cohomology group $H^0 \Na^{\mbk}_{\mbt}(S/I)$ vanishes if and only if the 
support of $S/I$ in 
$[\mathbf 0, \mbt + \one - \mbk]$ is the complement of a (possibly empty) interval.

b. Its third cohomology group $H^3 \Na^{\mbk}_{\mbt}(S/I)$ vanishes if and only if the support of $S/I$ in 
$[\mbt + \one - \mbk, \mbt]$
is a (possibly empty) interval.
\end{lem}

\begin{proof} This is just a translation of Proposition \ref{sharpzero}
when we take into consideration that the intervals are in the plane.
\end{proof}

\begin{lem} The second cohomology group $H^2 \Na^{\mbk}_{\mbt}(S/I)$ 
vanishes if and only if $S/I$ is zero in degree $\mbt + \one - \mbk$.
\end{lem}

\begin{proof} Suppose $(S/I)_{\mbt + \one -\mbk}$ is nonzero. Letting $\mbr = \mbk - \one$, 
Proposition \ref{nakcoh}
gives $\gamma = 2$, $\mbu = \mbt - \mbk + \one$, and $\mbv = \mbt$. Then 
$\Delta_\mbu^\mbv(S/I)$ is $\{\emptyset \}$ and so its $\widetilde H^{-1}$ cohomology is 
$K$. By Theorem \ref{thecalc1}, the second cohomology group in the lemma statement does not
 vanish.

Conversely, suppose $H^2 \Na^{\mbk}_{\mbt}(S/I)$ does not vanish in degree $\mbr$. 
By Theorem \ref{thecalc1},
$\widetilde H^{1- \gamma}(\Delta_\mbu^\mbv (S/I))$ is nonvanishing. This is the reduced 
cohomology 
group of a simplicial complex on two vertices, and so either i) $1-\gamma = -1$, 
giving $\gamma = 2$,
or ii) $1-\gamma = 0$ giving $\gamma = 1$. 

\noindent i) Then $\mbu$ is $\mbt + \one - \mbk$
and $\Delta_\mbu^\mbv (S/I)$ is $\{\emptyset \}$, implying $S/I$ nonzero in degree $\mbu$.

\noindent ii) Then $\mbu$ is $(r_1 - k_1, t_2 + 1 - k_2)$ (or the analog with coordinates 
switched), and $\mbv$ is $(t_1- k_1, t_2 + 1 - k_2 + r_2)$. Since $\Delta_\mbu^\mbv (S/I)$ 
in this case must consist of two points, 
the definition of this simplicial complex implies that $S/I$ is 
nonzero in degree $\mbt + \one - \mbk$.
\end{proof}

\begin{lem}  The first cohomology group $H^1 \Na^{\mbk}_{\mbt}(S/I)$ vanishes 
if and only if either i) $S/I$ is nonzero in degree
$\mbt + \one - \mbk$ or ii) the degrees for which $S/I$ is nonzero is contained in
$[\mathbf 0, \mbt - \mbk]$.
\end{lem}

\begin{proof}
Suppose this first cohomology group is nonvanishing in degree $\mbr$.
By Theorem \ref{thecalc1},
$\widetilde H^{- \gamma}(\Delta_\mbu^\mbv (S/I))$ is nonvanishing. This is the reduced cohomology 
group of a simplicial complex in two variables, and so either i) $-\gamma = -1$,
or ii) $-\gamma = 0$.
%, then by Theorem \ref{thecalc1},
%either i) $\gamma = 1$ and $\widetilde H^{-1}(\Delta_\mbu^\mbv(S/I))$ is nonzero, or 
%ii) $\gamma = 0$ and $\widetilde H^{0}(\Delta_\mbu^\mbv(S/I))$ is nonzero.

i) In this case $\mbu$ is $(r_1 - k_1, t_2 + 1 - k_2)$ (or the analog with the coordinates 
switched), $\mbv$ is $(t_1 - k_1, t_2 +1 +r_2 - k_2)$ and 
$\Delta_\mbu^\mbv(S/I)$ is $\{\emptyset \}$. By the definition of this
simplicial complex,  $S/I$ is nonzero in degree
$\mbu$ and zero in degree $\mbt + \one - \mbk$. 

ii) In this case $\mbu$ is $(r_1 - k_1, r_2 - k_2)$ and $\mbv$ is $(t_1 - k_1, t_2- k_2)$
and $\Delta_\mbu^\mbv(S/I)$ consists of two points. Hence by its definition, $S/I$ 
is nonzero in degree, say $(t_1 - k_1 + 1, r_2 - k_2)$, and zero in degree $\mbt + \one - \mbk$.

Conversely, suppose $S/I$ is zero in degree $\mbt + \one - \mbk$ and the support of $S/I$ is not
contained in $[\mathbf 0 , \mbt - \mbk]$. 
If $(t_1 + 1-k_1, 0)$ and $(0, t_2 + 1-k_2)$ are both in the support of $S/I$ we can 
realize case ii) above for $\mbr = \mbk$. If, say only the second is in the support, 
we can realize case i) above with $\mbr = \mbk + \epsilon_2$.
\end{proof}
 
We can now summarize these results as follows.

\begin{prop} Suppose $S$ is a polynomial ring in two variables and $\mbk \geq \one$. Then
$\Na^{\mbk}_{\mbt} (S/I)$ has at most  one nonvanishing cohomology group if and only if either

a. The support of $S/I$ is contained in $[\mathbf 0, \mbt - \mbk]$.

b. The support of $S/I$ in $[\mbt + \one - \mbk, \mbt]$ is a nonempty interval.

c. The support of $S/I$ in $[\mathbf 0 , \mbt + \one - \mbk]$, is the complement of a
nonempty interval.

\end{prop}

\begin{proof} 

Suppose $\mbt + \one - \mbk$ is in the support of $S/I$, so cases a. and c. do not apply. 
Then $H^0$ and 
$H^1$ of  $\Na^{\mbk}_{\mbt}(S/I)$ vanishes, while $H^3$ vanishes if and only if we
are in case b.

Suppose $\mbt + \one - \mbk$ is not in the support of $S/I$, so case b. does not apply.
Then $H^2$ and $H^3$ of the Nakayama functor vanishes. That $H^1$
vanishes is equivalent to case a., and that $H^0$ vanishes is equivalent to case c.
\end{proof}

\section{The Nakayama functor on interval modules}
\label{sec:intmod}

In this section we consider the Nakayama functor applied to interval modules.
We investigate quite explicitly and in detail the one variable case $n=1$,
Subsection \ref{sec:onena}.
In particular we develop lemmata which will be pivotal in the final proofs.
In Subsection \ref{sec:genint} we infer immediate 
consequences to the general setting $n \geq 1$. 
Subsection \ref{sec:boundhom} concerns the explicit description of a connecting
homomorphism needed for the proof of Theorem \ref{nakcohS}.

%This section contains the core of our calculations. The main
%ingredient is a careful analysis of the case $n=1$ in 
%Section \ref{sec:onena}. In Section \ref{sec:genint} we infer 
%immediate consequences of this to the general situation when $n \geq.$ 
%Section \ref{sec:boundhom} concerns the explicit description of a connecting
%homomorphism needed for the proof of Theorem \ref{nakcohS}.

\subsection{One dimensional intervals}
\label{sec:onena}

In order to keep track of naturality properties of the Nakayama
functor we need to be quite explicit in our calculation of
one-dimensional Nakayama cohomology. 
\begin{defn}
\label{nakcx}
  Let $k,t \in \N$ with $k \le t+1$. Given $y \in [0,t]$ we define 
  the chain
  complex $\mycx^k_{t}(t-y)$
% = (\mycx^k_{t}(t-y)_0 \xrightarrow i \mycx^k_{t}(t-y)_{-1})$ of
  of $t$-determined $S$-modules concentrated in degrees $0$ and $-1$,
  where it is given by the formula 
  \begin{displaymath}
    \mycx^k_{t}(t-y) =
    \begin{cases}
      K_{t}\{k,t-y+k\} \to 0 & \text{if $k \le y$} \\
      \qquad \quad \,\,\, K_{t}\{k,t\} \xrightarrow i K_{t}\{k-y-1,t\}     & \text{if $y+1 \le k $}.
    \end{cases}      
    \end{displaymath}
  Here $i$ is 
  the canonical inclusion introduced in Example \ref{canex}. 
For $0 \le y
\le y' \le t$ there is a naturally defined chain homomorphism 
\[ \mycx^k_{t}(y,y') \colon \mycx^k_{t}(t-y) \to
\mycx^k_{t}(t-y') \]
as follows: 
\begin{enumerate}
\item If $k \le y$, then $\mycx^k_{t}(y,y')$ is the
canonical projection $p \colon K_{t}\{k,t-y+k\} \to
K_{t}\{k,t-y'+k\}$ introduced in Example \ref{canex}.
\item If $y < k \le 
y'$, then $\mycx^k_{t}(y,y')$ in degree zero is the canonical projection $p
\colon K_{t}\{k,t\} \to
K_{t}\{k,t-y'+k\}$.
\item If $y' < k $, then $\mycx^k_{t}(y,y')$ is the identity
in degree zero and the canonical inclusion $i \colon K_{t}\{k-y-1,t\} \to
K_{t}\{k-y'-1,t\}$ in degree $-1$.
\end{enumerate}
\end{defn}

The following lemma gives an explicit description of the Nakayama
functor applied to $S/I$ for a graded ideal $I$ in $S = K[x]$. 
\begin{lem}
\label{nathelper1}
  For $0 \le k \le t+1$ and $0 \le y \le t$ there is a
  quasi-isomorphism 
  \begin{displaymath}
    \Na^k_{t} K_{t}\{0,t-y\} \xrightarrow {f^k_t(y)} \mycx^k_{t}(t-y)
  \end{displaymath}
  of chain-complexes of $S$-modules. Further, if $y \le y'$, then
  $$\mycx^k_{t}(y,y') \circ f^k_t(y) = f^k_t(y') \circ \Na^k_{t}(p),$$
 where $p \colon
  K_{t}\{0,t-y\} \to K_{t}\{0,t-y'\}$ is the canonical projection.
\end{lem}
\begin{proof}
  We prove the result by induction on $k$. In Example \ref{firstnakone} we saw that
  \begin{displaymath}
    \Na^1_{t} K_{t}\{a,b\} \cong
    \begin{cases}
      K_{t}\{a+1,b+1\} & \text{if $b < {t}$}\\
      K_{t}\{a+1,{t}\} \to K_{t}\{0,{t}\} & \text{if $b = {t}$}
    \end{cases}
  \end{displaymath}
  In particular $\Na^1_{t} K_{t}\{0,{t}-y\} \cong \mycx^1_{t}(t-y)$, so the lemma holds for
  $k = 1$. Next suppose that the lemma holds for $k-1$. In order to
  see that it holds for $k$ it suffices to provide a quasi-isomorphism
  $g(y) \colon \Na^1_t \mycx^{k-1}_t(t-y) \to \mycx_t^k(t-y)$ such that
  $$\mycx^k_t(y,y') \circ g(y) = g(y') \circ \Na^1_t \mycx_t^{k-1}(y,y')$$
  when $y
  \le y'$.
% a sufficiently
%   natural quasi-isomorphism $\Na_{t}^1 D^{k-1}_{t}(y) \to
%   \mycx^k_{t}(t-y)$. 
  In the case $k
  \le y+1$ this is very easy because the two chain complexes are
  isomorphic. Otherwise, if $k \ge y+2$, then $\Na^1_{t} \mycx^{k-1}_t(t-y)$ is
  the total complex of the bicomplex
  \begin{displaymath}
    \begin{CD}
      K_{t}\{k,{t}\} @>>> K_{t}\{k-y-1,{t}\} \\
      @VVV @VVV \\
      K_{t}\{0,{t}\} @>{=}>> K_{t}\{0,{t}\}.
    \end{CD}
  \end{displaymath}
  Since $\mycx^k_t(t-y)$ is equal to the top row of this bicomplex there is
  a quasi-isomorphism $g(y) \colon \Na^1_t \mycx^{k-1}_t(t-y) \to \mycx_t^k(t-y)$
  projecting onto the top row.
%   However this total complex is clearly quasi-isomorphic to the chain
%   complex $\mycx^k_{t}(t-y)$ given by its first row. We leave it to the reader
%   to check the required naturality conditions.
\end{proof}
The above lemma enables us to verify Proposition \ref{nakcoh} in the case $n=1$. 
\begin{cor}
\label{intervaluv}
  For $0 \le k \le t+1$ and $0 \le y \le t$ the cohomology of the chain complex
  $\Na^k_t K_t \{0,t-y\}$ is isomorphic to  
  $K_t
    \{t-v^k_t(y),t-u^k_t(y)\}$
  concentrated in cohomological degree ${\gamma^k_t(y)}$.
% there is an isomorphism
%   \begin{displaymath}
%     H^{\gamma^k_t(y)} \Na^k_t K_t \{0,t-y\} \cong K_t
%     \{t-u^k_t(y)-v^k_t(y)+1,t-u^k_t(y)\} \ne 0
%   \end{displaymath}
%   and $H^i \Na^k_t K_t \{0,t-y\} = 0$ for $i \ne {\gamma^k_t(y)}$.
\end{cor}
The following somewhat surprising result is the main input in our
proof in Section \ref{sec:nakmon} of Theorem \ref{thecalc1},
which computes the cohomology of the Nakayama functor.
\begin{lem}
\label{nathelper2}
  For every $u$ with $0 \le u \le {t}$ and every $y$ with $0 \le y \le
  t$ there is a quasi-isomorphism
  $$\mycx^k_{t}(t-y)_u \xrightarrow {g^k_t(y,u)} \mycx^k_{t}(t-u)_y.$$
  Moreover   for every chain $d \in
  \mycx^k_{t}(t-y)_u$ and all $y,y'$ with $0 \le y \le y' \le t$ we have
  $$x^{y'-y}
  \cdot g^k_t(y,u)(d) =  g^k_t(y',u)(\mycx^k_{t}(y,y')(d))$$ 
\end{lem}
\begin{proof}
  We define $g^k_t(y,u)$ to be the identity on $K$ whenever
  $$\mycx^k_{t}(t-y)_u = \mycx^k_{t}(t-u)_y = K$$
  and to be zero otherwise. When $k
  = t+1$ the assertions of the lemma are obvious. When $k \le t$ the
  assertions can be verified by splitting up into the six cases
  obtained from the three cases $k \le y \le y'$ or $y < k \le y'$ or $y
  \le y' < k$ times the two cases $k \le u$ or $u<k$.
%\marginpar{check} 
\end{proof}
\begin{cor}
\label{natswapuy}
  For all $u$ and all $y,y'$ with $0 \le y \le y' \le t$ the following
  diagram commutes.
  \begin{displaymath}
    \begin{CD}
      (\Na_t^k K_t \{0,t-y\})_u @>{f^k_t(y)_u}>> \mycx^k_t(t-y)_u
      @>{g^k_t(y,u)}>> \mycx^k_t(t-u)_y 
      @<{f^k_t(u)_y}<<    (\Na^k_t K_t \{0,t-u\})_y \\
      @V{(\Na^k_t p)_u}VV @V{\mycx^k_t(y,y')_u}VV @V{x^{y'-y}}VV @V{x^{y'-y}}VV \\
      (\Na_t^k K_t \{0,t-y'\})_u @>{f^k_t(y')_u}>> \mycx^k_t(t-y')_u @>{g^k_t(y',u)}>> 
      \mycx^k_t(t-u)_{y'} @<{f^k_t(u)_{y'}}<<
      (\Na^k_t K_t \{0,t-u\})_{y'} 
    \end{CD}
  \end{displaymath}
\end{cor}

% For each fixed $u$ as above we can consider the chain complex $C^k_t(u)$
% of
% ${t}$-determined
% $S$-modules with $C^k_t(u)_y = \Na^k_{t}K_{t}\{0,{t}-y\}_u$ for $0 \le y \le
% t$ and with 
% multiplication by $x$ given by $\Na_{t}^k p_t({0,{t}-y-1,{t}-y})$.
% With this notation the above two lemmas have the following
% consequence:
% \begin{cor}
% \label{nathelper}
%   For each fixed $u$ with $0 \le u \le {t}$ there is a quasi-isomorphism
%   $C^k_t(u) \to \mycx^k_{t}(u)$.
% \end{cor}
% \begin{proof}
%   In Lemma \ref{nathelper1} we have shown that there is a quasi-isomorphism of
%   ${t}$-determined $S$-modules with the homomorphism $C_t^k(u)_y \to
%   \mycx^k_{t}(t-y)_u$ in degree $y$, and in Lemma \ref{nathelper2} we have provided a
%   quasi-isomorphism of ${t}$-determined $S$-modules with $\mycx^k_{t}(t-y)_u \to
%   \mycx^k_{t}(u)_y$ in degree $y$.
% \end{proof}
% The above corollary is the most important ingredient in our calculations.
% When we come to apply it we shall state is as a
% quasi-isomorphism
% \begin{displaymath}
%   \Na^k_{t} K_{t}\{0,t-y\}_u \simeq \Na^k_{t} K_{t}\{0,t-u\}_y
% \end{displaymath}
% which for $u$ fixed is natural in $y$. 
% \begin{remark}
%   Note that $\mycx^k_{t}(u) \simeq \Na^{k}_{t} K_{t}\{0,t-u\}$ has homology
%   concentrated in one homological degree, and therefore 
%   Corollary \ref{nathelper} is equivalent to the statement that
%   considered as an $S$-module the homology of $C^k_t(u)$ is isomorphic
%   to the homology of
%   $\mycx^k_{t}(u)$. 
% \end{remark}

\begin{cor}
\label{firststepbetti}
%   Here Gunnar must help!! (I would be happy if we can avoid working
%   with inverse Nakayama functors.)
%\marginpar{fixme}
 Let $F$ denote the projective resolution $S(-1) \to
  S$ of $K = S/\mm$ over $S$.
  Given $t,k \in \N$ with $0 \le k \le t +1$ and $u,y \in [ 0,t]$, 
  there is a chain of quasi-isomorphisms of chain complexes of $K$-vector spaces
  between $(T\Na^{k+1}_{t} K_{t}\{t-u,t-u\})_{y}$ and $(F \otimes_S
  \Na^{k}_{t} K_{t}\{ 0,t-y\})_{u}$. 
%   \begin{displaymath}
%     T\Na^{k+1}_{t} (F \otimes_S K_{t}\{0,t-u\}_{y}) 
% %\simeq
% %    T\Na^{k+1}_{t} K_{t}\{t-u,t-u\}_{y} 
% \simeq
%     (F \otimes_S \Na^{k}_{t} K_{t}\{ 0,t-y\})_{u}.
%   \end{displaymath}
  For each fixed $u$ these quasi-isomorphisms are natural in $y$.
%  This quasi-isomorphism is natural in $x$ and $y$.
\end{cor}
\begin{proof}
%  Let us start by noting that $F \otimes_S K_t\{t-u,t\}$ is a free
%  resolution of $K_t\{t-u,t-u\}$. 
We split up into the cases $u=0$ and
  $u>0$. First the case $u=0$.
  The computation in Example \ref{firstnakone} gives
  a quasi-isomorphism
  \begin{displaymath}
    (T\Na^{k+1}_tK_t\{t,t\})_y \simeq (\Na^{k}_tK_t\{0,t\})_y.
  \end{displaymath}
  Lemma \ref{nathelper2} implies that the latter chain complex is quasi-isomorphic to
  \begin{displaymath}
    (\Na^k_t K_t\{0,t-y\})_{0} = (F \otimes_S \Na^k_t K_t\{0,t-y\})_0. 
  \end{displaymath}

  Next we treat the case $u>0$.
  By exactness of the Nakayama functor
  and the computation in Example \ref{firstnakone}
  we obtain quasi-isomorphisms
  \begin{eqnarray*}
    (T\Na^{k+1}_tK_t\{t-u,t-u\})_y &\simeq& (T \Na_t^{k+1} (F \otimes_S
    K_t\{t-u,t\}))_y \\
    &\simeq& (\Na_t^k(K_t\{0,t-u+1\} \xrightarrow p  K_t\{0,t-u\}))_y.
  \end{eqnarray*}
  By exactness of the Nakayama functor again the latter of these chain
  complexes is isomorphic to the mapping cone of the chain complex
  \begin{displaymath}
    (\Na_t^kK_t\{0,t-u+1\})_y \xrightarrow {(\Na_t^k(p))_y} (\Na_t^kK_t\{0,t-u\})_y.
  \end{displaymath}
Starting from the left side of the diagram of 
Corollary \ref{natswapuy} 
we see that this map (notation is a bit switced, so $y^\prime = u$) 
is quasi-isomorphic to the mapping cone of the chain complex
  to the chain complex
  \begin{displaymath}
(\Na^k_t K_t\{0,t-y\})_{u-1}
    \xrightarrow x (\Na^k_t K_t\{0,t-y\})_{u}  = 
    (F \otimes_S \Na^k_t (K_t\{0,t-y\})_u. 
  \end{displaymath}
  The naturality statement follows from the fact that
  the $K$-vector spaces
  $T\Na^{k+1}_{t} K_{t}\{t-u,t-u\}_y$ have homology concentrated in the
  same homological degree for all $y$. 
% This however is easy if we keep the
%   following two facts in mind:
%   Firstly, there is a weak equivalence
%   \begin{displaymath}
%     (F \otimes_S K_t\{a,t-b\})_u \simeq T\Na^1_{t} K_{t}\{t+a-u,t+a-u\})_{a+b},
%   \end{displaymath}
%   and secondly $K_{t}\{v,v\} =\Na_{t}^v K_{t}\{0,0\}$.
\end{proof}
\subsection{General intervals}
\label{sec:genint}

We shall now transfer the results for one dimensional intervals to
general intervals. This is very easy because if $0 \le \mbr \le \mbt$ in
$\N^n$ then there is an isomorphism 
% Let $\mbt \in \N^n$.
% % In this section we fix $l \in (\N \cup \{\infty\}  )^n$ and $k \in
% % \N^n$ with $k_j \le t_j+2$ for all $j$. 
% Given $r,s \in \N^n$ with $r
% \le s \le l$ the {\em interval module $K\{0,r\}$} is the 
% $l$-determined 
% $S = K[x_1,\dots,x_n]$-module $M$ with the property that if $0 \le 0 \le l$ then $M_a =
% K$ if $a \in [r,s]$ and $M_a = 0$ 
% otherwise. Multiplication by $x_j$ is 
% the identity homomorphism from $M_a$ to $M_{a + \varepsilon_j}$
% whenever these are equal to $K$.
% Explicitely,
% \begin{displaymath}
%   K\{r,s\} = \bigotimes_{j=1}^n K\{r_j,s_j\} = (S/x^{s-r+1}S(-r))[0,l].
% \end{displaymath}
% Note that
  \begin{displaymath}
    \Na^{\mbk}_{\mbt} K_{\mbt}\{ 0,\mbr\} \cong \bigotimes_{j=1}^n \Na^{k_j}_{t_j}K_{t_j}\{0,r_j\}
  \end{displaymath}
where the outer tensor product is over $K$, making it a module over
$S = \bigotimes_{j=1}^n \kr[x_j]$. 

With the notation $\mycx^{\mbk}_{\mbt}(\mbr) = \bigotimes_{j=1}^n 
%S
%\otimes_{K[x_j]} 
\mycx^{k_j}_{t_j}(r_j)$ there are obvious
$\Z^n$-graded versions of all the results in Section
\ref{sec:onena}. 
In particular Proposition \ref{nakcoh} is the $\Z^n$-graded version of
Corollary \ref{intervaluv}.

\begin{proof}[Proof of Proposition \ref{trans}.]
First assume $C$ is a module $M$. Choose a resolution $M \to E$ of $M$ where each $E_i$ is a direct
  sum of $\Z^n$-graded $S$-modules of the form $K_{t}\{0,\mbr\}$. Such a
  resolution clearly exists, it is an injective resolution of $M$; see e.g. \cite{Miller2000} for further
  details. By the $\Z^n$-graded version of Lemma \ref{nathelper1}, for
  each $i$, we have a quasi-isomorphism 
  $\Na^{(t_j+2)\varepsilon_j}_{\mbt}E_i \xrightarrow {\simeq} T^{-2}E_i$, and therefore by
  naturality we obtain a quasi-isomorphism $\Na^{(t_j+2)\varepsilon_j}_{\mbt} E
  \xrightarrow \simeq T^{-2}E$. Since the Nakayama functor
  is exact, we are done.

Now assume $C$ is any chain complex. It is then a colimit of bounded
below chain complexes. Since the Nakayama functor commutes with 
colimits, is exact, and colimits is exact on chain complexes, we reduce
to the case when $C$ is bounded below. A bounded below complex 
is a colimit of truncations $\tau_{\leq p} C$ where 
\[ (\tau_{\leq p} C)_i = \begin{cases} 0  & i > p \\
   \ker (C_p \rightarrow C_{p-1}) & i = p \\
    C_i & i < p
\   \end{cases}
\]
Hence we may reduce to the case when $C$ is bounded, and again by exactness
of the Nakayama functor, we reduce to the case when $C$ is a module $M$.
%   \begin{displaymath}
%     \Na_{\mbt}^{(t_j+2)\varepsilon_j} M \xrightarrow \simeq \Na_{\mbt}^{(t_j+2)\varepsilon_j}E \xrightarrow
%     \simeq T^{-2}E \xleftarrow \simeq T^{-2}M.
%   \end{displaymath}
\end{proof}

\subsection{Connecting homomorphisms}
\label{sec:boundhom}

Let $C' \xrightarrow g C$ be a morphism
of chain complexes of $S$-modules. The
{\em mapping cone} of $g$ is the chain complex $C(g)$ with $C(g)_i =
C'_{i-1} \oplus C_i$ and with differential sending $(x,y)$ to $(-dx,gx+dy)$. 
If $\iota$ denotes the inclusion $\iota \colon C
\to C(g)$ with $\iota(y) = (0,y)$, there
is a short exact sequence
%\begin{displaymath}
$  C \xrightarrow \iota C(g) \xrightarrow \delta TC'$
%\end{displaymath}
where $\delta(x,y) = x$. We call $\delta$ the {\em connecting
  homomorphism} associated to $g$.

Let $M$ and $N$ be $S$-modules.
Considering $M$ and $N$  as chain complexes concentrated in degree
zero there is for every pair of  integers $\gamma$ and $\rho$ an isomorphism 
\begin{eqnarray*}
 \varphi
\colon (T^{\gamma} M) \otimes_S 
C(g) \otimes_S (T^{\rho} N) &\to& C((T^{\gamma}M)\otimes g \otimes
(T^{\rho} N)) \\
  m \otimes (x,y) \otimes n & \mapsto & 
((-1)^{\gamma}m \otimes x \otimes n ,m \otimes y \otimes n) 
\end{eqnarray*}
of chain complexes. Note that according to the Koszul Sign Convention
or the conventions of \cite{We}, p.8 and 9, the differentials of
the left and right complexes are given respectively by :
\begin{eqnarray*}
m  \te (x,y) \te n & \mapsto &  (-1)^{\gamma} m \te (-dx, gx + dy) \te n \\
(m \te x \te n, m \te y \te n) & \mapsto & 
((-1)^{\gamma} m \te dx \te n, m \te gx \te n + (-1)^{\gamma} m \te dy \te n)
\end{eqnarray*}

We then have a commutative diagram :
\begin{displaymath}
  \begin{CD}
    (T^{\gamma}M)
    \otimes_S C(g) \otimes_S (T^{\rho} N) 
    @>{\varphi}>>
    C((T^{\gamma}M) \otimes g \otimes (T^{\rho} N)) \\
    @V{(T^{\gamma}M) \otimes
      \delta \otimes (T^{\rho} N) }VV 
    @V{\delta}VV
    \\
 (T^{\gamma}M)
    \otimes_S T C' \otimes_S (T^{\rho} N) 
     @>{(-1)^{\gamma}}>>
    T((T^{\gamma}M)
    \otimes_S C' \otimes_S (T^{\rho} N)). 
  \end{CD}
\end{displaymath}

\medskip
Now consider the inclusion
\[ K_t\{k,t \} \mto{i} K_t \{ 0, t \}. \]
Note that the domain of $i$ is $N_t^k(t-k)$ and that $N_t^k(t-k+1)$
is the mapping cone of this morphism. Furthermore the map 
$N_t^k(k-1,k)$ is the connecting map $\delta$ corresponding to the 
morphism $i$. 

For arbitrary $n$, the map $N_\mbt^{\mbk}(\mby - \epsilon_i, \mby)$ 
where $y_j = k_j$ may be identified as
\[ \left( \bigotimes_{i<j} \mycx_{t_i}^{k_i}(t_i-y_i) \right)
  \otimes \delta \otimes
  \left( \bigotimes_{i>j} \mycx_{t_i}^{k_i}(t_i-y_i) \right)  \]
which is then $(-1)^{\sum_{i=1}^{j-1} \gamma_i}$
times the connecting map of the injection
\begin{eqnarray*}  
& &  \left( \bigotimes_{i<j} \mycx_{t_i}^{k_i}(t_i-y_i) \right)
  \otimes K_t\{ k_j, t_j \} \otimes
  \left( \bigotimes_{i>j} \mycx_{t_i}^{k_i}(t_i-y_i) \right) \\
& \rightarrow & 
\left( \bigotimes_{i<j} \mycx_{t_i}^{k_i}(t_i-y_i) \right)
\otimes K_t\{ 0, t_j \} \otimes
\left( \bigotimes_{i>j} \mycx_{t_i}^{k_i}(t_i-y_i) \right).
\end{eqnarray*}

\section{Homological algebra over posets}
\label{sec:hoalposets}

% In this section we shall establish connections between particular
% $KP$-modules and cohomology of simplicial complexes. 

In this section we give some relations between reduced 
cohomology groups of order complexes of various subposets of a poset $P$
and Ext-groups of $KP$-modules. 

\subsection{Modules over posets}
\label{sec:postemod}

Let $P$ be a partially ordered set. 
A $KP$-module is a functor from $P$, considered as a category with a
morphism $p \to p'$ if and 
only if $p \le p'$, to the category of $K$-vector spaces. Given two
$KP$-modules $M$ and $N$ we denote the $K$-vector space of $K$-linear
natural transformations from $M$ to $N$ by $\Hom_{KP}(M,N)$.
If $A$ is a (commutative) $K$-algebra we may also consider $AP$-modules, that
is, functors from
$P$ to the category of $A$-modules. If $M$ is a $KP$-modules, and $N$ an $AP$-module, 
then $\Hom_{KP}(M,N)$ is naturally an $A$-module. 

On occasion we allow $M$ and $N$ to be chain complexes of $K$-vector spaces. Then
$\Hom_{KP}(M,N)$ also becomes a chain complex, by taking the total complex of a
double complex.

\begin{defn} \label{DefKP}
  Let $X$ be a convex subset of a poset $P$ i.e. with the property that 
if $x,y \in X$ then all elements $z$ with $x \le z \le y$ are in 
$X$. We define the $KP$-module $K\{X\}$ by letting $K\{X\}(p)$ be $K$ if $p \in X$
and letting it be $0$ otherwise. If $p \le p'$ are both in $X$, then the structure homomorphism
$K\{X\}(p) \to K\{X\}(p')$ is the identity on $K$. Otherwise this
structure homomorphism is zero.
\end{defn}
\begin{remark} When $P$ is finite
  the category of $KP$-modules is isomorphic to the category of (left)
  modules over the incidence algebra $I(P)$ of $P$ over $K$. Given an
  order preserving map, that is, functor $f \colon P \to Q$ of
  partially ordered sets 
  there is an induced functor $f^*$ from the category of $KQ$-modules
  to the category of $KP$-modules given by precomposition with
  $f$. However $f$ does not in general induce a ring homomorphism from
  $I(P)$ to $I(Q)$. (Not even when $P = \{0 < 1\}$ and $Q
  = \{0\}$.) This is the main reason why we prefer to work with
  $KP$-modules instead of $I(P)$-modules.
\end{remark}
A
subset $J$ of $P$ is an {\em order ideal} if $x \le y$ in $P$ and $y
\in J$ implies $x \in J$. Dually an {\em order filter} in $P$ is a subset $F$
such that $y \le x$ in $P$ and $y
\in F$ implies $x \in F$.

% Let $P$ denote a partially ordered set. A $KP$-module is a functor
% from $P$, considered as a category with a morphim $p \to p'$ if and
% only if $p \le p'$, to the category of $K$-vector spaces. This is an
% abelian category with enough injectives and enough projectives. Thus
% we can freely do homological algebra in the category of
% $KP$-modules. Moreover the morphism sets in the category of
% $KP$-modules are all equipped with the structure of a $K$-vector
% space. Given two $KP$-modules $M$ and $N$ we shall let $\Hom_{KP}(M,N)$
% denote the $K$-vector space of morphisms, that is, $K$-linear natural
% transforamtions from $M$ to $N$. 
% If $X$ is a subset of $P$ with the property that if $y \in P$ and if
% there exist $x,z \in X$ such that $x \le y \le z$ then $y$ must be in
% $X$, then there is a $KP$-module $K\{X\}$ with $K\{X\}(p) = K$ if $p \in X$
% and with $K\{X\}(p) = 0$ otherwise. 
%% If $X = [r,s] = \{a \colon r \le a \le s\}$ for $r,s 
%% \in \N^n$, then the $KP$-module $KX = K[r,s]$ may be confused with the
%% interval module of the above section. We shall make it clear
%% from the context if $K[r,s]$ denotes an interval module or a $KP$-module.
\begin{lem}
\label{uniquemaxinj}
  If an order ideal $J$ in a poset $P$ has a unique maximal element $x$,
  then $K\{J\}$ is an injective $KP$-module.
\end{lem}
\begin{proof}
  This is immediate from the isomorphism
  \begin{displaymath}
    \Hom_{KP}(N,K\{J\}) \cong \Hom_K(N(x),K), \qquad f \mapsto f(x).
  \end{displaymath}
\end{proof}
Recall that the {\em order complex} of a partially
ordered set $P$ is the simplicial complex $\Delta(P)$ with the
underlying set of $P$ as vertex set, and a subset $F$ of $P$ is in $\Delta(P)$ if and
only if it is a chain, that is, a totally ordered subset of $P$. 
\begin{prop}
\label{homiscoh}
  Let $P$ be a poset. If $E$ is a projective resolution of the
  $KP$-module $K\{P\}$ then there is an isomorphism
  \begin{displaymath}
    H^{*} \Hom_{KP}(E,K\{P\}) \cong H^* (\Delta(P)).
  \end{displaymath}
\end{prop}
\begin{proof}
Note first of all that the cohomology of $\Hom_{KP}(E,K\{P\})$ is
independent of the choice of $E$. Also for $p$ in $P$, the projective
cover of $K\{p\}$ is the module $E_{p}$ with $E_{p}(q) = K$
for $q \geq p$ and zero otherwise, and for $p \leq q \leq q^\prime$
the morpshism $E_{p}(q) \rightarrow E_{p}(q^\prime)$ is the identity.

For each $k$-chain in the poset $P$
\[c \, : \, p_0 < p_1 < \cdots < p_k\] 
we get a projective $E_c = E_{p_k}$. 
Let $E_k$ be the projective which is the direct sum of the $E_c$ where
$c$ ranges over the $k$-chains. Then $E_k(q)$ has a basis consisisting
of the all $(k+1)$-chains $p_0 < \cdots < p_k \leq q$. The $E_k$ give a projective
resolution of $K\{P \}$ with differential
\[ (p_0 < p_1 < \cdots < p_k \leq q) \mapsto
\sum_{i = 0}^ k (-1)^i (p_0 < \cdots \hat{p_i} \cdots < p_k \leq q). \] 
This complex augmented with $K\{P\}$ 
is exact as $K$-vector spaces, since the map sending
a $(k+1)$-chain $c$ to the $(k+2)$-chain with $q$ repeated if $p_k < q$
and to zero if $p_k = q$, gives a homotopy of the augmentation.
Since $\Hom_{KP}(E_c,K\{P\})$ may be identified with the one-dimensional
vector space with basis the dual of the chain $c$, 
we see that $\Hom_{KP}(E_k, K\{P\})$
is the vector space with basis the dual of  all $(k+1)$-chains, 
and so $\Hom_{KP}(E,K\{P\})$
is the cochain complex of the order complex $\Delta(P)$.

%If we choose $E$ to be the
%  projective resolution with $E_k(p)$ equal to the $K$-vector space
%  with basis the set of $k+2$-chains in $P$ of the form $p_0 < p_1 <
%  \dots < p_k <p$, then $\Hom_{KP}(E,K\{P\})$ is the simplicial cochain
%  complex for $\Delta(P)$.
\end{proof}
\begin{lem}
\label{flasquelem}
Let $J$ be an order ideal in $\N^n$
and let $R$ denote the functor
from $J^{\op}$ to $\Z^n$-graded $S$-modules with 
\begin{displaymath}
  R(\mbu) = K\{0,\mbu\} := S/\mm^{\mbu+\one}.
\end{displaymath}
For every projective resolution $E \to K\{J^{\op}\}$ of the
$J^{\op}$-module $K\{J^{\op}\}$
the homomorphism 
$\Hom_{KJ^{\op}}(K\{J^{\op}\},R) \to \Hom_{KJ^{\op}}(E,R)$ is a 
quasi-isomorphism of chain complexes of $\Z^n$-graded $S$-modules.   
\end{lem}
\begin{proof}
  It suffices to show that the above homomorphism is a quasi-isomorphism 
  in every degree $\mbr$. Note that the $KJ^{\op}$-module $\mbu
  \mapsto R(\mbu)_{\mbr}$ is equal to the $KJ^{\op}$-module
  $K\{((\mbr+\N^n) \cap J)^{\op}\}$. 
  Thus we have to show that the 
  homomorphism
  \begin{displaymath}
    \Hom_{KJ^{\op}}(K\{J^{\op}\},K\{((\mbr+\N^n) \cap J)^{\op}\}) \to
    \Hom_{KJ^{\op}}(E,K\{((\mbr+\N^n) \cap J)^{\op}\}) 
  \end{displaymath}
  is a quasi-isomorphism of chain complexes of $K$-vector spaces.
  This a consequence of Lemma \ref{uniquemaxinj}.
\end{proof}

\begin{remark}
If $F$ is an order filter in a poset $P$, then $\Hom_{KP}(K\{F\}, N)$ is the 
(inverse) limit $\varprojlim_{p \in F} N(p)$. 
The order ideal $J$ above corresponds to the monomial ideal $I$ containing the monomials
$x^{\mba}$ such that $\mba$ is not in $J$. Then $\Hom_{KJ^{\op}}(K\{J^{\op}\}, R)$ is
simply equal to $S/I$. 
\end{remark}

\subsection{Posets and cohomology of simplicial complexes}
\label{sec:possimp}

%Let $I$ be a monomial ideal in $S$ and let $J$ denote the
%order ideal in $\N^n$ consisting of those $\mba \in \N^n$ with $(S/I)_{\mba} =
%K$. Given $\mba,\mbb \in \N^n$ 
%we shall relate the cochain complex of the simplicial complex
%$\Delta^{\mbb}_{\mba}(S/I;\mbt)$ from Section \ref{sec:cxfromid}
%to Ext-groups in the category of $KJ^{\op}$-modules. The following
%well-known fact will be important for us.

We shall need the following standard fact relating partially ordered
sets and the geometric realization of their order complexes.

\begin{prop}
If $P$ and $Q$ are partially ordered sets and $\varphi \colon P \to
Q$ and $\psi \colon Q \to P$ are order preserving maps satisfying 
$\psi(\varphi(x)) \le x$ and $\varphi(\psi(y)) \le y$, then the induced maps
$|\Delta(\varphi)|$ and $|\Delta(\psi)|$ of geometric realizations of
order complexes are inverse 
homotopy equivalences.  
\end{prop}
\begin{proof}
  Consider the order preserving map $h \colon \{0<1\} \times P
  \to P$ with $h(0,x) = \psi(\varphi(x))$ and $h(1,x) = x$ and use the
  homeomorphism 
  $$|\Delta(\{0<1\} \times P)| \cong 
  |\Delta(\{0<1\})| \times |\Delta(P)|.$$
\end{proof}

%Let $I$ be a monomial ideal in $S$ and let $\mba \le
%\mbb$ in $\N^n$. The simplicial complex
%$\Delta_{\mba}^{\mbb}(S/I)$ on the vertex set $\{1,\dots,n\}$
%is defined to consist of the subsets $F$ of the
%set $\{1,\dots,n\}$ with the property that the degree 
%$$\mbx = \sum_{i \notin F} a_i
%\varepsilon_i + \sum _{i \in F} (b_{i} +1) \varepsilon_i$$
%satisfies
%$(S/I)_{\mbx} = K$. Note that for $\mbt \ge \mbb$ we have
%$$\Delta_{\mba}^{\mbb}(S/I;\mbt) = \Delta_{\mba}^{\mbb}(S/(I + \mm^{\mbt+\one})).$$

Given an order ideal $J$ in $\N^n$ and elements $\mba$ and $\mbb$ in $\N^n$. 
Define the simplicial complex $\Delta_\mba^\mbb(J)$ to consist of the subsets 
$F$ of $\{1, 2, \ldots, n \}$ such that the degree 
\[\sum_{i \notin F} a_i
\varepsilon_i + \sum _{i \in F} (b_{i} +1) \varepsilon_i \]
is in $J$.

\begin{lem}
\label{heqlemma}
%   Given an order ideal $J$ in $\N^n$ and $\mba,\mbb \in \N^n$  
%   %satisfy $y+1 \ge x$, 
  Given $\mba \le \mbb$ in $\N^n$, the geometric realization of the simplicial complex
  $\Delta_{\mba}^{\mbb}(J)$ is homotopy equivalent to the
  geometric realization of the order
  complex of the poset $J 
  \cap (\mba + (\N^n \setminus [0,\mbb-\mba]))$.
\end{lem}
\begin{proof}
  Firstly the order poset of $\Delta_{\mba}^{\mbb}(J)$ is
  isomorphic to the source of the inclusion 
% poset 
%   $$J
%   \cap (\mba + \prod_{i=1}^n \{0,b_i-a_i+1\})\setminus \{\mba\}.$$ 
%   We claim that
%   the inclusion  
  \begin{displaymath}
    \varphi \colon J \cap (\mba + \prod_{i=1}^n \{0,b_i-a_i+1\}) \setminus \{\mba\}
    \xrightarrow \subseteq J 
  \cap (\mba + (\N^n \setminus [0,\mbb-\mba])).
  \end{displaymath}
  It suffices to show that there is an order
% this inclusion is a deformation retraction, that is,
%   that there exists an order 
  preserving map $\psi$ in the opposite direction such that $\psi
  (\varphi (\mbx)) = \mbx$ and 
  $\varphi (\psi(\mby)) \le \mby$. In order to see this we note that the order
  preserving map 
  \begin{displaymath}
    \psi \colon \mba + (\N^n \setminus [0,\mbb -\mba]) \to 
    (\mba + \prod_{i=1}^n \{0,b_i-a_i+1\} )\setminus \{\mba\}
  \end{displaymath}
  with
  \begin{displaymath}
   \psi(\mby)_i =
   \begin{cases}
     a_i & \text{if $a_i \le y_i \le b_i$} \\
     b_i +1 & \text{if $b_i +1 \le y_i$}
   \end{cases}
  \end{displaymath}
is such a map.
\end{proof}

\begin{prop}
\label{modscx}
\label{natrem2}
  Given an order ideal $J$ in $\N^n$ and $\mba \le \mbb$ in $\N^n$, 
and let $J_\mba^\mbb = (J \cap [\mba,\mbb])^{\op}$. 

a. For every 
projective resolution $E$ of the $KJ^{\op}$-module
$K\{J^{\op}\}$ there is an isomorphism of the form
  \begin{displaymath}
    H^{i + 1} \Hom_{KJ^{\op}}(E,K\{J_\mba^\mbb\}) \cong \widetilde H^{i}(\Delta_{\mba}^{\mbb}(J)).
  \end{displaymath}
%  chain of
%   quasi-isomorphism between $T \Hom_{KJ^{\op}}(E,K\{J \cap
%   [\mba,\mba+\mbb]\})$ and the cochain complex $\widetilde
%   C^{-*}(J_{\,[\mba]}^{[\mbb+\one]})$.

b. If $\mba' \le \mbb'$ in $\N^n$ satisfy $\mba \le \mba'$ and $\mbb \le \mbb'$, then 
$\Delta_{\mba}^{\mbb'}(J)$ and
$\Delta_{\mba'}^{\mbb}(J)$ are subcomplexes of
$\Delta_{\mba}^{\mbb}(J)$, there is an injection 
$K\{J_\mba^{\mbb}\} \to K\{J_\mba^{\mbb'}\}$ and a projection 
$K\{J_{\mba}^{\mbb}\} \to K\{J_{\mba'}^\mbb\}$ of
$KJ^{\op}$-modules, 
and the diagrams
\begin{displaymath}
  \begin{CD}
        H^{i + 1} \Hom_{KJ^{\op}}(E,K\{J_\mba^\mbb\}) @>\cong>> \widetilde
        H^{i}(\Delta_{\mba}^{\mbb}(J)) \\
        @VVV @VVV \\
        H^{i + 1} \Hom_{KJ^{\op}}(E,K\{J_\mba^{\mbb'}\}) @>\cong>> \widetilde
        H^{i}(\Delta_{\mba}^{\mbb'}(J)) 
  \end{CD}
\end{displaymath}
and
\begin{displaymath}
  \begin{CD}
        H^{i + 1} \Hom_{KJ^{\op}}(E,K\{J_\mba^\mbb\}) @>\cong>> \widetilde
        H^{i}(\Delta_{\mba}^{\mbb}(J)) \\
        @VVV @VVV \\
        H^{i + 1} \Hom_{KJ^{\op}}(E,K\{J_{\mba'}^{\mbb}\}) @>\cong>> \widetilde
        H^{i}(\Delta_{\mba'}^{\mbb}(J)) 
  \end{CD}
\end{displaymath}
commute.
\end{prop}
\begin{proof} a. 
  If $\mba \notin J$, then $\Delta_{\mba}^{\mbb}(J)$ is the empty simplicial complex
  $\emptyset$ and both sides are zero. 
  Otherwise we note that the partially ordered set $P = (J \cap (\mba
  + \N^n))^{\op}$ is the disjoint union of the order filter
  $F = (J \cap ([\mba,\mbb]))^{\op}$ and the order ideal $A = (J \cap (\mba + (\N^n
  \setminus [0,\mbb-\mba])))^{\op}$. 
  From the short exact sequence $K\{F\} \to K\{P\} \to K\{A\}$ we deduce that
  there is a short exact sequence
  \begin{displaymath}
    \Hom_{KP}(E|_P,K\{F\}) \to \Hom_{KP}(E|_P,K\{P\}) \to \Hom_{KP}(E|_P,K\{A\}).
  \end{displaymath}
  Since $P$ has a unique maximal element we can use Lemma
  \ref{uniquemaxinj} to see that the cohomology of the middle term is
  a copy of $K$ concentrated in cohomological degree zero. 
%   Thus the
%   long exact cohomology sequence provides an isomorphism
%   for all $i$.
  On the other hand, applying Proposition \ref{homiscoh} we obtain 
  isomorphisms
  \begin{displaymath}
    H^{i} \Hom_{KP}(E|_P,K\{A\}) \cong H^i \Hom_{KA}(E|_A,K\{A\}) \cong H^{i}(\Delta(A)).
  \end{displaymath}
  Thus there is an isomorphism
  \begin{displaymath}
     \widetilde H^i(\Delta(A)) \xrightarrow \cong  H^{i+1} \Hom_{KP}(E|_P,K\{F\})
  \end{displaymath}
  for all $i$.
  Since $P$ is an order ideal in $J^{\op}$ there is an isomorphism 
  \begin{displaymath}
    H^{i+1}\Hom_{KP} (E|_P,K\{F\}) \cong H^{i+1} {\Hom_{KJ^{\op}} (E,K\{F\})}.
  \end{displaymath}
  Moreover there are
  isomorphisms 
  \begin{displaymath}
    \widetilde H^i(\Delta(A)) \cong
    \widetilde H^i(\Delta(A^{\op})) \cong
    \widetilde H^i(\Delta(J \cap (\mba+(\N^n \setminus [0,\mbb-\mba])))). 
  \end{displaymath}
  The result now follows by applying Lemma \ref{heqlemma}. We leave
  the proof of the naturality statements in b. for the reader.
%\marginpar{is this too tough? There is a homotopy entering!}
\end{proof}

\section{Independence of $\mbt$}
\label{sec:stabt}

For a monomial ideal $I$ there will be many $\mbt$ such that $I$ is
positively $\mbt$-determined. This raises the question as to how our calculations
depend on $\mbt$. In particular, what structural properties of the
complexes $\Na_{\mbt}^{\mbk}(S/I)$ do not depend on $\mbt$ and
hence give intrinsic properties of the monomial ideal? In this section
we consider such properties as vanishing of cohomology modules of these
complexes and also linearity conditions on these complexes.

%We start by observing that under certain circumstances the Nakayama
%functor $\Na_{\mbt}^{\mbk}(M)$ evaluated on a $\Z^n$-graded $S$-module
%$M$ is independent of $\mbt$. We begin with two results that can
%readily be reduced to the case $n=1$. The first result is a sharpening
%of 
%Proposition \ref{extstab}.
%\begin{prop}
%  If $\mbr + \mbk \le \mbt$ and $M$ is $\mbr$-determined, then
%  $\Na^{\mbk}_{\mbt} M$ is $(\mbr + \mbk)$-determined and there is an
%  isomorphism
%  \begin{displaymath}
%    \Na^{\mbk}_{\mbt} (M) \cong \Na_{\mbr + \mbk}^{\mbk}(M).
%  \end{displaymath}
%\end{prop}
%\begin{prop}
%  If $M$ is positively $\mbt$-determined and $M_{\mba} = 0$ whenever $k_i \ne 0$
%  and $t_i-k_i+1 \le a_i$ for some $i$ then there is an isomorphism
%  $\Na^{\mbk}_{\mbt}(M) \cong M(-\mbk)$.
%\end{prop}

\subsection{How complexes vary when $\mbt$ vary.}
Since $\Na^{\one}_{\mbt}(M)$ essentially computes the local cohomology modules
of $M$, it will have one non-vanishing cohomology module if and only if $M$ is a 
Cohen-Macaulay module over $S$. Hence this is a property of the complex which
is independent of which $\mbt$ we chose. 

We now present a different  stability result for the Nakayama functor. 
Given $k$ with $0 \le k \le t$ and $r \ge 0$ we
let $q_k^r$ denote the order preserving endomorphism of $\N$ defined
by the formula
\begin{displaymath}
  q_k^r(i) =
  \begin{cases}
    i & \text{if $0 \le i \le k$} \\
    k & \text{if $k \le i \le k+r$} \\
    i-r & \text{if $k+r \le i$}.
  \end{cases}
\end{displaymath}
%Moreover for $c = -1$ and $r \ge 0$ we define the order preserving
%endomorphism $q_{-1}^r$ of $\Z_*$ by the formula
%\begin{displaymath}
%  q_{-1}^r(i) =
%  \begin{cases}
%    -\infty & \text{if $i<r$} \\
%    i-r & \text{if $i \ge r$}.
%  \end{cases}
%\end{displaymath}
%Note that Gunnar's 11.3 follows from these two porpositions.
With this notation it is readily seen that for  $0 \le k \le t$ and
$0 \le y \le t$ that 
\begin{displaymath}
  \mycx_{t+r}^{k+1+r}(t-y) = (q_{k}^r)^* \mycx_t^{k+1}(t-y).
\end{displaymath}
By Lemma \ref{nathelper1} we can conclude that if $n=1$ and $I$ is a
$t$-determined 
ideal in $S = K[x]$, then the chain complexes $\Na^{k+1+r}_{t+r}(S/I)$
and $(q_k^r)^*\Na^{k+1}_t(S/I)$ are quasi-isomorphic.

Let us pass to the situation where $n > 1$. For $\mbk,\mbr \in \N^n$
with ${\mathbf 0}\le \mbk \le \mbt $ we define $q_{\mbk}^{\mbr} \colon
\N^n \to \N^n$ by letting $q_{\mbk}^{\mbr}(\mbx) =
(q_{k_1}^{r_1}(x_1),\dots,q_{k_n}^{r_n}(x_n))$.
% with the convention
%that if any coordinate $y_i$ of $\mby = (y_1,\dots,y_n)$ is equal to
%$-\infty$, then $\mby = -\infty$.
\begin{prop}
\label{realstability}
  Let $M$ be a positively $\mbt$-determined $S$-module and let $\mbk,\mbr \in \N^n$
  with ${\mathbf 0} \le \mbk \le \mbt$. There is a chain of
  quasi-isomorphisms between the chain complexes
  $\Na_{\mbt+\mbr}^{\mbk+\one + \mbr}(M)$ and $(q_{\mbk}^{\mbr})^* 
\Na_{\mbt}^{\mbk+ \one}(M)$.
\end{prop}
\begin{proof}
  First let us consider the case where $M = S/I$ for a
  positively $\mbt$-determined ideal $I$ in $S$.
  Let $J$ denote the
  order ideal in $\N^n$ consisting of those $\mba \in [0,\mbt]$ with
  $(S/I)_{\mba} =
  K$. 
% In different terms $J$ is the 
%   poset 
%   $$J = \{\mba \in \N^n \colon \text{$I_{\mba} = 0$
%   and $\mba \le \mbt$}\}.$$
  Let $R$ denote the functor from $J^{\op}$ to the category of
  positively $\mbt$-determined $S$-modules with 
  $R(\mbw) = K_{\mbt}\{0,\mbw\}$. 
  Clearly, as an $S$-module
  \begin{displaymath}
    S/I \cong \lim_{\mbw \in J^{\op}} R(\mbw) = \Hom_{KJ^{\op}} (K\{J\},R).
  \end{displaymath}
  We have seen in Lemma \ref{flasquelem} that 
  for $E \to K\{J\}$ a projective resolution of the $KJ^{\op}$-module $K\{J\}$
  there is a quasi-isomorphism
  \begin{displaymath}
    S/I \cong \Hom_{KJ^{\op}} (K\{J\},R) \xrightarrow \simeq \Hom_{KJ^{\op}} (E,R).
  \end{displaymath}
  Since the Nakayama functor is exact
% constructed from functors of the form
%   $\Hom_S(P,-)$ 
  it commutes
  with $\Hom_{KJ^{\op}}(E,-)$ in the sense that
  \begin{eqnarray*}
%    \Na^{\mbk}_{\mbt}(S/I) \xrightarrow \simeq 
    \Na^{\mbk + \one}_{\mbt}(\Hom_{KJ^{\op}}(E,R)) \cong
    \Hom_{KJ^{\op}}(E,\Na^{\mbk + \one}_{\mbt} \circ R). 
  \end{eqnarray*}
  One way to see this is by noting that
  $\Hom_{KJ^{\op}}(E,R)$ is the 
  kernel of the homomorphism
  \begin{displaymath}
    \prod_{\mbx \in J} \Hom_K(E(\mbx),R(\mbx)) \to
    \prod_{(\mbx \ge \mby)} \Hom_K(E(\mbx),R(\mby)),
  \end{displaymath}
  taking the collection of maps $(f(\mbx))_{\mbx \in J}$ to the collection of maps
$(g(\mbx \ge \mby))_{(\mbx \ge
    \mby)}$, where $g(\mbx \ge \mby) \colon E(\mbx) \to R(\mby)$ is the
  map 
  $$g(\mbx \ge \mby) = R(\mbx \ge \mby) \circ f(\mbx) - f(\mby) \circ
  E(\mbx \ge \mby).$$ 
Since the Nakayama functor commutes with arbitrary direct products, this
shows our claim.
  By Lemma \ref{nathelper1} there is a quasi-isomorphism
  \begin{displaymath}
    \Na^{\mbk + \one}_{\mbt} R(\mbw) = \Na^{\mbk + \one}_{\mbt}
    K\{0,\mbw\} \xrightarrow \simeq
    \mycx_{\mbt}^{\mbk+ \one}(\mbw).
  \end{displaymath}
  Thus we have provided a quasi-isomorphism
  \begin{displaymath}
    \Na_{\mbt}^{\mbk+\one}(S/I) \xrightarrow {\simeq} 
\Hom_{KJ^{\op}}(E, \mycx_{\mbt}^{\mbk+\one}). 
  \end{displaymath}
(Note that $\mycx_{\mbt}^{\mbk+\one}$ is a functor from $J^{\op}$ to the category 
of chain complexes of $S$-modules.) 
Similarly let $J^\prime$ 
be the $\mba \in [\mathbf 0, \mbt + \mbr]$ such that
$(S/I)_{\mba} = K$. With $E^\prime$ the analog of $E$,
there is a quasi-isomorphism
  \begin{displaymath}
    \Na_{\mbt+\mbr}^{\mbk+\mbr+\one}(S/I) \xrightarrow {\simeq} 
\Hom_{KJ^{\prime \op}}(E^\prime, \mycx_{\mbt+\mbr+\one}^{\mbk+\mbr + \one}). 
  \end{displaymath}
  The result now follows from the fact that
  $\mycx_{\mbt+\mbr}^{\mbk+\one +\mbr}= (q_{\mbk}^{\mbr})^*
  \mycx_{\mbt}^{\mbk+\one}$ for $\mathbf 0 \le \mbk \le \mbt$,
  and that the structure maps in the above
  limits correspond to each other under this identification.

In the category of positively $\mbt$-determined $S$-modules, the
indecomposable injectives are the $K_\mbt \{0 , \mba\}$
where $\mba \in [0, \mbt]$. Note that these are of the form $S/I$.
Any module in this category has a finite injective resolution. Since
the Nakayama functor is exact and by Remark \ref{RemDefineNak2} it commutes
with colimits and hence with direct sums, we get our result for any $M$.

 % Given two positively $\mbt$-determined ideals $I$ and $I'$ in $S$ with $I
 % \subseteq I'$ there is a short-exact sequence
 % \begin{displaymath}
 %   0 \to Q \to S/I \to S/I' \to 0,
 % \end{displaymath}
 % where $Q$ is an indecomposable positively $\mbt$-determined $S$-module. By
 % exactness of the Nakayama functor 
 % there also is a chain of
 % quasi--isomorphisms between the chain complexes
 % $\Na_{\mbt+\mbr}^{\mbk+\one +\mbr}(Q)$ and $(q_{\mbk}^{\mbr})^*
 % \Na_{\mbt}^{\mbk+\one}(Q)$. 
%Since any finitely generated module positively $\mbt$-determined $M$ has a filt%ration of 
%modules of this form,
%our result follows for such $M$. If $M$ is any positively $\mbt$-determined
%modules, it is a colimit of fintely generated ones. Since by Remark \ref{RemDef%ineNak2},
%the Nakayama functor commutes with colimits, and colimits is exact on the
%category of chain complexes, we get our result for any $M$.
\end{proof}

The above proposition may be stated in a different way, perhaps making the independence 
of  $\mbt$ more transparent. Recall from Proposition \ref{trans} 
that $\Na_{\mbt}^{\mbt + \mathbf 2}$ 
is naturally quasi-isomorphic to the translation $T^{-2n}$. Since the Nakayama
functors give auto-equivalences on derived categories, the following is natural.
\begin{defn} For a positively $\mbt$-determined $S$-module $M$ and 
$\mathbf 0 \le \mbk \le \mbt + \mathbf 2$ let
\[ \Na_{\mbt}^{-\mbk}(M) = T^{2n} \Na_{\mbt}^{\mbt + \mathbf 2 -\mbk}(M). \]
\end{defn}

\noindent The following is then a reformulation of the previous proposition.

\begin{prop}
  Let $M$ be a finitely generated positively $\mbt$-determined $S$-module and let 
$\mathbf 0 \le \mbk \le \mbt \le \mbt^\prime$ be multidegrees in $\N^n$.
There is a chain of
quasi-isomorphisms between the chain complexes
$\Na_{\mbt^\prime}^{-\mbk-\one}(M)$ and 
$(q_{\mbt - \mbk}^{\mbt^\prime - \mbt})^* 
\Na_{\mbt}^{- \mbk - \one}(M)$.
\end{prop}

In particular for a monomial ideal $I$, the vanishing of a cohomology group
of $\Na_{\mbt}^{- \mbk-\one}(S/I)$ is independent of which $\mbt$ is chosen such
that $I$ is positively $\mbt$-determined, and is thus an intrinsic property of the monomial ideal.

\begin{remark} In particular we would like to draw attention to the case when $\mbk = \one$.
Then the vanishing of a cohomology group of $\Na_{\mbt}^{-2}(S/I)$ is independent
of $\mbt$ for any $\mbt \geq \one$ such that $S/I$ is positively $\mbt$-determined. In the case
when $S/I$ is square free, i.e. $\one$-determined (it is then a Stanly-Reisner ring), the condition that 
$\Na_{\mbt}^{-2}(S/I)$ has only one non-vanishing cohomology group is equivalent
to $S/I$ being Cohen-Macaulay. Thus this condition may provide a good(?) generalization to
positively $\mbt$-determined monomial quotient rings, of 
the concept of a Stanley-Reisner ring being Cohen-Macaulay.
\end{remark}

%\begin{cor}
%  For a finitely generated positively $\mbt$-determined module $M$ the vanishing
%  of a cohomology module of $\Na_{\mbt}^{-\one - \mbk}(M)$ is
%  independent of $\mbt$.
%\end{cor}
%\begin{remark}
%  In view of Proposition \ref{trans} the conclusion of Proposition
%  \ref{realstability} can be formulated as follows:
%  there is a chain of quasi-isomorphisms between
%  $\Na_{\mbt+\mbr}^{-\mbk}(M)$ and $(q_{\mbk}^{\mbr})^*
%  \Na_{\mbt}^{-\mbk}(M)$, and thus
%  the vanishing
%  of a cohomology module of $\Na_{\mbt+\mbr}^{-\mbk}(M)$ is
%  independent of $\mbt$.
%\end{remark}

\subsection{Linearity}
\label{sec:lini}

%Let $C$ be a chain complex of $\Z^n$-graded $S$-module and let $F$ be a projective
%$\Z^n$-graded resolution of $S/\mm$ over $S$. The $i$'th Betti space
%of $C$ is the $\Z^n$-graded $K$-vector space $B_i(C) = H_i(F \otimes_S C)$.

The celebrated result of Eagon and Reiner stating that a simplicial complex
is Cohen-Macaulay if and only if the square free monomial ideal
of its Alexander dual complex has linear resolution, was generalized by Miller
\cite[Thm.4.20]{Miller2000}.
Let us recall his result. For a degree $\mba$ in $\N^n$ let the support of $\mba$,
$\text{supp}(\mba)$,
be the subset of $\{1,\ldots, n\}$ consisting of those $i$ for which $a_i > 0$.
A $\N^n$-graded module $M$ is said to have {\it support-linear} 
resolution if there is an integer $d$ such that degree of every minimal generator
of $M$ has support of cardinality $d$, and $d \geq |\text{supp}(\mbb)| - i$
for every multidegree $\mbb$ and homological index $i$ for which the Betti space
$B_{i, \mbb}$ is nonzero. Miller shows the following.

\begin{thm} Let $M$ be a positively $\mbt$-determined module. The Alexander dual 
$A_{\mbt}(M)$ has
support-linear resolution if and only if $M$ is Cohen-Macaulay, i.e. 
$\Na^{\one}_{\mbt}(M)$ has only one non-vanishing cohomology module.
\end{thm}

We do not know of any other relationship between the vanishing of cohomology of
some $\Na^{\mbk}_{\mbt}(M)$, and linearity conditions of some other 
$\Na^{\mbk^\prime}_{\mbt}(M)$ or $\Na^{\mbk^\prime}_{\mbt}A_{\mbt}(M)$.
But there are notions of linearity of complexes which in the situation of 
$\Na^{- \mbk - \one}_{\mbt}(M)$ are independent of $\mbt$.

\begin{defn}
  Let $C$ be a chain complex of $\Z^n$-graded $S$-module and let $\mbc \in \N^n$. We
  say that that $C$ is {\em $\mbc$-linear} if there exists a number
  $p_0$ such that $B_p(C)_{\mbb} \ne 0$ implies that
  the cardinality of $\{j \, | \, b_j \geq c_j \}$ is 
\begin{displaymath}
  \begin{cases}
    n & \text{if $p \geq p_0 + n$} \\
    p-p_0 & \text{if $p_0 \leq p \leq p_0 + n$} \\
    0 & \text{if $p \leq p_0$}.
  \end{cases}
\end{displaymath}
\end{defn}

Note that Millers notion of support-linear resolution is not a special case of $\mbc$-linearity.
%\noindent The result of Miller \cite[Thm.XX]{Mi} then states.

%\begin{thm} $\Na^{-\one}_{\mbt}(M)$ is $\mbt$-linear if and only if $M$ is Cohen-Macaulay,
%or equivalently $\Na^{\one}_{\mbt}(M)$ has only one non-zero cohomology module.
%\end{thm}

%Unfortunately we have not been able to see any other relationship between
%the linearity of $\Na^{\mbk}_{\mbt}(M)$ and the vanishing of cohomology of some
%other $\Na^{\mbk^\prime}_{\mbt}(M)$.
%However the following give linearity coniditons on the Nakayama functors of 
%$M$ which are inedependent of $\mbt$.

\begin{prop} \label{IndPropLin} Let $\mathbf 0 \le \mbk \le \mbt$. Suppose 
$\Na^{ \mbk + \one}_{\mbt}(M)$ is $\mathbf c$-linear.

a. If $\mbc \le \mbk $, then $\Na^{ \mbk + \one + \mathbf r}_{\mbt + \mathbf r}(M)$
is $\mbc$-linear for any $\mbr$ in $\N^n$.

b. If $\mbc \ge \mbk $, then  $\Na^{\mbk + \one + \mathbf r}_{\mbt + \mathbf r}(M)$
is $\mathbf {c+r}$-linear for any $\mbr$ in $\N^n$.
\end{prop}

\begin{proof}
By Proposition  \ref{realstability} the Betti spaces
\[ B_p(\Na^{\mbk + \one + \mathbf r}_{\mbt + \mathbf r}(M))_{\mathbf d} =
B_p(\Na^{\mbk + \one}_{\mbt}(M))_{q^\mbr_{\mbk}(\mathbf d)}. \]
Then a. follows because coordinate $i$ in $\mathbf d$ is $< c_i$ if and only if the 
same holds in $q^\mbr_{\mbk}(\mathbf d)$.
Part b. follows because coordinate $i$ in $\mathbf d$ is $\geq c_i + r_i$ if and only if 
coordinate $i$ in $q^\mbr_{\mbk}(\mathbf d)$ is $\geq c_i$.
\end{proof}

Alternatively this may be stated as.
\begin{prop} \label{IndPropLin2}Let $\mathbf 0 \le \mbk \le \mbt \le \mbt^\prime$. Suppose 
$\Na^{- \mbk- \one}_{\mbt}(M)$ is $\mbt \!- \!\mbc$-linear.

a. If $\mbc \ge \mbk$, then $\Na^{- \mbk - \one}_{\mbt^\prime}(M)$
is $\mbt \!-\! \mbc$-linear.

b. If $\mbc \le \mbk$, then $\Na^{- \mbk - \one}_{\mbt^\prime}(M)$
is $\mbt^\prime \!-\! \mbc $-linear.
\end{prop}

In the pivotal case of Proposition \ref{IndPropLin2}, when $\mbc = \mbk$, 
and the polynomial ring
$S$ is in two variables, we have worked out explicitly what it means for 
$\Na^{- \mbk - \one}(S/I)$ to be $\mbt - \mbk$-linear. The reader may wish to compare this with the results
in Section \ref{sec:van} concerning the vanishing of cohomology modules of these complexes, when we have
two variables. We state the following without proof.

\begin{prop} Suppose $S$ is the polynomial ring in two variables, and let  
$\one \leq \mbk \leq \mbt$. Then $\Na^{- \mbk - \one}_{\mbt} (S/I)$ is 
$\mbt - \mbk$-linear if and only if either

a. $I$ is $\mbk$-determined.

b. $I$ is $(x^a y^b)$ where $(a,b) \geq \mbk$.

c. $I$ is $(x^{a_0}, x^{a_1}y^{b_1}, \ldots, x^{a_{p-1}}y^{b_{p-1}}, y^{b_p})$ where all
$t_1 \geq a_i \geq k_1$ and $t_2 \geq b_i \geq k_2$.

\end{prop}

\section{Proofs of the main theorems}
\label{sec:nakmon}
In this section we prove Theorem \ref{thecalc1}, Theorem \ref{bettithm} and Theorem
\ref{nakcohS}. Originally our approach was to use a homotopy limit spectral sequence argument,
see for instance \cite{Ziegler1999}.
However the spectral sequences degenerated so nicely that we are able to 
do without them.
Let us fix $\mbk,\mbt \in \N^n$ with $\mbk \le \mbt +
\one$.
\begin{proof}[Proof of Theorem \ref{thecalc1}]
  Let $J$ denote the
  order ideal in $\N^n$ consisting of those $\mba \in [0,\mbt]$ with
  $(S/I)_{\mba} =
  K$. 
% In different terms $J$ is the 
%   poset 
%   $$J = \{\mba \in \N^n \colon \text{$I_{\mba} = 0$
%   and $\mba \le \mbt$}\}.$$
  Let $R$ denote the functor from $J^{\op}$ to the category of
  positively $\mbt$-determined $S$-modules with 
  $R(\mbw) = K_{\mbt}\{0,\mbw\}$. 
  Clearly, as an $S$-module
  \begin{displaymath}
    S/I \cong \lim_{\mbw \in J^{\op}} R(\mbw) = \Hom_{KJ^{\op}} (K\{J^{\op}\},R).
  \end{displaymath}
  We have seen in Lemma \ref{flasquelem} that 
  for $E \to K\{J^{\op}\}$ a projective resolution of the $KJ^{\op}$-module $K\{J^{\op}\}$
  there is a quasi-isomorphism
  \begin{displaymath}
    S/I \cong \Hom_{KJ^{\op}} (K\{J^{\op}\},R) \xrightarrow \simeq \Hom_{KJ^{\op}} (E,R).
  \end{displaymath}
  Since the Nakayama functor is exact
% constructed from functors of the form
%   $\Hom_S(P,-)$ 
  it commutes
  with $\Hom_{KJ^{\op}}(E,-)$ in the sense that
  \begin{eqnarray*}
%    \Na^{\mbk}_{\mbt}(S/I) \xrightarrow \simeq 
    \Na^{\mbk}_{\mbt}(\Hom_{KJ^{\op}}(E,R)) \cong
    \Hom_{KJ^{\op}}(E,\Na^{\mbk}_{\mbt} \circ R), 
  \end{eqnarray*}
as we explained in the first part of the proof of 
Proposition \ref{realstability}.
%  One way to see is is by noting that $J$ is finite and that
%  $\Hom_{KJ^{\op}}(E,R)$ is the 
%  kernel of the homomorphism
%  \begin{displaymath}
%    \prod_{\mbx \in J} \Hom_K(E(\mbx),R(\mbx)) \to
%    \prod_{(\mbx \ge \mby)} \Hom_K(E(\mbx),R(\mby)),
%  \end{displaymath}
%  taking $(f(\mbx))_{\mbx \in J}$ to $(g(\mbx \ge \mby))_{(\mbx \ge
%    \mby)}$, where $g(\mbx \ge \mby) \colon E(\mbx) \to R(\mby)$ is the
%  map 
%  $$g(\mbx \ge \mby) = R(\mbx \ge \mby) \circ f(\mbx) - f(\mby) \circ
%  E(\mbx \ge \mby).$$ 
%   Since the Nakayama functor is exact 
%   it commutes
%   with $\Hom_{KJ^{\op}}$ in the sense that
%   \begin{eqnarray*}
% %    \Na^{\mbk}_{\mbt}(S/I) \xrightarrow \simeq 
%     \Na^{\mbk}_{\mbt}(\Hom_{KJ^{\op}}(E,R)) \cong
%     \Hom_{KJ^{\op}}(E,\Na^{\mbk}_{\mbt} \circ R). 
%   \end{eqnarray*}
  For $\mbw$ in $J$ there is by Lemma \ref{nathelper1} a quasi-isomorphism
  \begin{displaymath}
    (\Na^{\mbk}_{\mbt} R(\mbw))_{\mbr} = (\Na^{\mbk}_{\mbt}
    K_{\mbt}\{0,\mbw\})_{\mbr} \xrightarrow \simeq
    (\mycx_{\mbt}^{\mbk}(\mbw))_{\mbr}.
  \end{displaymath}
  Moreover by Lemma \ref{nathelper2} there is a quasi-isomorphism
  \begin{displaymath}
    (\mycx_{\mbt}^{\mbk}(\mbw))_{\mbr} \xrightarrow \simeq 
    (\mycx_{\mbt}^{\mbk}(\mbt-\mbr))_{\mbt-\mbw},
  \end{displaymath}
  and this quasi-isomorphism is natural in $\mbw$, by Corollary \ref{natswapuy}. 
  Let $\gamma = \gamma_{\mbt}^{\mbk}(\mbr)$, let $\mbu = \mbu_{\mbt}^{\mbk}(\mbr)$
  and let $\mbv = \mbv_{\mbt}^{\mbk}(\mbr)$.
 % By Lemma \ref{nathelper2} and 
By Proposition  \ref{nakcoh} and Remark \ref{RemMainPro1}, there is a quasi-isomorphism
  \begin{displaymath}
    (\mycx_{\mbt}^{\mbk}(\mbt-\mbr))_{\mbt-\mbw} \xrightarrow \simeq
    T^{-\gamma} K_{\mbt}\{\mbt-\mbv,\mbt-\mbu\}_{\mbt-\mbw}
    =
    T^{-\gamma} K\{[\mbu,\mbv]^{\op}\}({\mbw}),
  \end{displaymath}
where $K\{[\mbu,\mbv]^{\op}\}$ is a module as in 
Definition \ref{DefKP}. 
  Since $\mbw \in J^{\op}$ the latter is 
  $ T^{-\gamma} K\{(J \cap [\mbu,\mbv])^{\op}\}({\mbw})$.   
Summing up we have a quasi-isomorphism
  \begin{displaymath}
    (\Na^{\mbk}_{\mbt} R(\mbw))_{\mbr} \xrightarrow \simeq T^{-\gamma}
    K\{(J \cap [\mbu,\mbv])^{\op}\}(\mbw).
  \end{displaymath}
  which is natural in $\mbw$ in the sense that if $\mbw \le \mbw'$
  then the corresponding diagram commutes.
%   Applying Corollary \ref{bignathelper} and Theorem \ref{nakonbiginterval} we get the 
%   quasi-isomorphisms
%   \begin{eqnarray*}
%     \Na^{\mbk}_{\mbt} M(\mbw)_{\mbr} &=& \Na^{\mbk}_{\mbt} K\{0,\mbw\}_{\mbr} \\ 
%     &\simeq &\Na^{k}_l
%     K\{0,l-r\}_{l -w} \\
%     &\simeq &
%     T^{-\gamma(l-r)} K\{a(l-r),b(l-r)\}_{l -w} \\
%     &\cong &
%     T^{-\gamma(l-r)} K\{l-b(l-r),l-a(l-r)\}_{w} \\
%     &\cong &
%     T^{-\gamma(l-r)} K\{u(r),u(r)+v(r)- \one\}_{w},
%   \end{eqnarray*}
%   and these equivalences are natural in $w \in J^{\op}$.
  Thus we have a quasi-isomorphism
  \begin{displaymath}
    \Na^{\mbk}_{\mbt} (S/I))_{\mbr} \xrightarrow \simeq \Hom_{KJ^{\op}} (E,
    T^{-\gamma}
    K\{(J \cap [\mbu,\mbv])^{\op}\}).
  \end{displaymath}
  Finally applying Proposition \ref{modscx} we obtain isomorphisms
  \begin{displaymath}
    H^i \Na^{\mbk}_{\mbt}(S/I)_{\mbr} \cong H^{i-\gamma} (\Hom_{KJ^{\op}} (E,
    K\{(J \cap [\mbu,\mbv])^{\op}\})) \cong \widetilde 
H^{i-\gamma-1}(\Delta^{\mbv}_{\mbu}(J)).
  \end{displaymath}
%   \begin{eqnarray*}
%     T(\Na^{k}_l R)_{r} &\simeq& T\Hom_{J^{\op}} (E,
%     T^{-\gamma(l-r)}
%     K\{u(r),u(r)+v(r)-\one\}) \\
%     &\simeq&
%     T^{-\gamma(l-r)}\widetilde C^{-*}(
%     Z(J,u(r),v(r)) \\
%     &=&
%     T^{-\gamma(l-r)}\widetilde C^{-*}(
%     \Delta(R,r)).
%   \end{eqnarray*}
\end{proof}
\begin{proof}[Proof of Theorem \ref{bettithm}]
  The functor $F \otimes_S \Na_{\mbt}^{\mbk}$ is exact, and therefore
  it commutes with $\Hom_{KJ^{\op}}(E,-)$.
  By Corollary \ref{firststepbetti} there is a quasi-isomorphism
  \begin{displaymath}
    (F \otimes_S
    \Na^{\mbk}_{\mbt} R(\mbw))_{\mbr} =
    (F \otimes_S
    \Na^{\mbk}_{\mbt} K_{t}\{ 0,\mbw\})_{\mbr}
    \simeq 
    T^n \Na^{\mbk+\one}_{\mbt} (K_{\mbt}\{\mbt-\mbr,\mbt-\mbr\})_{\mbt - \mbw},
  \end{displaymath}
  and this quasi-isomorphism is natural in $\mbw$.
  Let $\gamma$, $\mba$ and $\mbb$ be such that there is a quasi-isomorphism
  \begin{displaymath}
    \Na^{\mbk+\one}_{\mbt} (K_{\mbt}\{\mbt-\mbr,\mbt-\mbr\})
    \simeq
    T^{-\gamma - n} K_{\mbt}\{\mbt-\mbb,\mbt-\mba\}.
  \end{displaymath}
  Note that since $\mbw \in J$ we have
%there is an isomorphism
  \begin{displaymath}
    K_{\mbt}\{\mbt-\mbb,\mbt-\mba\}_{\mbt-\mbw} =
   K\{(J \cap [\mba,\mbb])^{\op}\}({\mbw}).
  \end{displaymath}
  Summing up we have a quasi-isomorphism
  \begin{displaymath}
    (F \otimes_S
    \Na^{\mbk}_{\mbt} R(\mbw))_{\mbr} 
%    \Na^{\mbk+\one}_{\mbt} (K_{\mbt}\{\mbt-\mbu,\mbt-\mbu\})_{\mbw}    
    \simeq T^{-\gamma}
    K\{(J \cap [\mba,\mbb])^{\op}\}(\mbw),
  \end{displaymath}
  which is natural in $\mbw$ in the sense that if $\mbw \le \mbw'$
  then the corresponding diagram commutes.
%   Applying Corollary \ref{bignathelper} and Theorem \ref{nakonbiginterval} we get the 
%   quasi-isomorphisms
%   \begin{eqnarray*}
%     \Na^{\mbk}_{\mbt} M(\mbw)_{\mbr} &=& \Na^{\mbk}_{\mbt} K\{0,\mbw\}_{\mbr} \\ 
%     &\simeq &\Na^{k}_l
%     K\{0,l-r\}_{l -w} \\
%     &\simeq &
%     T^{-\gamma(l-r)} K\{a(l-r),b(l-r)\}_{l -w} \\
%     &\cong &
%     T^{-\gamma(l-r)} K\{l-b(l-r),l-a(l-r)\}_{w} \\
%     &\cong &
%     T^{-\gamma(l-r)} K\{u(r),u(r)+v(r)- \one\}_{w},
%   \end{eqnarray*}
%   and these equivalences are natural in $w \in J^{\op}$.
  Thus we have a quasi-isomorphism
  \begin{displaymath}
    F \otimes_S \Na^{\mbk}_{\mbt} (S/I))_{\mbr} \xrightarrow \simeq \Hom_{KJ^{\op}} (E,
    T^{-\gamma} K\{(J \cap [\mba,\mbb])^{\op}\}
    ).
  \end{displaymath}
  Finally applying Proposition \ref{modscx} we obtain isomorphisms
  \begin{displaymath}
    H^i (F \otimes_S \Na^{\mbk}_{\mbt}(S/I))_{\mbr} \cong H^{i-\gamma} (\Hom_{KJ^{\op}} (E,
    K\{(J \cap [\mba,\mbb])^{\op}\})) \cong \widetilde 
H^{i-\gamma-1}(\Delta^{\mbb}_{\mba}(J)).
  \end{displaymath}
%   \begin{eqnarray*}
%     T(\Na^{k}_l R)_{r} &\simeq& T\Hom_{J^{\op}} (E,
%     T^{-\gamma(l-r)}
%     K\{u(r),u(r)+v(r)-\one\}) \\
%     &\simeq&
%     T^{-\gamma(l-r)}\widetilde C^{-*}(
%     Z(J,u(r),v(r)) \\
%     &=&
%     T^{-\gamma(l-r)}\widetilde C^{-*}(
%     \Delta(R,r)).
%   \end{eqnarray*}
\end{proof}

\begin{proof}[Proof of Theorem \ref{nakcohS}]
We shall describe 
the homomorphism
$$H^i(\Na^{\mbk}_{\mbt} S/I)_{\mbr-\varepsilon_j} \to H^i(\Na^{\mbk}_{\mbt}
S/I)_{\mbr}$$
for every $i$ and every $\mbr$ with $\varepsilon_j \le \mbr \le \mbt$. 
In order to describe this homomorphism we need to take a closer look at
the maps in the proof of Theorem \ref{thecalc1}. Let us continue with
the notation introduced in that proof.
% at the equivalences
%   \begin{eqnarray*}
%     \Na^{k}_l M(w)_{r} &=& \Na^{k} K\{0,w\}_{r} \\ 
%     &\simeq &\Na^{k}
%     K\{0,l-r\}_{l -w} \\
%     &\simeq &
%     T^{-\gamma(l-r)} K\{a(l-r),b(l-r)\}_{l -w} \\
%     &\cong &
%     T^{-\gamma(l-r)} K\{u(r),u(r)+v(r)-\one\}_{u}.
%   \end{eqnarray*}
% Let $\psi \colon [0,l]^{\op} \to [0,l]$ denote the isomorphism of
% posets given
% by $\psi(w) = l-w$ so that
% $\psi^* K\{a(l-r),b(l-r)\} = K\{l-b(l-r),l-a(l-r)\} = K\{u(r),u(r)+v(r) - \one\}$. 
We split up into the three cases where $k_j < r_j$, where $k_j =
r_j$ and where $k_j > r_j$.

Firstly if $k_j < r_j$, 
the natural homomorphism $\Na^{\mbk}_{\mbt} R(\mbw)_{\mbr
  - \varepsilon_j} \to \Na^{\mbk}_{\mbt}
R(\mbw)_{\mbr}$ of functors in $\mbw \in J^{\op}$ corresponds to
the homomorphism of $KJ^{\op}$-modules
\begin{displaymath}
  T^{-\gamma} K\{(J \cap [\mbu - \varepsilon_j,\mbv])^{\op}\} \to T^{-\gamma} 
  K\{(J \cap[\mbu,\mbv])^{\op}\}.
\end{displaymath}
Applying $T\Hom_{KJ^{\op}}(E,-)$ to this homomorphism we obtain 
the
homomorphism 
\begin{displaymath}
  (\Na^{\mbk}_{\mbt} (S/I))_{\mbr-\varepsilon_j} \to (\Na^{\mbk}_{\mbt} (S/I))_{\mbr}.
\end{displaymath}
The second naturality statement of Proposition \ref{natrem2} 
gives that the associated homomorphism of $i$'th cohomology groups
is isomorphic to the homomorphism of $i-\gamma-1$'th reduced cohomology 
groups coming from the inclusion of spaces 
$ \Delta_{\mbu}^{\mbv} \subseteq \Delta_{\mbu -\varepsilon_j}^{\mbv}$.

Next we consider the case $k_j = r_j$.
Consider the map 
\[(\Na^{\mbk}_{\mbt} \circ R)(\mbw)_{\mbr 
  - \varepsilon_j} \to (\Na^{\mbk}_{\mbt} \circ 
R)(\mbw)_{\mbr}.\] 
As explained in the proof of Theorem \ref{thecalc1}, it corresponds
to the natural map
\[ N_\mbt^{\mbk}(\mbr - \epsilon_j, \mbr)_{\mbt - \mbw} \, \, :
\, \, N_{\mbt}^\mbk(\mbt - \mbr + \epsilon_j)_{\mbt - \mbw} \rightarrow 
N^{\mbk}_{\mbt}(\mbt - \mbr)_{\mbt - \mbw}. \]
The right side of the map above is 
$T^{-\gamma} K\{(J \cap [\mbu, \mbv]^{\op} \}(\mbw)$, the identification
being natural in $\mbw$. 
For slim notation let
\begin{eqnarray*}
\alpha = u_{t_j}^{k_j}(r_j-1)- u_{t_j}^{k_j}(r_j) = t_j + 1 - k_j, 
\quad \beta = v_{t_j}^{k_j}(r_j-1)- v_{t_j}^{k_j}(r_j) = k_j 
\end{eqnarray*}
By Subsection 6.3 the above map, considered as depending functorially on
$\mbw$, is 
$(-1)^{\sum_{i=1}^{j-1} \gamma_{t_i}^{k_i}(r_i)}$
times the connecting homomorphism
of the first map $\hat{i}$ in the short exact sequence
\begin{eqnarray*}
0 & \rightarrow &  T^{-\gamma-1} K\{(J \cap[\mbu,\mbv])^{\op}\} \mto{\hat{i}}
    T^{-\gamma-1} K\{(J \cap[\mbu,\mbv+\beta\varepsilon_j])^{\op}\} \\ 
&\rightarrow & 
    T^{-\gamma-1} K\{(J \cap[\mbu + \alpha \varepsilon_j ,\mbv+
    \beta\varepsilon_j])^{\op}\} \rightarrow 0. 
\end{eqnarray*}
%Note also that the mapping cone of $\hat{i}$ is quasi-isomporphic to
%the third term in the sequence above. 

We now apply $T\Hom_{KJ^{\op}}(E,-)$
to this short exact sequence, and obtain a new short exact sequence
of complexes. Note the following three things.

\medskip
\noindent 1. The new short exact sequence is by Proposition \ref{natrem2} 
quasi-isomorphic to
the (non-exact) sequence of homologically shifted reduced cochain
complexes associated to the inclusions
\begin{displaymath}
  \Delta^{\mbv + \beta \varepsilon_j}_{\mbu + \alpha
    \varepsilon_j} \subseteq \Delta_{\mbu}^{\mbv + \beta
    \varepsilon_j} \subseteq \Delta_{\mbu}^{\mbv}.
\end{displaymath}
The simplicial complex $\Delta^{\mbv }_{\mbu + \alpha
  \varepsilon_j}$ is contractible because it is a cone on the vertex
$j$, that is, it contains $F$ if and only it contains $F \cup \{j\}$.
Since we are in the situation of Lemma \ref{pbsituation}
the above short exact sequence is quasi-isomorphic to the short exact
sequence
\begin{displaymath}
  \widetilde C^* (\Delta_{\mbu}^{\mbv})
\to
\widetilde C^* (\Delta_{\mbu}^{\mbv + \beta
    \varepsilon_j}) \oplus \widetilde C^*(\Delta^{\mbv}_{\mbu + \alpha
  \varepsilon_j}) \to
  \widetilde C^*(\Delta^{\mbv + \beta \varepsilon_j}_{\mbu + \alpha
    \varepsilon_j})  
\end{displaymath}
of reduced cochain complexes underlying the Mayer--Vietoris cohomology
exact sequence.  

\noindent 2. The connection homomorphism for this short exact sequence
\begin{displaymath}
  \delta^{i-\gamma-2} \colon 
   \widetilde H^{i-\gamma-2}(\Delta^{\mbv + \beta \varepsilon_j}_{\mbu + \alpha
    \varepsilon_j})
\to \widetilde H^{i-\gamma-1}(\Delta_{\mbu}^{\mbv}).
\end{displaymath}
identifies as the morphism on homology we get by applying 
$T\Hom_{KJ^{\op}}(E,-)$ to the connecting homomorphism of $\hat{i}$. 

\noindent 3. If we apply $T\Hom_{KJ^{\op}}(E,-)$ 
to the morphism of $KJ^{\op}$-modules 
\[(\Na^{\mbk}_{\mbt} \circ R)(\cdot)_{\mbr 
  - \varepsilon_j} \to (\Na^{\mbk}_{\mbt} \circ 
R)(\cdot)_{\mbr}\] 
and take homology this gives the map 
\begin{displaymath}
  H^i (\Na^{\mbk}_\mbt S/I)_{\mbr-\varepsilon_j} \to H^i (\Na^{\mbk}_\mbt S/I)_{\mbr}. 
\end{displaymath}

Hence this map is isomorphic to
$(-1)^{\sum_{i=1}^{j-1} 
  \gamma_{t_i}^{k_i}(r_i)}$ times
the Mayer-Vietoris connecting homomorphism.

\medskip
Finally we consider the case $k_j > r_j$.
In this case the natural homomorphism $\Na^{\mbk}_\mbt R(\mbw)_{\mbr
  - \varepsilon_j} \to \Na^{\mbk}_{\mbt} 
R(\mbw)_{\mbr}$ of functors in $\mbw \in J^{\op}$ corresponds to
the homomorphism
\begin{displaymath}
  T^{-\gamma} K\{(J \cap [\mbu,\mbv-\varepsilon_j])^{\op}\} \to T^{-\gamma} 
  K\{(J \cap [\mbu,\mbv])^{\op}\}.
\end{displaymath}
Applying $T\Hom_{KJ^{\op}}(E,-)$ to this homomorphism we obtain the
homomorphism 
\begin{displaymath}
  (\Na^{\mbk}_\mbt S/I)_{\mbr-\varepsilon_j} \to (\Na^{\mbk}_\mbt S/I)_{\mbr}.
\end{displaymath}
By the first naturality part of Proposition \ref{natrem2} this is
quasi-isomorphic to the
homomorphism of reduced cochain complexes induced by the inclusion
$\Delta_{\mbu}^{\mbv} \subseteq \Delta_{\mbu}^{\mbv-\varepsilon_j}$.
% \begin{displaymath}
%   \Delta(R,r) = Z(J,x,y) \subseteq Z(J,x,y-\varepsilon_j) = \Delta(R,r-\varepsilon_j). 
% \end{displaymath}
\end{proof}


\begin{thebibliography} {99}



 \bibitem{ARS} M.Auslander, I.Reiten, S.Smal\o . {\it Representations of 
 artin algebras} Cambridge studies in advanced mathematics.


%\bibitem{BBR} Brun, Bruns, R\"omer.

%\bibitem{BCP} D.Bayer, H.Charalambous, S.Popescu, 
%{\em Extremal Betti numbers and applications to monomial ideals,}  J. Algebra  221  (1999),  no. 2, 497--512. 
%\bibitem{Br} A.E.Brouwer, in 
% R.L. Graham, M. Gr\"otschel, L. Lovasz (eds.)
%{\it Handbook of combinatorics} North-Holland, Elsevier Science B.V.
%(1995).



\bibitem{BH} W.Bruns and J.Herzog, {\it Cohen-Macaulay rings} Cambridge
studies in advanced mathematics 39, Cambridge University Press 1993.


% \bibitem{Ei} D.Eisenbud, {\it Commutative algebra with a view towards
% algebraic geometry}, GTM 150, Springer-Verlag, 1995. 

 \bibitem{ER}
 J.A.Eagon and V.Reiner, {\it Resolutions of Stanley-Reisner rings and
   Alexander duality}, Journal of Pure and Applied Algebra 130 (1998)
 p.265-275. 



\bibitem{Floystad2004} G.Fl\o ystad, {\it Enriched homology and cohomology modules
of simplicial complexes}, Journal of Algebraic Combinatorics 25 (2007),
p.285-307.


%\bibitem{FV} G.Fl\o ystad and J.E.Vatne, {\it (Bi)-Cohen-Macaulay simplicial
%complexes and their associated coherent sheaves}, to appear in 
%Communications in Algebra.


%\bibitem{Fr} R. Fr\"oberg {\it Rings with monomial relations having 
%linear resolutions} Journal of Pure and Applied Algebra 38 (1998) p.235-241.


\bibitem{Graebe1984} H.-G. Gr\"abe, 
{\em The canonical module of a Stanley-Reisner ring.}
J. Algebra 86 (1984), no. 1, 272--281.


%\bibitem{HT} T.Hibi and N.Terai, {\it Finite free resolutions and 
%$1$-skeletons
%of simplicial complexes}, J. Algebraic Combinatorics {\bf 6} (1997), no.1,
%p.89-93. 

\bibitem{Ha} D. Happel, {\em Triangulated categories in the
    representation theory of finite-dimensional algebras}, 
London Mathematical Society Lecture Note Series, 119. 

%\bibitem{Ha2} D. Happel, Comment math helv 1987.
%(see Peter Joergensen)

\bibitem{Hochster1977} M. Hochster, {\it Cohen-Macaulay rings, combinatorics, and
simplicial complexes}, Ring Theory II (Proc. Second Oklahoma Conference)
(B.R. McDonald and R. Morris ed.), Dekker, New York, 1977, p. 171-223.

\bibitem{Miller2000} E.Miller, {\it Alexander duality functors and local duality
with monomial support}, Journal of Algebra {\bf 231} (2000), p.180-234.

\bibitem{MiSt} E.Miller, B. Sturmfels {\it Combinatorial Commutative Algebra},
GTM 227, Springer-Verlag 2005.

%\bibitem{Mustata2000} M. Musta\c t\v a, {\em Local Cohomology at Monomial Ideals}, 
%Symbolic computation in algebra, analysis, and geometry 
%(Berkeley, CA, 1998).  J. Symbolic Comput.  29  (2000),  no. 4-5, 709--720. 

\bibitem{Romer2001} T.R\"omer, {\it Generalized Alexander duality and applications},
 Osaka J. Math. {\bf 38} (2001), no.2, p.469-485.

\bibitem{Takayama2004} Y. Takayama. {\em A generalized Hochster's formula for 
local cohomologies of monomial ideals}, preprint, http://xxx.lanl.gov/math.AC/0411649. 

%\bibitem{Miller2000} E.Miller, {\it Alexander duality functors and local duality
%with monomial support}, Journal of Algebra {\bf 231} (2000), p.180-234.

\bibitem{Ziegler1999} V.Welker, G.Ziegler, R.Zivaljevic, {\it Homotopy colimits-
comparison lemmas for combinatorial applications}, Journal f\"ur die reine und
angewandte Mathematik {\bf 509}, (1999), 117-149.

\bibitem{We} C.Weibel, {\it An introduction to homological algebra}, 
Cambridge University Press, 1994.

 \bibitem{Yanagawa2000} K.Yanagawa, {\it Alexander duality for Stanley-Reisner rings
 and squarefree ${\bf N}^n$-graded modules}, Journal of Algebra {\bf 225},
 (2000) p. 630-645.

\bibitem{Yanagawa2004} K.Yanagawa, {\it
Derived category of squarefree modules and local cohomology with monomial ideal support},  
J. Math. Soc. Japan  56  (2004),  no. 1, 289--308. 

% \bibitem{Mu} J.Munkres, {\it Topological results in combinatorics},
% Michigan Math. J. {\bf 31} (1984), p.113-128.


%\bibitem{Ro} T.R\"omer, {\it Generalized Alexander duality and applications},
% Osaka J. Math. {\bf 38} (2001), no.2, p.469-485.



% \bibitem{St} R.Stanley, {\it Combinatorics and Commutative Algebra},
% Second Edition, Birkh\"auser 1996.

% \bibitem{StE} R.Stanley, {\it Enumerative combinatorics I}, Cambridge Studies
% in Advanced Mathematics 49, Cambridge University Press 1997. 


% \bibitem{Ya} K.Yanagawa, {\it Alexander duality for Stanley-Reisner rings
% and squarefree ${\bf N}^n$-graded modules}, Journal of Algebra {\bf 225},
% (2000) p. 630-645.

%\bibitem{Ya2} K.Yanagawa, {\it Derived category of squarefree modules
%and local cohomology with monomial ideal support}, 
%http://xxx.lanl.gov/math.AC/0402406.




\end{thebibliography}
\end{document}